\documentclass[11pt, noinfoline, preprint, a4paper, twoside]{imsart}

\usepackage[utf8]{inputenc}
\usepackage[T1]{fontenc}
\usepackage[english]{babel}
\usepackage{dsfont}
\usepackage{hyperref}
\usepackage{color}
\usepackage{graphics}
\usepackage{epstopdf}
\usepackage{amsmath,amsthm,amssymb}
\usepackage{enumitem}
\usepackage{enumerate}
\usepackage{tikz,pgfplots}
\usepackage{a4wide}
\usepackage{lineno} %line numbers
\usepackage{algpseudocode,algorithm,algorithmicx}

\newtheorem{theorem}{Theorem}%[section]
\newtheorem{lemma}[theorem]{Lemma}
\newtheorem{corollary}[theorem]{Corollary}
\newtheorem{prop}[theorem]{Proposition}
\newtheorem{definition}{Definition}

\newtheorem{remark}{Remark}
\newtheorem{example}{Example}

\def\beq{\begin{equation}}
\def\eq{\begin{equation}}
\def\eeq{\end{equation}}
\def\qe{\end{equation}}
\def\beqn{\begin{eqnarray*}}
\def\eeqn{\end{eqnarray*}}
\def\bitem{\begin{itemize}}
\def\eitem{\end{itemize}}
\def\benum{\begin{enumerate}}
\def\eenum{\end{enumerate}}
\def\bmult{\begin{multline*}}
\def\emult{\end{multline*}}
\def\bcenter{\begin{center}}
\def\ecenter{\end{center}}
\def\balign{\begin{align}}
\def\ealign{\end{align}}

\def\cC{\mathcal{C}}

\def\cF{\mathcal{F}}
\def\cH{\mathcal{H}}
\def\cL{\mathcal{L}}
\def\cM{\mathcal{M}}
\def\bI{\boldsymbol{I}}

\def\bbE{\mathds{E}}
\def\bbF{\mathds{F}}
\def\bbN{\mathds{N}}
\def\bbZ{\mathds{Z}}

\newcommand{\tauc}{\tau}   % cut-off Example
\def\Pm{\mathcal{P}_m}       % Dual certificate
\def\Qmk{\mathcal{Q}_m^k}       % Dual certificate at k
\def\DP{D_{\Pm}} % Divergence
\newcommand{\dd}{\mathrm{d}} % pour integration
\newcommand{\sinc}{\text{sinc}} % fonction sinc
\newcommand{\1}{\mathds{1}} % indicatrice
\def\iid{\stackrel{\rm iid}{\sim}} % i.i.d
\newcommand{\com}{c_{0,m}} % 
\def\bR{\mathds{R}} % espace des reels -> j'ai harmonise les deux notations existantes dans article
\def\etaa{\gamma}  % la constante anciennement \eta dans Far region \frac{\eta \upsilon}{d^4}
\def\bbH{\mathds{H}} % RKHS pour $h$ - $\phi$ - $\Phi$
\def\bbL{\mathds{L}} % RKHS pour $\ell$ - $\lambda$ - $L$
\newcommand{\E}{\mathds{E}} %Espérance
\def\vn{\Gamma_n}

\begin{document}
\sloppy

\begin{frontmatter}

\title{SuperMix: Sparse Regularization for Mixtures}
\runtitle{L1 Regularization of Mixture Problems}

\begin{aug}

\author{\fnms{Y.} \snm{De Castro${}^{\star}$}\ead[label=e1]{yohann.de-castro@ec-lyon.fr},}
\author{\fnms{S.} \snm{Gadat${}^{\circ\ddagger}$}\ead[label=e2]{sebastien.gadat@tse-fr.eu},}
\author{\fnms{C.} \snm{Marteau${}^{\bullet}$}\ead[label=e3]{marteau@math.univ-lyon1.fr}}
\and
\author{\fnms{C.} \snm{Maugis-Rabusseau${}^{\dagger}$.}\ead[label=e4]{maugis@insa-toulouse.fr}}

\runauthor{De Castro, Gadat, Marteau and Maugis-Rabusseau}

\affiliation{ENPC}
\address{${}^\star$Institut Camille Jordan, CNRS UMR 5208\\ École Centrale de Lyon\\
 F-69134 Écully, France.%\\\printead{e4}
}  

\address{${}^{\circ}$ Toulouse School of Economics, CNRS UMR 5314\\ Universit\'e Toulouse 1 Capitole\\
Esplanade de l'Université, Toulouse, France.\\
${}^{\ddagger}$ Institut Universitaire de France.} 

\address{${}^{\bullet}$Institut Camille Jordan, CNRS UMR 5208\\
 Universit\'e Claude Bernard Lyon 1  \\ F-69622 Villeurbanne, France.}

\address{${}^{\dagger}$Institut de Math\'ematiques de Toulouse; UMR5219\\
Universit\'e de Toulouse; CNRS\\
INSA, F-31077 Toulouse, France}
\end{aug}

%%%%%%%%%%%%%%%%%%%%%%abstract%%%%%%%%%%%%%%%%%%%%%%%%%%%
\begin{abstract}
This paper investigates the statistical estimation of a discrete mixing measure $\mu^0$ involved in a kernel mixture model. Using some recent advances in $\ell_1$-regularization over the space of measures, we introduce a ‘‘data fitting and regularization'' convex program for estimating $\mu^0$ in a grid-less manner  from a sample of mixture law, this method is referred to as Beurling-LASSO. 

Our contribution is two-fold: we derive a lower bound on the bandwidth of our data fitting term depending only on the support of $\mu^0$ and its so-called ‘‘minimum separation'' to ensure quantitative support localization error bounds; and under a so-called ‘‘non-degenerate source condition'' we derive  a non-asymptotic support stability property. This latter shows that for a sufficiently large sample size $n$, our estimator has exactly as many weighted Dirac masses as the target~$\mu^0$, converging in amplitude and localization towards the true ones. 
Finally, we also introduce some tractable algorithms for solving this convex program based on ‘‘Sliding Frank-Wolfe'' or ‘‘Conic Particle Gradient Descent''.

Statistical performances of this estimator are investigated designing a so-called ‘‘dual certificate'', which is appropriate to our setting. Some classical situations, as \textit{e.g.} mixtures of super-smooth distributions (\textit{e.g.} Gaussian distributions) or ordinary-smooth distributions (\textit{e.g.} Laplace distributions), are discussed at the end of the paper. 
\end{abstract}



\end{frontmatter}

%%%%%%%%%%%%%Title%%%%%%%%%%%%%%%%%%%%%%%%%%%%%%%%%%%%%%%
\maketitle

\section{Introduction}
\subsection{Mixture problems}
In this paper, we are interested in the estimation of a mixture distribution $\mu^0$ using some i.i.d. observations $\mathbf{X}:=(X_1,\ldots,X_n)\in(\bR^d)^n$ with the help of some $\ell_1$-regularization methods. More precisely, we consider the specific situation of a discrete distribution $\mu^0$ that is given by a finite sum of $K$ components:
\eq
\label{eq:target}
\mu^0:=\sum_{k=1}^{K} a_k^0\delta_{t_k}
\qe
where the set of positive weights $(a_k^0)_{1 \leq k \leq K}$ defines a discrete probability distribution, \textit{i.e.}  each $\delta_{t_k}$ is a Dirac mass at point $t_k\in \bR^d$ while
\[
\sum_{k=1}^Ka_k^0=1 \qquad \text{and} \qquad \forall k \in [K]:=\{1,\ldots,K\}:  \quad  a_k^0 >0\,.
\]
 We denote by $S^0:=\{t_1,\ldots,t_K\}$ the support of the target  distribution $\mu^0$. This  distribution is indirectly observed: we assume that our set of observations $\mathbf{X}$ in $\bR^d$ satisfies
  \[
X_i\iid \sum_{k=1}^{K} a_k^0 F_{t_k}\,,\quad  \forall i\in[n]:=\{1,\ldots,n\}\,,
\]
where $(F_{t})_{t\in\bR^d}$ is a family of \textit{known} distributions on $\bR^d$. Below, we consider the so-called location model where each distribution $F_t$ has a density with respect to the Lebesgue measure on $\bR^d$ given by the density function $\varphi(\cdot-t)$, where $\varphi$ denotes a \textit{known} density function. In this case, the density function $f^0$ of the data $\mathbf{X}$ can be written as a convolution, namely
\eq
\label{eq:target_prediction}
f^0(x) =\sum_{k=1}^{K} a_k^0\varphi(x-t_k)\,,\quad \forall x\in\bR^d\,.
\qe
\begin{remark}
Equation~\eqref{eq:target_prediction}
has a simple interpretation in the context considered here: the law of one observation $X_i$ is given by a sum of two independent random variables $U^0$ and $E$:
\[
X_i\sim U^0+E\,,
\]
where $U^0\in S^0$ is distributed according to $\mu^0$ $($i.e., the mixing law~\eqref{eq:target}$)$  and   $E$ has a distribution of density $\varphi$ with respect to the Lebesgue measure on $\bR^d$. In this context, recovering the distribution of $U^0$ from the sample $\mathbf{X}$ appears to be an inverse (deconvolution) problem. The main difference with former contributions (see, e.g.~\cite{Meister} for a comprehensive introduction) is that the probability measure associated to $U^0$ is discrete, which avoids classical regularization approaches. 
\end{remark}

Equation~\eqref{eq:target_prediction} is known in the literature as a \textit{mixture model}. A mixture model allows to describe some practical situations where a population of interest is composed of $K$ different sub-populations, each of them being associated to a proportion $a_k^0$ and to a location parameter~$t_k$. Mixture models have been intensively investigated during the last decades and have been involved in several fields as biology, genetics, astronomy, among others. We refer to~\cite{bookmixture06,McLachlan_Peel} for a complete  overview.

\subsection{Previous works}
The main goal of this paper is to provide an estimation of the discrete mixture law $\mu^0$ introduced in~\eqref{eq:target}. When the component number $K$ is available, the maximum likelihood estimator (MLE) appears to be the most natural candidate. Although no analytic expression is available for the model~\eqref{eq:target_prediction}, it can be numerically approximated. We mention for instance the well-known EM-algorithm  and refer to~\cite{wu83}, who established some of the most general convergence results known for the EM algorithm. 
However, the MLE (and the related EM-algorithm) does not always provide satisfactory results. First, the MLE suffers from several drawbacks (see, \textit{e.g.},~\cite{Lecam}) such as non-uniqueness of the solution, and second, obtaining theoretical guarantees for the EM-algorithm is still a difficult question (see, \textit{e.g.}, the recent contributions~\cite{balakrishnan17,dwivedi18}). Several alternative methods have been proposed in this context. Some contributions extensively use the MLE point of view to derive consistent properties in general semi-parametric models, including the Gaussian case (see \textit{e.g.}~\cite{vanderVaartmix}), whereas some other ones developed some contrast functions in a semi-parametric framework: with symmetry and number of component assumptions in~\cite{BMV_2006,BV_2014}, or with a fixed number of component settings in~\cite{Gadat_Kahn_Marteau_Maugis} and a $L^2$ contrast.
As a particular case,  the Gaussian setting has attracted a lot of attention:  a model selection strategy is developed in~\cite{Cathy} and a specific analysis of the EM algorithm with two Gaussian components is provided in~\cite{xu16}. The article~\cite{balakrishnan17} provides a general theoretical framework to analyze the
convergence of the EM updates in a neighborhood of the MLE, and derives some non-asymptotic
bounds on the Euclidean error of sample-based EM iterates. 
 Some of the aforementioned papers provide better results (for instance with parametric rates of convergence for the estimation of the weights $a^0_k$, see \textit{e.g.}~\cite{nguyen13,KH15}), but are obtained in more constrained settings: known fixed number of components (often $K=2$), univariate case, ...

{Our estimator will be any solution to a convex program and it does not require to know the number $K$ of components in the mixing law $\mu^0$. This estimator is based on ideas from super-resolution and ‘‘off-the-grid'' methods~\cite{Bhaskar_Tang_Recht12,Candes_FernandezGranda_14}, where one aims at recovering a discrete measure from linear measurements. The so-called ‘‘{\it sparse deconvolution}'' problem fits this framework since it concerns with estimating a target measure from the observation of some product of convolution between the target measure and known kernel as $f^0$ in~\eqref{eq:target_prediction}. Note that in mixture models, we do not observe $f^0$ but rather a sample drawn from it, and standard strategies such that $(1.15)$ in~\cite{Candes_FernandezGranda_14} cannot be invoked here. However, remark that one of the main advances has been the construction of the so-called ‘‘{\it dual certificate}'' in~\cite{Candes_FernandezGranda_14} which is the key to demonstrate the success of discreteness inducing norm regularization (see {\it e.g. }\cite{DeCastro_Gamboa_12,Candes_FernandezGranda_14,Duval_Peyre_JFOCM_15,de2017exact}).

Recent works have addressed  mixture models while assuming that the sampling law is known. }For example, the authors of~\cite{poon2018support}  study some dimension reduction techniques such as random ‘‘sketching'' problems using ‘‘off-the-grid'' minimization scheme. They prove convergence of random feature kernel towards the population kernel. We emphasize that the statistical estimation in terms of the sample size~$n$ has not been considered in the super-resolution research field. To the best of our knowledge, this paper is the first that bridges the gap between the recent ‘‘off-the-grid'' sparse regularization methods and a sharp statistical study of this estimation procedure in terms of the sample size and the bandwidth of the data fitting term.
 
\subsection{Contribution}
In this paper, we propose an estimator $\hat{\mu}_n$ of the measure 
$\mu^0$ (see Equation~\eqref{eq:target}) inspired by some recent results in $\ell_1$-regularization on the space of measures, sometimes referred to as super-resolution methods (see, \textit{e.g.},~\cite{DeCastro_Gamboa_12, Candes_FernandezGranda_14}). We investigate the statistical theoretical performances of $\hat{\mu}_n$. This estimator $\hat \mu_n$ is built according to the minimization of a criterion on the space of real measures on $\bR^d$ and does not require any grid for its computation. {The stability result and the construction of the {\it dual certificate} given in~\cite{Candes_FernandezGranda_14} played a central role in our work to obtain the statistical recovery. However, these authors work on the torus and their construction provides periodic dual certificates which are not useful in our present framework. One important contribution of this paper is thus a novel dual certificate construction, interpolating phases/signs on $\mathds R^d$ (and not the $d$-dimensional torus as in~\cite{Candes_FernandezGranda_14}). We also investigate the stability with respect to \textit{sampling} of our estimation strategy, \textit{i.e.} the ability of our procedure to recover the mixture when we compute $\hat{\mu}_n$ up to some i.i.d. observations $(X_i)_{i \in [n]}$ with $n \to +\infty$, which is a different problem from the stability issue studied in~\cite{Candes_FernandezGranda_14} that asserts the variation of the super-resolution solutions with  respect to an $\ell_1$ norm control on the low-frequency data.}

The minimized criterion requires to tune two parameters: a bandwidth parameter of the data fitting term denoted by  $m\geq 1$  and an $\ell_1$-regularization tuning parameter denoted by $\kappa>0$ below. We prove that the bandwidth parameter $m$ depends only on the intrinsic hardness of estimating the support~$S^0$ of the target $\mu^0$ through the so-called ‘‘minimum separation'' $\Delta$ introduced in~\cite{Candes_FernandezGranda_14} that refers to the minimal distance between two spikes:
 \[
 \Delta := \min_{k \neq \ell} \|t_k-t_\ell\|_2\,.
 \]

\noindent
We now assess briefly the performances of $\hat \mu_n$. We emphasize  that a complete version is displayed in Theorem \ref{thm:main} (for points $i)$ and $ii)$) and Theorem~\ref{thm:yoyo} (for point $iii)$) later on. 
 
\begin{theorem}
{Assume that the kernel $\varphi$ satisfies~\eqref{eq:H1_intro} with $\eta=4m$ (see Section~\ref{s:super} for a definition) for a bandwidth $m$ verifying
\eq
\label{eq:bias}
 m \gtrsim \sqrt{K} d^{3/2} \Delta^{-1}_{+}  \quad \mathrm{where} \quad \Delta_+ = \min(\Delta,1).
\qe
Then, some quantity $\mathcal{C}_{m}(\varphi)>0$ exists such that, setting
\eq
\label{eq:variance}
\rho_n=\mathcal O\Big(\sqrt{\frac{m^d}{n}}\,\Big),
\qe 
 our estimator $\hat{\mu}_n$ satisfies:}
\begin{itemize}
\item[$i)$] Spike detection property:
$$\forall A \subset \bR^d, \quad  \E[\hat\mu_n(A)]  \gtrsim \rho_n\mathcal{C}_{m}(\varphi) \quad\Longrightarrow \quad \min_{k\in[K]} \inf_{t\in A}\| t-t_k\|_2^2 \lesssim \frac{1}{m^2}.  
$$

\item[$ii)$] Weight reconstruction property:
$$
\forall k \in [K]: \qquad
\E \left[ |a_k^0-\hat{\mu}_n(\bbN_k(\epsilon))| \right] \lesssim 
\rho_n\mathcal{C}_{m}(\varphi),
$$
where 
$\bbN_k(\epsilon)$ denotes a region that contains $t_k$ and $\epsilon=\epsilon_{n,m}(d)$ is made explicit later on.
\item[$iii)$]
Support stability property: if $\varphi$ satisfies the \emph{Non-Degenerated Bandwidth} condition~\eqref{eq:NDSCB} (see Section~\ref{sec:support} for a definition), for $n$ large enough, with an overwhelming probability, $\hat\mu_n$ can be written as
\[
\hat\mu_n=\sum_{k=1}^{\hat K} \hat a_k\delta_{\hat t_k}\,,
\]
with $\hat K=K$. Furthermore, $(\hat a_k,\hat t_k)\to(a^0_k,t_k)$ for all $k\in[K]$, as $n$ tends to infinity.
\end{itemize}
\end{theorem} 

\noindent
{Note that the constant $\mathcal{C}_{m}(\varphi)$ will depend on other quantities introduced later. It will be specified in Proposition~\ref{prop:up}. }

These three results deserve several comments.
First, $i)$ indicates that when a set $A$ has enough mass w.r.t. the estimated measure $\hat \mu_n$, it includes a true spike with an accuracy of the order $m^{-2}$. 
The second result $ii)$ provides some statistical guarantees on the mass set by $\hat{\mu}_n$ near a true spike $t_k$ that converges to $\mu^0(\{t_k\}) = a_k^0$.
\noindent
Condition~\eqref{eq:NDSCB} is inspired from the so-called ‘‘non-degenerated source condition'' (NDSC) introduced in~\cite{Duval_Peyre_JFOCM_15} and allows to derive the support stability. 
The last result $iii)$ shows that, for large enough sample size, $\ell_1$-regularization successfully recovers the number of mixing components. The estimated weights on the  Dirac masses then converge towards the true ones in amplitudes and localizations.

The bandwidth $m$ has to be adjusted to avoid over and under-fitting. Condition~\eqref{eq:bias} ensures that the target point is admissible for our convex program and it may be seen as a condition to avoid a large bias term and under-fitting. Condition~\eqref{eq:variance} ensures that the sample size is sufficiently large with respect to the model size $m$ and it might be seen as a condition to avoid over-fitting and therefore to upper-bound the variance of estimation.

Below, we will pay attention to the role of Fourier analysis of $\varphi$ and to the dimension $d$ of the ambient space.
These results are applied to specific settings {(super-smooth and ordinary-smooth mixtures)}. 

\subsection{Outline}
This paper is organized as follows. Section \ref{sec:definitions} introduces some standard ingredients of $\ell_1$ regularization methods and gives a deterministic analysis of the exact recovery property of~$\mu^0$ from $f^0$. Section \ref{s:blasso} provides a description of the statistical estimator $\hat{\mu}_n$ derived from a deconvolution with a Beurling-LASSO strategy (BLASSO) (see \textit{e.g.}~\cite{DeCastro_Gamboa_12}). {Tractable algorithms solving BLASSO when the observation is a sample from a mixing law are introduced in Section~\ref{sec:algos}}. Section \ref{sec:stat_mu0} focuses on the statistical performances of our estimator whereas Section \ref{s:rates} details the rates of convergence for specific mixture models. The main proofs are gathered in Section \ref{sec:proofs} whereas the most technical ones are deferred to the appendix.
%companion supplementary material~\cite{super_mix_supp}. 

\section{Assumptions, notation and first results}\label{sec:definitions}

This section gathers the main assumptions  on the  mixture model~\eqref{eq:target_prediction}. Preliminary theoretical results in an ``ideal'' setting are   stated in order to ease the understanding of the forthcoming paragraphs. 

\subsection{Functional framework}
We introduce some notation  used all along the paper.

\begin{definition}[Set $(\cM(\bR^d,\bR),\|\cdot\|_1)$]
 We denote
by $(\cM(\bR^d,\bR),\|\cdot\|_1)$ the space of \emph{real valued} measures on $\bR^d$ equipped with the \emph{total variation norm}~$\|\cdot\|_1$, which is defined as 
\[
 \| \mu \|_1 := \int_{\bR^d} \dd |\mu | \quad \forall \mu \in \cM(\bR^d,\bR)\,,
 \]
where $|\mu| = \mu^+ + \mu^-$ and $\mu= \mu^+ - \mu^-$ is the Jordan decomposition associated to a given measure $\mu \in \cM(\bR^d,\bR)$. 
\end{definition}

\noindent
A standard argument proves that the total variation of $\mu$ is also described with the help of a variational relationship:
$$
\|\mu\|_1 = \sup \left\{ \int_{\bR^d} f \dd \mu \,: f \, \text{is} \, \text{$\mu$-measurable and } \, |f|\leq 1 \right\}.
$$
Recall that $\varphi$ used in Equation~\eqref{eq:target_prediction}
is a probability density function so that $\varphi\in   L^1(\bR^d)$. 

\begin{definition}[Fourier transform over $L^1(\bR^d)$ and $\cM(\bR^d,\bR)$]
We denote by $\cF$ the Fourier transform defined by:
\[
\forall x\in\bR^d,\forall f\in L^1(\bR^d),\quad \cF[f](x):=\int_{\bR^d}e^{-\imath x^\top\omega}f(\omega)\dd \omega\,.
\]
A standard approximation argument extends the Fourier transform to $\cM(\bR^d,\bR)$ with:
\[
\forall x\in\bR^d,\forall \mu\in \cM(\bR^d,\bR),\quad \cF[\mu](x):=\int_{\bR^d}e^{-\imath x^\top\omega}\dd \mu(\omega)\,.
\]
\end{definition}
We denote by $\cC_0(\bR^d,\bR)$ the space of continuous real valued functions {\it vanishing at infinity} on~$\bR^d$ and recall that $\mathcal{F}\left(L^1(\bR^d)\right)$ is a dense subset of $\cC_0(\bR^d,\bR)$. We shall also introduce the convolution operator $\Phi$ as
\eq
\label{eq:def_Phi}
\mu\mapsto \Phi(\mu):= \varphi\star\mu=\int_{\bR^d}\varphi(\cdot-x)\mathrm \dd\mu(x)\,,\quad\mu\in\cM(\bR^d,\bR)\,,
\qe
and
it holds equivalently that
 (see \textit{e.g.}~\cite[Section 9.14]{rudin1991functional}):
\eq
\label{eq:Fourier_Phi}
\forall \mu\in \cM(\bR^d,\bR)\,,\quad \cF[ \Phi(\mu)]=\cF[\varphi]\cF[ \mu]\, .
\qe
Concerning the density $\varphi$ involved in (\ref{eq:target_prediction}), we will do the following assumption.
\eq 
\label{eq:H0}\text{The function} \ \varphi \ \text{is a {\it bounded continuous symmetric function of positive definite type}.} \tag{$\cH_0$} \qe

\noindent
In particular, the positive definite type property involved in Assumption (\ref{eq:H0}) is equivalent to require that for any finite set of points $\{x_1,\ldots,x_n\} \in \bR^d$ and for any $(z_1,\ldots,z_n) \in \mathds{C}^n$:
$$
\sum_{i=1}^n \sum_{j=1}^n \varphi(x_i-x_j) z_i \bar{z}_j \geq 0.
$$
In what follows, we consider $h: \bR^d \times \bR^d \longrightarrow \bR$  the function defined by $ h(x,y) = \varphi (x-y)$ for all $ x,y \in \bR^d$. In such a case, Assumption (\ref{eq:H0}) entails that $h(\cdot,\cdot)$ is a bounded continuous symmetric positive definite kernel. By Bochner's theorem (see, \textit{e.g.},~\cite[Theorem 11.32]{rudin1991functional}),~$\varphi$ is the inverse Fourier transform of a nonnegative measure~$\Sigma$ referred to as the \textit{spectral measure}. The Fourier inversion theorem states that $\Sigma$ has a nonnegative density $\sigma\geq0$ with respect to the Lebesgue measure on $\bR^d$ such that $\sigma\in L^1(\bR^d)$. Hence, it holds from the preceding discussion that  
\eq
\label{eq:Fourierphi}
\varphi=\cF^{-1}[\sigma]\text{ for some nonnegative }\sigma\in L^1(\bR^d)\,.
\qe 
Below, the set of points where the Fourier transform of a function does not vanish will play an important role. We will denote this support by $\mathrm{Supp}(\sigma)$:
$$
\mathrm{Supp}(\sigma) = \left\{ \omega \in \bR^d \, : \, \sigma(\omega) \neq 0\right\}.
$$
Some examples of densities $\varphi$ that satisfies (\ref{eq:H0}) will be given and discussed in the forthcoming sections. We emphasize that this assumption is not restrictive and concerns for instance Gaussian, Laplace or Cauchy distributions, this list being not exhaustive. \\

\noindent
\textbf{Additional notation.} Given two real sequences $(a_n)_{n\in \mathds{N}}$ and $(b_n)_{n\in\mathds{N}}$, we write $a_n \lesssim b_n$ (resp. $a_n \gtrsim b_n$) if there exists a constant $C>0$ independent of $n$ such that $a_n \leq b_n$ (resp. $a_n \geq b_n$) for all $n\in \mathds{N}$. Similarly, we write $a_n \ll b_n$ if $a_n/b_n \rightarrow 0$ as $n\rightarrow +\infty$. The set $\mathds{N}^*$ stands for $\mathds{N} \setminus \lbrace 0 \rbrace$.

\subsection{Exact Recovery of $\mu^0$ from $f^0$ - Case $
\mathrm{Supp}(\sigma) = \bR^d$}

In this paragraph, we are interested in an ``ideal'' problem 
where we are looking for $\mu^0$ not from a sample $X_1,\dots, X_n$ distributed according to Equation~\eqref{eq:target_prediction}, but from the population law $f^0$ itself. Of course, this situation does not occur in practice since in concrete situations, we do not observe $f^0$ but an empirical version of it and we will have to preliminary use an estimation of $f^0$ before solving the deconvolution inverse problem.  Nevertheless, this toy problem already provides the first ingredients for a better understanding of the difficulties that arise in the context we consider. 

We stress that $f^0:=\Phi(\mu^0)$ where $\Phi$ is defined by~\eqref{eq:def_Phi}. Hence, this paragraph concerns the recovery of $\mu^0$ from its convolution by the kernel $\varphi$. We thus face an inverse (deconvolution) problem. 
Several solutions could be provided and a standard method would rely on Fourier inversion 
\[\mu^0=  \cF^{-1} \left[ \cF(f^0) \sigma^{-1} \right]\,,\]
where $\sigma$ is given by~\eqref{eq:Fourierphi}.

Here, we prove in a first step that this deconvolution problem can be efficiently solved  using a $\ell_1$-regularization approach.  We will be interested in the convex program~\eqref{eq:program_P0} given by:
\eq
\label{eq:program_P0}
\min_{\mu\in\cM(\bR^d,\bR)\ \text{:}\ \Phi(\mu)=f^0} \|\mu\|_1\,.
\qe
In particular, we investigate under which conditions the solution set of~\eqref{eq:program_P0} is the singleton~$\{\mu^0\}$, that we referred to as the ``Perfect Recovery'' property. 
We introduce the set of admissible points to the program~\eqref{eq:program_P0}, denoted by $\cM(f^0)$ and defined as:
$$
\cM(f^0) := \{\mu\in\cM(\bR^d,\bR)\ \text{:}\ \Phi(\mu)=f^0\}.
$$
In this context, some different assumptions on the kernel $\varphi$ shall  be used in our forthcoming results.

A first reasonable situation is when the spectral density $\sigma = \cF(\varphi)$ has its support equal to~$\bR^d$ and in this case we denote $\sigma>0$. This requirement can be summarized in the next assumption on the function $\varphi$:
\eq
\label{eq:H0_intro}
\tag{$\cH_\infty$}
\varphi=\cF^{-1}[\sigma]\text{, } \sigma(\omega)=\sigma(-\omega)\text{ a.e. with }\mathrm{Supp}(\sigma)=\bR^d: \forall \omega \in \bR^d \quad \sigma(\omega)>0.
\qe

\begin{example}
It may be shown that the set of densities $\varphi$ that satisfy both Assumptions~\eqref{eq:H0} and~\eqref{eq:H0_intro} include the Gaussian, Laplace, $B_{2\ell+1}$-spline, inverse multi-quadrics, Mat\'ern class (see, e.g.,~\cite[top of page 2397]{sriperumbudur2011universality}) examples.
\end{example}

Under Assumptions~\eqref{eq:H0} and~\eqref{eq:H0_intro}, any target measure $\mu^0\in\cM(\bR^d,\bR)$ is the only admissible point of the program~\eqref{eq:program_P0}.

\begin{theorem}[Perfect Recovery under~\eqref{eq:H0} and~\eqref{eq:H0_intro}]\label{thm:perfect}
Assume that the convolution kernel satisfies~\eqref{eq:H0} and~\eqref{eq:H0_intro}, then for any target $\mu^0$  the program~\eqref{eq:program_P0} has $\mu^0$ as unique solution point:
$$\cM(f^0) = \{\mu^0\}.$$
\end{theorem}
We emphasize that the previous result also holds for measures 
$\mu^0$ that are not  necessarily discrete.
The proof is given in Appendix \ref{app-PerfRecov}.
%Section 3.1 of~\cite{super_mix_supp}. 

\subsection{The Super-resolution phenomenon}
\label{s:super}

Theorem \ref{thm:perfect} entails that the measure $\mu^0$ can be recovered as soon as the spectrum of $f^0$ is observed and as soon as its support is $\bR^d$. Surprisingly, this latter assumption can be relaxed and reconstruction can be obtained in some specific situations. Such a phenomenon is associated to the super-resolution theory and has been popularized by~\cite{Candes_FernandezGranda_14} among others. 

Of course, this reconstruction is feasible at the expense of an assumption on the Fourier transform of $\varphi$. For the sake of simplicity, we assume that the spectral density $\sigma$ has a support that contains the hypercube $[-\eta,\eta]^d$ for some frequency threshold $\eta>0$: 
\eq
\label{eq:H1_intro}
\tag{$\cH_\eta$}
\varphi=\cF^{-1}[\sigma]\text{, } \sigma(\omega)=\sigma(-\omega)\text{ a.e. with } [-\eta,\eta]^d  \subset \mathrm{Supp}(\sigma).
\qe

\begin{remark}The densities $\varphi$ that satisfy~\eqref{eq:H1_intro} and for which $\mathrm{Supp}(\sigma)=[-\eta,\eta]^d$ act as ‘‘\,low pass filters''. The convolution operator $\Phi$ described in~\eqref{eq:def_Phi} cancels all frequencies above $\eta$, see for instance~\eqref{eq:Fourier_Phi}. Of course, the larger $\eta$, the easier the inverse deconvolution problem.
\end{remark}

\noindent
Under~\eqref{eq:H0} and~\eqref{eq:H1_intro}, the target measure $\mu^0\in\cM(\bR^d,\bR)$ is not the only admissible point in $\cM(f^0)$ to the program~\eqref{eq:program_P0}. We will need to ensure the existence of a specific function, called in what follows a \textit{dual certificate}, that will entail that $\mu^0$ is  still the only solution of the program~\eqref{eq:program_P0}. 

\begin{theorem}[Dual Certificate for~\eqref{eq:program_P0}]
\label{thm:Certificate_eta}
Assume that the density $\varphi$ satisfies~\eqref{eq:H0} and~\eqref{eq:H1_intro} for some $\eta>0$. Assume that $\mu^0$ and $S^0=\{t_1,\ldots,t_K\}$ are given by Equation~\eqref{eq:target} and that a function~$\mathcal{P}_\eta$ exists such that it satisfies the interpolation conditions:

\begin{itemize}
\item
$
\forall t \in\{t_1,\ldots,t_K\} \, : \mathcal{P}_\eta(t)=1 \quad \text{and} \quad  \forall t\notin\{t_1,\ldots,t_K\} \,:  |\mathcal{P}_\eta(t)|<1,
$ 
\end{itemize}
and the smoothness conditions:
\begin{itemize}
\item $\mathcal{P}_\eta\in\cC_0(\bR^d,\bR)\cap L^1(\bR^d)$,
\item  the support of the Fourier transform $\cF[\mathcal{P}_\eta]$ satisfies   $
\mathrm{Supp}\left(\cF[\mathcal{P}_\eta]\right) \subset [-\eta,\eta]^d$.
\end{itemize}
Then the program~\eqref{eq:program_P0} has $\mu^0$ as unique solution point $($Perfect Recovery$)$.
\end{theorem}

\noindent
The proof is given in Appendix \ref{s:proof_Certificate_eta}. %Section 3.2 of~\cite{super_mix_supp}. 
A construction of such a certificate~$\mathcal{P}_\eta$ is presented in Appendix \ref{s:dual}%Section 6 of~\cite{super_mix_supp}  
with some additional constraints.  In particular, it  will make it possible to address the more realistic statistical problem where only an empirical measure of the data is available. 

\begin{remark}
The previous theorem can be extended to the case where the convolution kernel is bounded, continuous and symmetric positive definite. The proof is the same substituting~$[-\eta,\eta]^d $ by the support~$\Omega$ of its spectral density. Remark that since $\sigma$ is nonzero, necessarily~$\Omega$ has a nonempty interior.
\end{remark}

\section{Off-The-Grid estimation via the Beurling-LASSO (BLASSO)}
\label{s:blasso}

In this section, we consider the statistical situation where the density $f^0$ is not available and we deal instead with a sample $\mathbf{X}=(X_1,\dots, X_n)$ of i.i.d. observations distributed with the density $f^0$. In this context, only the empirical measure 
\begin{equation}
\hat f_n:=\frac1n\sum_{i=1}^n \delta_{X_i},
\label{eq:empirical}
\end{equation}
is available, and our aim is to recover $\mu^0$ from $\hat f_n$. To this end, we use in this paper a {\it BLASSO} procedure (see \textit{e.g.}~\cite{Azais_DeCastro_Gamboa_15}). Namely we deal with the following estimator $\hat\mu_n$ of the unknown discrete measure $\mu^0$
defined as:
\begin{equation}
\hat \mu_n := \arg\min_{\mu \in \mathcal{M}(\bR^d,\bR)} \big\{C(\Phi\mu,\hat f_n) + \kappa \| \mu \|_1 \big\},
\label{def:mun}
\end{equation}
where $\kappa$ is a regularization parameter whose value will be made precise later on, and~$C(\Phi\mu,\hat f_n)$ is a \textit{data fidelity} term that depends on the sample $\mathbf{X}$.
The purpose of the data fidelity term is to measure the \textit{distance} between the target~$\mu^0$ and any candidate $\mu \in  \mathcal{M}(\bR^d,\bR)$. 

Some examples of possible cost functions  $C:\bbH \times \cM(\bR^d,\bR) \to\bR$ are discussed in  Section~\ref{sec:losses}. Our  goal is then to derive some theoretical results associated to this estimation procedure.

\subsection{Kernel approach}
\label{sec:losses}

\subsubsection{RKHS functional structure}
In order to design the data fidelity term,  we need to define a space where we can compare the observations $\mathbf{X}=(X_1,\dots, X_n)$ and any model $f=\varphi\star \mu=\Phi\mu$ for $\mu\in\cM(\bR^d,\bR)$. In this work, we focus our attention on a kernel approach. 

\paragraph{Reminders on RKHS}
The difficulty lies in the fact that the empirical law $\hat f_n$ introduced in~\eqref{eq:empirical} does not belong to~$\cC_0(\bR^d,\bR)$. To compare the prediction $\Phi\mu$ with $\hat f_n$, we need to embed these quantities in the same space. We consider here a Reproducing Kernel Hilbert Space (RKHS) structure, which provides  a lot of interesting properties and has been at the core of several investigations and applications in approximation theory~\cite{wahba}, as well as in the statistical and machine learning communities, (see~\cite{sriperumbudur2011universality} and the references therein). We briefly   recall the definition of such a space. 

\begin{definition}
Let $(\bbL,\|.\|_\bbL)$ be a Hilbert space containing function from $\bR^d$ to $\bR$. The space $\bbL$ is said to be a  RKHS if $\delta_x: f \mapsto f(x)$ are continuous for all $x\in \bR^d$ from $(\bbL,\|.\|_\bbL)$ to~$(\bR,|.|)$.
\end{definition}
The Riesz theorem leads to the existence of a function $\ell$ that satisfies the \textit{representation} property:
\begin{equation}
\langle f, \ell(x,.) \rangle_\bbL = f(x) \quad \forall f\in \bbL, \quad \forall x\in \bR^d.
\label{eq:repr}
\end{equation}
The function $\ell$ is called the \textit{reproducing kernel} associated to $\bbL$. 
Below, we consider a kernel~$\ell$ such that $\ell(x,y)=\lambda(x-y)$ for all $x,y \in \bR^d$ where $\lambda$ satisfies~\eqref{eq:H0}.
\noindent
Again, the Bochner theorem yields the existence of a \textit{nonnegative} measure $\Lambda\in\cM(\bR^d,\bR)$ such that $\lambda$ is its inverse Fourier transform 
\eq
\notag
\lambda = \cF^{-1}(\Lambda), \quad \text{namely} \quad \forall x\in\bR^d,\quad\lambda(x)=\int_{\bR^d}e^{\imath x^\top\omega}\dd \Lambda(\omega)\,.
\qe
Moreover, since $\lambda$ is continuous, $\Lambda$ is then a bounded measure and the
Mercer theorem (see \textit{e.g.}~\cite{berlinet}) proves that the
 RKHS $\bbL$ is exactly characterized  by
\begin{equation}
\bbL = \left\lbrace f:\bR^d \rightarrow \bR \ \mathrm{s.t.} \ \| f \|_\bbL^2 = \int_{\bR^d} \frac{ |\mathcal{F}[f](t)|^2}{\mathcal{F}[\lambda](t)} \dd t < +\infty \right\rbrace\,,
\label{eq:equaL}
\end{equation}
with dot product
\[
\forall f,g\in\bbL\,,\quad \langle f,g\rangle_{\bbL}=\int_{\bR^d} \frac{ \overline{\mathcal{F}[f]}(t)\mathcal{F}[g](t)}{\mathcal{F}[\lambda](t)} \dd t\,.
\]
\paragraph{Convolution in the RKHS}
The RKHS structure associated to the kernel $\lambda$ entails a comparison between the empirical measure and any candidate $\Phi \mu$. Indeed, 
a convolution operator $L$ similar to  the one defined in Equation (\ref{eq:def_Phi}) can be associated to the RKHS as pointed out by the next result.
\begin{prop}\label{prop:belongL}
For any $\nu \in \cM(\bR^d,\bR)$, the convolution $L\nu = \lambda\star \nu$ belongs to $\bbL$.
\end{prop}
\noindent 
The proof of Proposition~\ref{prop:belongL} is given in Appendix \ref{proof:belongL}. %Section 2.1 of~\cite{super_mix_supp}.

 \subsubsection{Data fidelity term}
 For any $\mu \in \cM(\bR^d,\bR)$, both $L\hat f_n$ and $L \circ \Phi\mu$ belong to~$\bbL$.  Hence, one may use the following data fidelity term
\eq
\label{eq:CL_norm}
\mathrm C_{\lambda}(\Phi\mu,\hat{f}_n):=\|L\hat f_n-L\circ\Phi\mu \|_{\bbL}^2, \quad  \forall  \mu \in \cM(\bR^d,\bR)\,.
\qe

\begin{example}\label{example}
An important example is given by the sinus-cardinal kernel $\mathrm{sinc}$. Given a frequency ``cut-off'' $1/\tauc>0$, one can consider the kernel 
\[
\lambda_{\tauc}(x):=\frac{1}{\tauc^d} \lambda_{\mathrm{sinc}}\left(\frac{x}{ \tauc}\right)
\quad \text{ where } \quad
\lambda_{\mathrm{sinc}}(x):=\prod_{j=1}^d\frac{\sin (\pi x_j)}{\pi x_j} \quad \forall x\in \bR^d.
\]
Then, the spectral measure is given by
\[
\dd\Lambda_{\tauc}(\omega)=\dd\Lambda_{\mathrm{sinc}}(\omega\pi \tauc):=\frac{1}{2^d}\prod_{j=1}^d\1_{[-1/\tauc,1/\tauc]}(\omega_j)\dd\omega\,, \quad \forall \omega \in \bR^d.
\]
In this particular case, we deduce that the convolution $L$ is a low-pass filter with a frequency cut-off $1/\tauc$ and the RKHS (denoted by $\bbL_\tauc$) is given by:
\begin{equation} 
\bbL_\tauc = \left\lbrace f 
\ \mathrm{s.t.} \ \| f \|_{\bbL_\tauc}^2 =\frac{1}{2^d} \int_{B_\infty(1/\tauc)} |\mathcal{F}[f]|^2 < +\infty\text{ and }\mathrm{Supp}(\mathcal{F}[f])\subseteq B_\infty(1/\tauc)\right\rbrace\,,
\label{eq:Lpassebas}
\end{equation}
where $B_\infty(1/\tauc)$ denotes the centered $\ell_\infty$ ball of radius $1/\tauc$. The RKHS $\bbL_\tauc$ then corresponds to the band-limited functions in $L^2(\bR^d)$ whose Fourier transform vanishes for a frequency larger than $1/\tauc$. In this context, our criterion becomes 
\eq
\notag
C_{\lambda_{\tauc}}(\Phi \mu, \hat{f}_n)=\frac1{2^{d}}\int_{[-1/\tauc,1/\tauc]^d}|\cF[\Phi \mu -\hat f_n](\omega)|^2\dd\omega
= \frac1{2^{d}}\int_{[-1/\tauc,1/\tauc]^d}| \sigma \cF[ \mu] -\cF[\hat f_n](\omega)|^2\dd\omega\,,
\qe
and it may be checked that
\eq 
\notag
C_{\lambda_{\tauc}}(\Phi \mu, \hat{f}_n)  
 = \frac{1}{2^d} \int_{\bR^d} \left| \lambda_{\tauc} \star (\Phi \mu-\hat f_n)(x) \right|^2 \dd x.
\qe
This loss focuses on the $L^2$-error of $\Phi\mu-\hat f_n$ for frequencies in the Fourier domain~$[-1/\tauc,1/\tauc]^d$. In some sense, the kernel estimator  $\lambda_\tauc \star \hat f_n$ has a bandwidth $\tauc$ that will prevent from over-fitting. 
\end{example}

We stress that, as it is the case in the previous low-pass filter example,  $C_{\lambda_{\tauc}}(\Phi \mu, \hat{f}_n)$ may depend on a tuning parameter (the bandwidth $\tauc$ in Example \ref{example}). For the ease of presentation, this parameter is not taken into account in the notation. However, its value will be discussed in Section \ref{s:rates}.

\subsubsection{Data-dependent computation}

The next proposition entails that the criterion $C_\lambda$ introduced in Equation~\eqref{eq:CL_norm} can be used in practice giving a useful expression to compute it.
\begin{prop}
\label{prop:critere}
For all $\mu \in \cM(\bR^d,\bR)$, we have:
\begin{eqnarray*}
\lefteqn{C_{\lambda }(\Phi \mu, \hat{f}_n)  = \|L\hat f_n-L\circ\Phi\mu\|_{\bbL}^2}  \\
&= &\|L\hat f_n\|_{\bbL}^2 + \int_{\bR^d}\big[-\frac2n\sum_{i=1}^n\lambda(t-X_i)\big](\Phi\mu)(t)\dd t+\int_{\bR^d\times\bR^d}\lambda(x-y)(\Phi\mu)(x)(\Phi\mu)(y)\dd x \dd y  
\,.
\end{eqnarray*}
\end{prop}
\noindent
We stress that  $\|L\hat f_n \|_{\bbL}^2$ does no depend on $\mu$ and can be removed from the criterion when it is used in the  program~\eqref{eq:program_P0}. The proof of Proposition \ref{prop:critere} is given in  Appendix \ref{proof:critere}. %Section 2.2 of~\cite{super_mix_supp}.

\subsection{Estimation by convex programming}

Our estimator is defined as a solution of the following optimization program with the data-fidelity term $C_{\lambda }(\Phi \mu, \hat{f}_n) $ introduced in (\ref{eq:CL_norm}). Hence, we  consider the optimization problem:
\eq
\inf_{\mu \in \mathcal{M}(\bR^d,\bR)}
\left\lbrace \frac12\| L\hat f_n - L\circ\Phi \mu \|_{\bbL}^2 + \kappa \| \mu \|_1 \right\rbrace, %\text{ such that }\int_{\bR^d} \dd \mu=1 
 \,,
\label{eq:blasso}
\tag{$\mathbf{P}_\kappa$}
\qe
where $\|\cdot \|_\bbL$ is the norm associated to the RKHS generated by $\lambda$ (see Section \ref{sec:losses}) and $\kappa>0$ is a tuning parameter whose value will be made precise later on. We emphasize that ($\mathbf{P}_\kappa$) is a convex programming optimization problem (convex function to be minimized on a convex constrained set).
The estimator~$\hat \mu_n$ is then any solution of
 \eq
\hat \mu_n \in  \arg\min_{\mu \in \mathcal{M}(\bR^d,\bR)} 
\left\lbrace \frac12\| L\hat f_n - L\circ\Phi \mu \|_{\bbL}^2 + \kappa \| \mu \|_1 \right\}\,.
\label{eq:muhat}
\qe
Algorithmic issues related to the computation of~\eqref{eq:muhat} are sketched in Section~\ref{sec:algos} and discussed in depth in Appendix \ref{app:algos}. 
%{Section 1 of~\cite{super_mix_supp}.} 

\begin{remark}
The tuning parameter $\kappa>0$ needs to be chosen carefully. First note that it may depend on the choice of the frequency cut-off $1/\tauc$ in $\lambda_{\tauc}$, which is  the bandwidth feature map (see Remark~\ref{example} for a definition). Our analysis shows that $\tau = 1/{4m}$ as in~\eqref{e:4m}, and $m$ is a standard nonparametric bandwidth in mixture models for which rates are given in Section~\ref{s:rates}. The main message being that it depends only on the regularity of $\varphi$ and on the sample size $n$ for $n$ large enough. %We suggest to tune $\tauc$ accordingly. %It remains to tune $\kappa>0$.
From a practical view point, it is not excluded to use a Cross-Validation scheme as it heuristically performs well for $L^1$-based methods such as LASSO. In this case, the score function can be chosen to be the data fidelity term $\| L\hat f_n - L\circ\Phi \mu_{\mathrm{cv}} \|_{\bbL}^2$ evaluated on the validation set.
From a theoretical view point, one may choose $\kappa$ as in Equation~\eqref{e:kappa}. Then, Equation~\eqref{eq:fourre_tout} results in 
\[
\kappa\geq\frac{\rho_n}{\mathcal{C}_m(\varphi,\lambda)}\,,
\]
and these quantities depend only on the regularity of $\varphi$ and the sample size $n$ for $n$ large enough, as studied in Section~\ref{s:rates}.
\end{remark}

Super-resolution is the ability to recover a discrete measure on the torus from some Fourier coefficients (recall that the Pontryagin's dual of the torus is $\bbZ^d$) while we want to recover a discrete measure on $\bR^d$ from some Fourier transform over $\bR^d$ (recall that the Pontryagin's dual of $\bR^d$ is $\bR^d$). In particular the dual of~\eqref{eq:blasso} does not involve a set of fixed degree trigonometric polynomials as in super-resolution but inverse Fourier transform of some tempered distribution.

%\pagebreak[3]

Hence, new theoretical guarantees are necessary in order to properly define the estimator~$\hat \mu_n$. This is the aim of the next theorem.
 In this view, we consider primal variables $\mu\in\cM(\bR^d,\bR)$ and $z\in\bbL$ and introduce the dual variable $c\in\bbL$ as well as the following Lagrangian:
\eq
\label{eq:lagrangian}
\cL(\mu,z,c):=
\frac12\| L\hat f_n - z \|_{\bbL}^2 + \kappa \| \mu \|_1
-\langle c,L\circ\Phi \mu-z\rangle_{\bbL}
%+\rho\left(\int_{\bR^d}\dd\mu-1\right)
\,.
\qe

It is immediate to check that if $z \neq L \circ \Phi\mu$, then the supremum of $\cL(\mu,z,c)$ over $c$  is $+ \infty$. Therefore, the primal expression coincides with  the supremum in the dual variables, namely
\[
\inf_{\mu,z}\sup_{c}\cL(\mu,z,c)=
\inf_{\mu \in \mathcal{M}(\bR^d,\bR)} 
\left\lbrace \frac12\| L\hat f_n - L\circ\Phi \mu \|_{\bbL}^2 + \kappa \| \mu \|_1 \right\}\, \Longleftrightarrow (\mathbf{P}_\kappa) .
\]  
In the meantime, the dual program of~\eqref{eq:blasso} is  given by
  \eq
  \sup_{c \in \bbL} \inf_{(\mu,z) \in \mathcal{M}(\bR^d,\bR) \times \bbL} \cL(\mu,z,c)\, .
  \label{eq:blasso_dual}
  \tag{$\mathbf{P}^*_\kappa$}
  \qe 
 
 %\pagebreak[3]
 
\begin{theorem}[Primal-Dual programs, strong duality]
\label{thm:sol_discrete}
The following statements are true.
\begin{itemize}
\item[$i)$] The primal problem~\eqref{eq:blasso} has at least one solution and it holds that 
\[
\hat z_n:=L\circ\Phi\hat\mu_n\quad \text{ and }\quad\hat m_n:=\|\hat \mu_n\|_1 \quad \text{are uniquely defined,}
\] 
hence, they do not depend on the choice of the solution~$\hat\mu_n$.
\item[$ii)$] The dual program of~\eqref{eq:blasso},  given by~\eqref{eq:blasso_dual}  satisfies
\[
\frac{\| L\hat f_n\|_{\bbL}^2}2-\inf
\Big\{\frac12\| L\hat f_n-c\|_{\bbL}^2\,:\ c\in\bbL\text{ s.t. }\|\Phi c\|_\infty\leq \kappa\Big\}\, \Longleftrightarrow (\mathbf{P}^*_\kappa),
\]
and there is {\it no duality gap} $($strong duality holds$)$. Furthermore, it has a unique solution  
\[
\hat c_n=L\hat f_n-\hat z_n\,.
\]
\item[$iii)$] Any solution $\hat \mu_n$ to~\eqref{eq:blasso} satisfies %the support inclusion:
\[
\mathrm{Supp}(\hat \mu_n)\subseteq\Big\{x\in\bR^d\ :\ |\hat\eta_n|(x)=1\Big\}\quad\text{and}\quad
\int_{\bR^d}{\hat\eta_n}\,\dd \hat\mu_n= \hat m_n\,,% \big[=\|\hat \mu_n \|_1\big]\,,
\]
where 
\eq
\label{e:dual_poly}
\hat\eta_n:=\frac{\Phi \hat c_n}\kappa=\frac1\kappa\Phi(L\hat f_n-z_n)\,,
\qe
{i.e. it is a sub-gradient of the total variation norm at point }$\hat\mu_n$.

\item[$iv)$] If $d=1$ and if at least one of the spectral measures $\Lambda$   or $\sigma$   has a bounded support, then  $\{x\in\bR\,:\,| \hat{\eta}_n|(x)=1\}$ is discrete with no accumulation point, any primal solution~$\hat\mu_n$ has an $($at most countable$)$ discrete support $\hat S\subset\bR$ with no accumulation point:
\eq
\label{eq:discrete}
\hat \mu_n=\sum_{t\in\hat S}\hat a_t\delta_t
\,.
\qe
\end{itemize} 
\end{theorem}
\noindent
The proof of this result can be found in Appendix \ref{proof:sol_discrete}. 
%Section 4 of~\cite{super_mix_supp}. 

\noindent
It is generally numerically admitted, see for instance~\cite[Page 939]{Candes_FernandezGranda_14}, that the extrema of the dual polynomial $\hat{\eta}_n=\Phi \hat c_n$ are located in a discrete set, so that  any solution to~\eqref{eq:blasso} has a discrete support by using $iii)$. However, this issue remains an open question. In practice, all solvers of~\eqref{eq:blasso} lead to discrete solutions: greedy methods are discrete by construction, and $L^1$-regularization methods empirically lead to discrete solutions, see \textit{e.g.}~\cite{Candes_FernandezGranda_14}.
Furthermore, as presented in Theorem~\ref{thm:yoyo}, our theoretical result shows that for large enough $n$ and under the so-called~\eqref{eq:NDSCB} condition, the support stability property holds. In this case, the solution of~\eqref{eq:blasso} is discrete with $\hat K=K$ atoms.

\begin{example} 
Observe that the low-pass filter defined in Example~\ref{example} satisfies the requirements of  $iv)$ in Theorem~\ref{thm:sol_discrete}:  we deduce that when  $d=1$, all solutions $\hat \mu_n$ are of the form~\eqref{eq:discrete}.
\end{example}

\subsection{Tractable Algorithms for BLASSO Mixture Models}
\label{sec:algos}

Available algorithms for solving~\eqref{eq:muhat}  with ``{\it \!off-the-grid\,}'' methodology can be roughly divided into two categories: greedy methods and Riemannian descent methods.
We emphasize that if the BLASSO has been studied in the past decade, the formulation~\eqref{eq:muhat} has two new important features. First the observation is a sample from a mixing law. Second, the data fidelity term has been tuned to incorporate a low pass filter kernel $\lambda$. 
For both methods,  we refer to Appendix \ref{app:algos}
%the supplementary material~\cite{super_mix_supp} 
for further details and references.

\begin{algorithm}[h]
    \caption{Sliding Frank Wolfe~Algorithm (SFW) for BLASSO Mixture Models}
    \label{sec:sfw-alg:sfw}
    \begin{algorithmic}[1]
    \State Initialize with $\hat\mu^{(0)}=0$
    \While{the stopping criterion is not met}
    \State\label{computeNextPos}
    $\hat\mu^{(k)}=\displaystyle\sum_{i=1}^{N^{(k)}}a_i^{(k)}\delta_{t_i^{(k)}}$, $a_i^{(k)}\in\mathds R$, $t_i^{(k)}\in\mathds R^d$ where $N^{(k)} = | Supp(\hat{\mu}^{(k)})|$ and find $t_\star^{(k)}$ such that 
    \eq
   \notag
    	t_\star^{(k)}\in\arg\max_{t\in\mathds R^d}\big|\eta^{(k)}(t)\big|\quad\text{where}\quad \eta^{(k)}=-\frac{\nabla\mathrm{F}(\hat\mu^{(k)})}\kappa
	\qe
    \If{$|\eta^{(k)}(t_\star^{(k)})|\leq 1$}{}\State {$\hat\mu^{(k)}$ is an {\it exact} solution {\bf Stop}}
        \Else{}
    \State
    Find $\hat\mu^{(k+\frac12)}=\displaystyle\sum_{i=1}^{N^{(k)}}a_i^{(k+\frac12)}\delta_{t_i^{(k)}}+a_i^{(k+\frac12)}\delta_{t_\star^{(k)}}$ such that 
    \eq
    \label{e:lasso_step}
    \tag{LASSO Step}
    a^{(k+\frac12)}\in\arg\min_{a\in\mathds R^{N^{(k)}+1}}F_{N^{(k)}+1}(a,t^{(k+\frac12)})+\kappa\|a\|_1
    \qe
    where $t^{(k+\frac12)}:=(t_1^{(k)},\ldots,t_{N^{(k)}}^{(k)},t_\star^{(k)})$ is kept fixed.
    \State\label{computeNextAmp} 
    Obtain $\hat\mu^{(k+1)}=\displaystyle\sum_{i=1}^{N^{(k)}+1}a_i^{(k+1)}\delta_{t_i^{(k+1)}}$ such that
       \eq
       \label{e:non_conv}
    (a^{(k+1)},t^{(k+1)})\in\arg\min_{(a,t)\in\mathds R^{N^{(k)}+1}\times(\mathds R^d)^{N^{(k)}+1}}F_{N^{(k)}+1}(a,t)+\kappa\|a\|_1
    \qe
\qquad \quad    using a non-convex solver initialized with $(a^{(k+\frac12)},t^{(k+\frac12)})$.
    \State Eventually remove zero amplitudes Dirac masses from $\hat\mu^{(k+1)}$.
    \EndIf
    \EndWhile
    \end{algorithmic}
\end{algorithm}

\paragraph{Greedy method: the Sliding Frank-Wolfe algorithm (SFW)}

The Frank-Wolfe algorithm is an interesting avenue for solving {\it differentiable convex} programs on {\it weakly compact convex} sets, see~\cite{denoyelle2019sliding} and references therein for further details, which can be adapted to compute approximate solutions of the BLASSO Mixture Models~\eqref{eq:muhat} with a supplementary \textit{sliding} step. 
For a measure $\mu_{a,t}$ that may be decomposed into a finite sum of Dirac masses, we define~$F_N$ the data-fitting term:
 \begin{equation}\label{def:FN}
\mu_{a,t}:=\displaystyle\sum_{i=1}^{N}a_i\delta_{t_i}\quad \text{and}\quad \mathrm F_N(a,t):={\mathrm F}(\mu_{a,t})
=\frac12\| L\hat f_n - \sum_{i=1}^{N}a_iL\circ\Phi\delta_{t_i} \|_{\bbL}^2\,.
\end{equation}

\noindent
The SFW method  is then described in Algorithm  \ref{sec:sfw-alg:sfw}. It is a greedy method that recursively builds 
\[
\eta_\mu:=-\frac{\nabla\mathrm{F}(\mu)}{\kappa}=\frac1{\kappa}{\Phi (L\hat f_n-L\circ\Phi\mu)}\,,%\quad\text{ so that }\quad \nabla\mathrm{F}(\mu)+\kappa = \kappa\big(1-\eta_\mu\big)\,,
\]
see Line 3 of Algorithm  \ref{sec:sfw-alg:sfw}.

\paragraph{ Conic Particle Gradient Descent (CPGD)}

Conic Particle Gradient Descent~\cite{chizat2019sparse} is an alternative promising avenue for solving BLASSO for Mixture Models~\eqref{eq:muhat}. The idea is still to discretize a positive measure into a system of particles, \textit{i.e.} a sum of $N$ Dirac masses following~\eqref{def:FN}
with $a_i=r_i^2$ and use a mean-field approximation in the Wasserstein space jointly associated with a Riemannian gradient descent with the conic metric. We refer to~\cite{chizat2019sparse} and the references therein for further details.
This method may be shown to be rapid, with a $\log(\epsilon^{-1})$ cost instead of $\epsilon^{-1/2}$ for standard convex programs. Adapted to the BLASSO for Mixture Models,  we derive in Algorithm \ref{sec:cpgd} a version of the Conic Particle Gradient Descent of~\cite[Algorithm 1]{chizat2019sparse} and we implemented this algorithm for Mixture Models in Figure~\ref{fig:cpgd}. 

More precisely, Figure~\ref{fig:cpgd} is a {\it proof-of-concept} and CPGD for Mixture Models would be investigated in future work. One may see that this method uncovers the right number of targets Dirac masses and their locations as some particules cluster around three poles. Some of particules vanishes and do not detect the support. Notice that a {\it soft-thresholding} effect tends to zero the small amplitudes as it may standardly be shown in $L^1$ regularization.

\begin{algorithm}[h]
    \caption{Conic Particle Gradient Descent~Algorithm for BLASSO Mixture Models}
    \label{sec:cpgd}
    \begin{algorithmic}[1]
    \State Choose two gradient step sizes $\alpha,\beta>0$ and the number of Particles $N\geq1$.
    \State Define $N$ Particles weights-locations $(r_i^{(0)},t_i^{(0)})_{i=1}^N$ representing the {\it initial measure}
    \[
    \hat\mu^{(0)}:=\frac1N\sum_{i=1}^N a_i^{(0)} \delta_{t_i^{(0)}}\,,
    \]
    where $a_i^{(0)}:=(r_i^{(0)})^2$. 
    \While{stopping criterion is not met}
    \State
    For all $i=1,\ldots,N$ update ({\it mirror descent step for $r$} associated to the KL divergence over $\mathds{R}_+^d$)
    \begin{align*}
    r_i^{(k+1)} &= r_i^{(k)}\exp\Big(2\,\alpha\,\kappa\,\big(\eta^{(k)}(t_i^{(k)})-1\big)\Big)\\
    t_i^{(k+1)} & = t_i^{(k)}  +\beta\,\kappa\,\nabla\eta^{(k)}(t_i^{(k)})
    \end{align*}
\qquad     where $\eta^{(k)}=-\frac{\nabla\mathrm{F}(\hat\mu^{(k)})}\kappa$, $\displaystyle\hat\mu^{(k)}:=\frac1N\sum_{i=1}^N a_i^{(k)} \delta_{t_i^{(k)}}${ and }$a_i^{(k)}=(r_i^{(k)})^2$.       \EndWhile
    \end{algorithmic}
\end{algorithm}

\begin{figure}[h]
\begin{center}
\includegraphics[height=6cm]{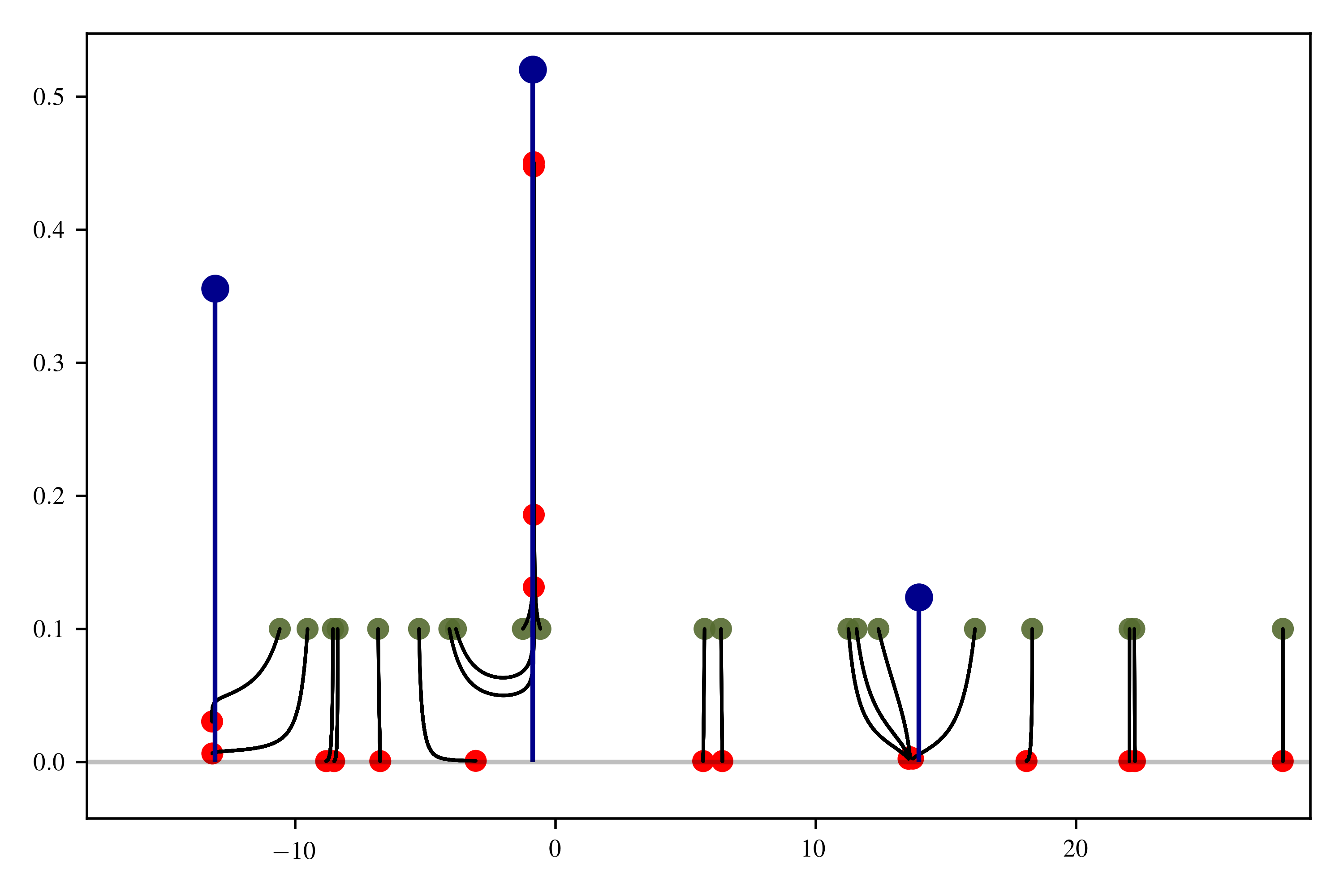}
\end{center}
\caption{Conic Particle Gradient Descent for BLASSO Mixture Models. We consider the mixing law~$\mu^0$ made by three Dirac masses (in blue) at location $(-13.1,-0.9,14.0)$ (chosen at random) and amplitudes $(0.36,0.52,0.12)$(chosen at random). We draw $n=200$ iid samples with respect to the mixture with density $f^0=\mu^0\star\varphi$ where $\varphi$ is the pdf of standard Gaussian. Then we start CPGD for BLASSO (with parameters $\kappa=0.01$ and $\tau=0.1$) with $20$ particules (in green) located at random and we run $2,500$ gradient steps (with parameters $\alpha=0.05$ and $\beta=1$) as in Algorithm~\ref{sec:cpgd}. The final locations $(t_i)$ and weights $(a_i)$ are given in red $($for readability we represented $(t_i,2*a_i))$.}
\label{fig:cpgd}
\end{figure}

\section{Statistical recovery of $\mu^0$}\label{sec:stat_mu0}
This section provides some theoretical results for $\hat \mu_n$, built  as the solution of~\eqref{eq:blasso}. Contrary to $\ell_1$-regularization in high-dimensions, standard  RIP or REC  compatibility conditions do not hold in our context, and all the cornerstone results of high-dimensional statistics cannot be used here. In our situation, the statistical analysis is performed using a ‘‘dual certificate''~$\Pm$ as in super-resolution, see~\cite{Azais_DeCastro_Gamboa_15,Bhaskar_Tang_Recht12,Candes_FernandezGranda_14,Duval_Peyre_JFOCM_15} for instance. The construction and the key properties satisfied by $\Pm$ are detailed in Section \ref{sec:dual_certificate}. However, our framework is quite different from super-resolution and we had to address two issues: build a dual certificate on  $\mathds R^d$ and adapt its ‘‘frequency cut-off'' (namely $4m$ in $iii)$ of Theorem~\ref{theo:main_certificate}) to the sample size $n$ and the tail of the kernel. This latter point is addressed in Section~\ref{s:rates}. 

%\pagebreak[3]

\subsection{Strong dual certificate}\label{sec:dual_certificate}

Let $S^0=\{t_1,\ldots,t_K\}$ be a fixed set of points in $\bR^d$ and define  $\Delta := \min_{k \neq \ell} \|t_k-t_\ell\|_2.$
For any $m\in \mathds{N}^*$, we consider the function $p_m^{\alpha,\beta}$ parameterized by a vector $\alpha$ and a matrix $\beta$ of coefficients, defined as:
\begin{equation}\label{def:pm_init}
 p^{\alpha,\beta}_m(t) = \sum_{k=1}^K \left\lbrace \alpha_k \psi_{m}(t-t_k) + \langle \beta_k,  \nabla \psi_{m}(t-t_k) \rangle \right\rbrace, \quad \forall t\in \bR^d,
 \end{equation}
where $\alpha=(\alpha_1,\ldots,\alpha_K)^{T}$, $\beta=(\beta_k^i)_{1 \leq k \leq K, 1 \leq i \leq d}$ with
\begin{equation}
\psi_m(.) = \psi^4(m .) \text{ with } 
\forall u=(u^1,\ldots,u^d) \in \bR^d \quad \psi(u) = \prod_{j=1}^d \mathrm{sinc}(u^j) 
\text{ and }  \mathrm{sinc}(x) =\frac{\sin(x)}{x}.
\label{eq:psim_init}
\end{equation} 
One important feature of $\psi_m$ is its ability to interpolate $1$ at the origin, while being positive and compactly supported in the Fourier domain. We then state the next result, which is of primary importance for the statistical accuracy of our procedure.

\begin{theorem}[Strong dual certificate]\label{theo:main_certificate}
Let be given a set of $K$ points $S^0=\{t_1,\ldots,t_K\}$ in~$\bR^d$ with
$\Delta := \min_{k \neq \ell} \|t_k-t_\ell\|_2$ {and $\Delta_+ = \min(\Delta,1)$}. Then, the following properties hold:
\begin{itemize}
\item $i)$ A function $\Pm$ defined by 
$
\Pm(t) = [p_m^{\alpha,\beta}(t)]^2 
$
exists with {$\displaystyle m \gtrsim  \sqrt{K}d^{3/2} \Delta_+^{-1} $}  such that
$$
\forall k \in [K],\ \Pm(t_k)=1 \qquad \text{and} \qquad 0 \leq \Pm  \leq 1$$
and 
$$\Pm(t)=1 \Longleftrightarrow t \in S^0 = \{t_1, \ldots, t_K\}.$$ 
\item $ii)$ A universal pair $(\upsilon,\etaa)$ independent from $n,m$ and $d$ exists such that  for 
{
$
\epsilon =\frac{\upsilon}{ m d}:
$
}
\begin{itemize}
\item \underline{Near region:} If we define 
$$\bbN(\epsilon):=\bigcup_{k=1}^K \bbN_{k}(\epsilon) \textrm{ where } \bbN_{k}(\epsilon):= \{ t : \|t-t_k\|_2 \leq \epsilon\},$$ 
 a positive constant $\mathcal{C}$ exists such that:
$$
\forall t \in \bbN_k(\epsilon): \qquad
0 \leq \Pm(t) \leq 1 - \mathcal{C} m^2 \|t-t_k\|_2^2.
$$
\item \underline{Far region:}
$$
\forall t \in \bbF(\epsilon):= \bR^d \setminus \bbN(\epsilon): \qquad 0 \leq \Pm(t) \leq 1-\etaa {\frac{\upsilon^2 }{  d^3}}.
$$
\end{itemize}
\item$iii)$ The support of the Fourier transform of $\Pm$ is growing linearly with $m$:  
$$
\mathrm{Supp}(\mathcal{F}[\Pm]) \subset [-4m,4m]^d \qquad \text{and} \qquad \|\Pm\|_2 \lesssim K^2 m^{-d/2}. 
$$
\item $iv)$ If  \eqref{eq:H1_intro} holds with $\eta=4m$, then an element $\com \in \bbL$ exists such that $\Pm = \Phi \com$. 
\end{itemize}

\end{theorem}

The proof of this result  is proposed in Appendix \ref{s:dual}. %Section 6 of~\cite{super_mix_supp}. 
This construction is inspired from the one given in~\cite{Candes_FernandezGranda_14}, which has been adapted to our specific setting. 
We emphasize that the size of the spectrum of $\Pm$ increases linearly with $m$, while the effect of the number of points $K$, the dimension $d$, and the spacing $\Delta$ between the location parameters $\lbrace t_1,\dots, t_K \rbrace$  is translated in the initial constraint {$m \gtrsim  \sqrt{K} d^{3/2}\Delta_+^{-1} $}.

We also state a complementary result, that will be useful for the proof of Theorem \ref{thm:main}, $iii)$.

\begin{corollary}\label{coro:Qmk}
Let be given a set of $K$ points $S^0=\{t_1,\ldots,t_K\}$  such that 
$\Delta := \min_{k \neq \ell} \|t_k-t_\ell\|_2$. Let {$m \gtrsim  \sqrt{K} d^{3/2}\Delta_+^{-1} $}.
Then, for any $k \in [K]$, a function $\Qmk$ exists such that
$$
\forall i \in [K] \qquad \Qmk(t_i) =\delta_{i}(k) \quad \text{and} \quad 0 \leq \Qmk \leq 1,
$$
and a universal couple of constants  $(\upsilon,\etaa)$ exists such that the function $\Qmk$ satisfies for {$\epsilon =\frac{\upsilon}{ m d}$:}
\begin{itemize}
\item[$i)$] Near region $\bbN_k(\epsilon)$:  a positive constant $\widetilde{\mathcal{C}}$ exists such that:
$$
\forall t \in \bR^d \qquad \|t-t_k\|_2 \leq \epsilon \Longrightarrow 0 \leq \Qmk(t) \leq 1 - \widetilde{\mathcal{C}} m^2 \|t-t_k\|_2^2,
$$
\item[$ii)$] Near region $\bbN(\epsilon) \setminus \bbN_k(\epsilon)$:
$$
\forall i \neq k \qquad 
\|t-t_i\|_2 \leq \epsilon \Longrightarrow |\Qmk(t)| \leq \widetilde{\mathcal{C}} m^2 \|t-t_i\|_2^2.
$$
\item[$iii)$] Far region $\bbF(\epsilon)$:
$$
\forall t \in \bbF(\epsilon),\ 
0 \leq \Qmk(t) \leq 1 - \gamma {\frac{\upsilon^2}{d^3}}.
$$
\item[$iv)$] A $c_{k,m} \in \bbL$ exists such that $\Qmk = \Phi c_{k,m}$. 
\end{itemize}
\end{corollary}
Proofs of $i),ii),iii)$ are similar to those of Theorem \ref{theo:main_certificate} and are omitted:  the construction of~$\Qmk$ obeys the same rules as the construction of $\Pm$ (the interpolation conditions only differ at points $t_i,i \neq k$ and are switched from $1$ to $0$). %Regarding now $iv)$, the upper bound of the~$\|\,\cdot\,\|_{\bbL}$ norm uses the same arguments as the ones given below in the proof of $ii)$, Proposition~\ref{prop:up}.
%Theorem \ref{thm:main}.

\subsection {Bregman divergence $\DP(\hat\mu_n,\mu^0)$}\label{sec:bregman}

Below, the statistical loss between $\hat\mu_n$ and $\mu^0$ will be obtained in terms of the Bregman divergence associated to the dual certificate $\Pm$ obtained in Theorem \ref{theo:main_certificate}. This divergence is defined by:
\begin{equation}\label{eq:def_bregman}
	\DP(\hat\mu_n,\mu^0) :=  \| \hat\mu_n \|_1 - \| \mu^0\|_1 - \int_{\bR^d} \Pm\dd(\hat\mu_n - \mu^0)\geq0\,.
\end{equation}

\noindent 
We also introduce the term $\vn$ defined as
$$ \vn = L\hat f_n - L\circ\Phi \mu^0,$$
which models the difference between the target $f^0 = \Phi \mu^0$ and its empirical counterpart $\hat f_n$ in the RKHS. The next result provides a control between $\hat \mu_n$ and~$\mu^0$ with the Bregman divergence. 

\begin{prop}
\label{prop:up}
Let $\Pm= \Phi \com$ the dual certificate obtained in Theorem \ref{theo:main_certificate}.
%Consider $m\geq 0$ such that $\Pm$ satisfies the properties $i)-iv)$ of Theorem \ref{theo:main_certificate}.
Let $(\rho_n)_{n\in \mathds{N}^*}$ be a sequence such that $\E[ \| \vn \|_\bbL^2] \leq \rho_n^2$ for all $n\in \mathds{N}^*$. 
If $\kappa$ is chosen such that 
\eq
\label{e:kappa}
\kappa = \frac{\rho_n }{\| \com \|_\bbL}
\qe 
and if $\hat{\mu}_n$ is defined in~\eqref{eq:blasso}, then:
\begin{itemize}
\item[$i)$] For any integer $n$:
$$
\mathds{E} \left[\DP(\hat\mu_n,\mu^0)\right] \leq \frac{3 \sqrt{2}}{2} \rho_n  \|\com\|_{\bbL} \, , 
$$
\item[$ii)$] $\com \in \bbL$ satisfies

\begin{equation}
 \|\com\|_{\bbL} {\leq}  \sqrt{ \frac{\| \Pm \|_2^2}
 {\displaystyle\inf_{\|t\|_\infty\leq 4m} \left\lbrace \sigma^2(t) \mathcal{F}[\lambda](t) \right\rbrace} } 
   \lesssim \underbrace{\frac{K^2 m^{-d/2}}
 {\sqrt{\displaystyle\inf_{\|t\|_\infty\leq 4m} \left\lbrace \sigma^2(t) \mathcal{F}[\lambda](t) \right\rbrace} }}_{:= \mathcal{C}_m(\varphi,\lambda)}
 % \quad \forall m\in \mathds{N}
 \,.
\label{eq:fourre_tout}
\end{equation}
\end{itemize}
%and $\sigma$ is the spectral measure associated to $\Phi$.
\end{prop}

\noindent
\newpage
The proof of Proposition \ref{prop:up} is postponed to Section \ref{s:propup}. The previous results indicate that the Bregman divergence between our estimator $\hat \mu_n$ and the target measure $\mu^0$ depends, up to some constants, on three main quantities:
\begin{itemize}
\item The variance of the empirical measure through the operator $L$ quantified by $\rho_n$,
\item The Fourier transform $\sigma$ of the convolution kernel $\varphi$ over the interval $[-4m;4m]^d$. This term measures the ill-posedness of the inverse problem, which is associated to the difficulty to recover $\mu^0$ with  indirect observations (here $f^0 = \Phi \mu^0$ and we need to invert~$\Phi$),
\item The structure of the RKHS  used to smooth the problem  identified through the kernel $\lambda$.
\end{itemize}

\begin{remark}
By using similar arguments to prove item ii) of Proposition \ref{prop:up}, we can complete item (iv) of Corollary \ref{coro:Qmk} as follows: 
A $c_{k,m} \in \bbL$ exists such that $\Qmk = \Phi c_{k,m}$ and 
\begin{equation}\label{ineq:ckm}
\|c_{k,m}\|_{\bbL} \lesssim \frac{K^2 m^{-d/2}}
{\sqrt{\displaystyle\inf_{\|t\|_\infty\leq 4m} \left\lbrace \sigma^2(t) \mathcal{F}[\lambda](t) \right\rbrace} }.
\end{equation}
\end{remark}

\begin{remark}
We will derive from Proposition \ref{prop:up} some explicit convergence rates in each specific situation, \textit{i.e.} as soon as the quantities involved in Equation~\eqref{eq:fourre_tout} are made precise on some concrete examples. These rates will depend on the tuning parameter $m$ for solving the optimization problem~\eqref{eq:blasso},
and on the choice of the kernel $\lambda$. Some examples will be discussed in Section \ref{s:rates}. 
 Indeed, $\kappa$ is related to $m$ through the relationship $\kappa = \rho_n / \| \com \|_\bbL$. Similarly, we will see in  Section \ref{s:rates} that the kernel $\lambda$ is also linked to $m$ in a transparent way. We stress that  according to Proposition \ref{prop:up}, $m \gtrsim  \sqrt{K} d^{3/2}\Delta_+^{-1} $. Such a condition will be satisfied provided $m$ is allowed to go to infinity with $n$ whereas $K,\Delta,d$ are kept fixed. 
\end{remark}

\begin{remark}
The upper bound proposed in Proposition \ref{prop:up} only uses items (iii) and (iv) of Theorem \ref{theo:main_certificate}. An enhanced control on the performances of $\hat \mu_n$ is provided in the next section. Alternative features will be also proposed with the alternative certificate $\mathcal{Q}_m$ introduced in Corollary \ref{coro:Qmk} .  
\end{remark}

\subsection{Statistical recovery of far and near regions}

The next result sheds light on the performance of the BLASSO estimator introduced in Equation~\eqref{def:mun}.
%\eqref{eq:blasso}. 
The goodness-of-fit reconstruction of the mixture distribution $\mu^0$ by $\hat{\mu}_n$
is translated by the statistical properties of the computed weights of $\hat{\mu}_n$ around the spikes of $\mu^0$ (the support points $S^0=\{t_1,\ldots,t_K\}$), which will define a family of $K$  \textit{near regions}, as well as the behaviour of $\hat{\mu}_n$ in the complementary set, the \textit{far region}. The sets $\bbF(\epsilon)$ and $\bbN(\epsilon)$ have already been introduced in Theorem \ref{theo:main_certificate}.
Our result takes advantage on the previous bounds and on $i)$ and $ii)$ of Theorem \ref{theo:main_certificate}. 

\begin{theorem}
\label{thm:main}
Let {$m \gtrsim  \sqrt{K} d^{3/2}\Delta_+^{-1} $} and let $\Pm$ be a dual certificate given in Theorem~\ref{theo:main_certificate}. Assume that $\hat \mu_n$ is the BLASSO estimator given by~\eqref{eq:blasso} with $\kappa=\kappa_n$ chosen in Proposition \ref{prop:up}. Let $\mathcal{C}_m(\varphi,\lambda)$ be the quantity introduced in Proposition \ref{prop:up}, $\hat\mu_n = \hat\mu_n^+ - \hat\mu_n^-$ the Jordan decomposition of $\hat\mu_n$. A universal couple of constants $(\gamma,\upsilon)$ exists such that,
if 
\begin{equation}
\label{eq:cmphi}
\epsilon =\frac{\upsilon}{ m d},
 \end{equation}    
\begin{itemize}
\item[$i)$] Far region and negative part:
$$
\E\left[ \hat\mu_n^-(\bR^d)\right] \leq  \frac{3\sqrt{2}}{2} \rho_n \, \mathcal{C}_m(\varphi,\lambda) \textrm{ and }  \E\left[\hat\mu_n^+(\bbF(\epsilon))\right] \leq  \frac{3\sqrt{2}}{2}\, {\frac{d^3}{\etaa \upsilon^2}}  \rho_n \, \mathcal{C}_m(\varphi,\lambda) .\\
$$
\item[$ii)$] Near region (spike detection): a positive constant $\mathcal{C}$ exists such that
$$\forall A \subset \bR^d, \quad  \E[\hat\mu_n^+(A)] > \frac{3\sqrt{2}}{2}\, {\frac{d^3}{\etaa \upsilon^2}} \,\rho_n \, \mathcal{C}_m(\varphi,\lambda) \quad\Longrightarrow \quad \min_{k\in[K]} \inf_{t\in A}\| t-t_k\|_2^2 \leq {\frac{\etaa \upsilon^2}{\mathcal{C} d^3 m^2}}.  
$$
\item[$iii)$] 
Near region (weight reconstruction):
 for any $k \in [K]$:
$$
\E \left[ |a_k^0-\hat{\mu}_n(\bbN_k(\epsilon))| \right] \lesssim 
\rho_n \mathcal{C}_m(\varphi,\lambda).
$$
\end{itemize}
\end{theorem}
The proof of this important result is deferred to Section \ref{s:proof_main}.

\begin{remark}
 It can be shown   in   specific situations (see, \textit{e.g.},  $iv)$ of Theorem \ref{thm:sol_discrete}) that the solution of~\eqref{eq:blasso} is indeed a discrete measure that can be written as
$$ \hat\mu_n = \sum_{t \in \hat{S}} \hat a_t \delta_{ t}. $$
In such a case, the relevance of the locations $\hat S$ of the reconstructed spikes $\hat a_t$ can be derived from the results of Theorem \ref{thm:main}. A discussion is given in some specific cases in Section \ref{s:rates}.
\end{remark}

\subsection{Support stability for large sample size}
\label{sec:support}
%Assume that $\lambda=\lambda_{\tau}$ with $\tau=1/(4m)$ for some bandwidth parameter $m\geq1$. 
We introduce $\mathcal P_0:=\Phi c_0$ the ‘‘minimal norm certificate'' (see \textit{e.g.}~\cite{Duval_Peyre_JFOCM_15}), which is defined by:
\[
c_0=\arg\min\big\{\|c\|_{\mathds L}^2\ :\ c\in\mathds L\quad \text{s.t.}\quad  \|\Phi c\|_{\infty}\leq 1\text{ and }(\Phi c)(t_k)=1\,,\ k\in[K]\big\}\,,
\]
when it exists.

We say that the support $S^0=\{t_1,\ldots,t_K\}$ of $\mu^0$ satisfies the \textit{Non-Degenerate Bandwidth} condition~\eqref{eq:NDSCB} if there exists $0<q<1$, $r>0$ and $\rho>0$ such that:
\eq
\label{eq:NDSCB}
\mathcal P_0\text{ exists}\,,\quad
\forall t\in\mathds F(r)\,,\  |\mathcal P_0(t)|<1-q\,,\quad
\forall t\in\mathds N(r)\,,\  \nabla^2 \mathcal P_0(t)\prec -\rho\,\mathrm{Id}_d\,.
\tag{NDB}
\qe
We then have the support stability result for large values of $n$.
\begin{theorem}
\label{thm:yoyo}
%There exists a universal constant $C>0$ such that the following holds. 
Let the triple $\lambda,\varphi,\mu^0$ be such that \eqref{eq:NDSCB} holds. Let $r_\kappa\in(0,\frac12)$ and set $\kappa_n=\sqrt{\lambda(0)}\,n^{-r_\kappa}$. Let $\hat \mu_n$ be the BLASSO estimator~\eqref{eq:blasso} 
%($without the constraint $\int\dd\mu=1)$ 
with a tuning parameter $\kappa=\kappa_n$.% such that $\|\vn\|_\bbL=o_{\mathds P}(\kappa_n)$ and $\kappa_n=o(1)$.  

Then for $n$ large enough, and with probability at least $1-Ce^{- n^{\frac{1}2-r_\kappa}}$ for a universal constant $C>0$, it holds that $\hat\mu_n$ has $K$ spikes with exactly one spike~$\hat t_k$ in each region $\mathds N_k(r)$. These spikes converge to the true ones, and so do the amplitudes $\hat a_k$, as $n$ tends to infinity.
\end{theorem}

\noindent
The proof can be found in Appendix \ref{s:stability}. %Section 5 of~\cite{super_mix_supp}. 
We emphasize that $C$ is independent from the dimension $d$, from the RKHS used $\bbL$ or the location of the spikes for example.
%Choices for are $\kappa_n$ is $\kappa_n=n^{-\frac12 + \delta_\kappa}$, or $\kappa_n=n^{-r}$ with $0<r<1/2$. 

\begin{remark}
In Theorem~\ref{thm:yoyo}, note that the data fidelity kernel $\lambda$ is fixed but in practice, the bandwidth of $\lambda$ often depends on the sample size $n$. Theorem~\ref{thm:yoyo} suggests the heuristics that the data fidelity kernel $\lambda=\lambda_n$ may depend on $n$ and it might be such that $\kappa_n=\sqrt{\lambda_n(0)}\,n^{-r_\kappa}$ vanishes as $n$ tends to infinity.
\end{remark}

\begin{remark}
Assume that the mixing kernel $\varphi$ is such that $\varphi\star\lambda=\psi_m\star\lambda$ where $\psi_m$ is defined by \eqref{eq:psim_init}, assume that the data fidelity kernel is such that $\lambda=\lambda_{1/(4m)}$ and assume that $\displaystyle m \gtrsim  \sqrt{K}d^{3/2} \Delta_+^{-1} $. Then our certificate $\Pm$ is called the {\bf vanishing derivatives pre-certificate} by~\cite[Section 4, Page 1335]{Duval_Peyre_JFOCM_15}. According to Theorem~\ref{theo:main_certificate}, we know that $\|\Pm\|_\infty\leq 1$. In this case, vanishing derivatives pre-certificate and certificate of minimal norm coincide so that $\Pm$ is the minimal norm certificate $\mathcal P_0$ appearing in~\eqref{eq:NDSCB}, and Theorem~\ref{theo:main_certificate} shows that~\eqref{eq:NDSCB} holds.
\end{remark}

\section{Rates of convergence for some usual mixture models}
\label{s:rates}

\subsection{Frequency cut-off and $\sinc$ kernel}
\label{sec:sinc}
In this section, we describe the consequences of Theorem \ref{thm:main}
for some  mixture models with classical densities $\varphi$.
For this purpose, we will consider the sinus-cardinal kernel $\sinc$ with a frequency cut-off $1/\tauc$, which is introduced in Example \ref{example}.
As a band-limited function $\lambda_\tauc$, we have that
$$
\|t\|_{\infty} \geq \frac{1}{\tauc} \Longrightarrow
\cF[\lambda_\tauc](t) = 0.
$$
Hence, to obtain a tractable version of Theorem \ref{thm:main} with $\mathcal{C}_m(\varphi,\lambda) < + \infty$ {(see Equation~\eqref{eq:fourre_tout})} we are led to consider $\tauc$ such that
\eq
\label{e:4m}
\frac{1}{\tauc} = 4 m.
\qe
In that case, $\cF[\lambda_\tauc]$ is a constant function over its support and the term $\mathcal{C}_{m}(\varphi,\lambda_\tauc) $ involved in Proposition \ref{prop:up} and Theorem \ref{thm:main} appears to be equal to 
$$
\mathcal{C}_{m}(\varphi,\lambda_\tauc) = \frac{K^2 m^{-d/2} 2^{d/2}}{ \inf_{\|t\|_{\infty} \leq 4 m} \sigma(t)}\,.
$$
To make use of Theorem \ref{thm:main}, we also need an explicit expression of  $(\rho_n)_{n\in \mathds{N}^*}$, which itself strongly depends on the kernel $\lambda_{\tauc}$.  In this context, some straightforward and standard computations yield
\begin{eqnarray*}
 \E\left[\| \vn \|_\bbL^2\right]
& = &\E \left[\| L\hat f_n - L f^0 \|_\bbL^2 \right], \\
& = & \E \left[\int_{\| t \|_\infty \leq 1/\tauc} \left| \mathcal{F}[\hat f_n](t) - \mathcal{F}[f^0](t) \right|^2dt \right], \\
& = & \int_{\| t \|_\infty \leq 1/\tauc} \mathrm{Var}(\mathcal{F}[\hat f_n](t)) \dd t \leq \frac{1}{n \tauc^d}.
\end{eqnarray*}
This provides a natural choice for the sequence $(\rho_n)_{n\in \mathds{N}^*}$ as
$$ \forall n\in \mathds{N}^* \qquad \rho_n = \frac{1}{\sqrt{n \tauc^d}}  = \frac{2^{d} m^{d/2}}{\sqrt{n}}.$$
Therefore, the statistical rate obtained in Theorem \ref{thm:main}
satisfies
\begin{equation}\label{eq:rhoncmphi}
\rho_n \mathcal{C}_m(\varphi,\lambda_\tauc) \leq \frac{K^2 2^{3d/2}}{\sqrt{n}  \displaystyle \times \inf_{\|t\|_{\infty} \leq 4 m}\sigma(t) }.
\end{equation}
We should understand the previous inequality as an upper bound that translates a tradeoff between the sharpness of the window where spikes are located (given by {$\epsilon = \mathcal{O} (1/(m d))$} in~\eqref{eq:cmphi}) and the associated 
statistical ability to recover a such targeted accuracy (given by the bound $\rho_n \mathcal{C}_m(\varphi,\lambda_{\tauc})$ on the Bregman divergence).
A careful inspection of the previous tradeoff leads to the following conclusion: the window size $\epsilon$ is improved for large values of $m$ but the statistical variability is then degraded according to the decrease rate of the Fourier transform~$\sigma$ of $\varphi$, which typically translates an inverse problem phenomenon.

Finally, we emphasize that the dimensionality effect is not only involved in the term $2^{3d/2}$ of Equation~\eqref{eq:rhoncmphi} but is also hidden in the constraint 
$$
m \gtrsim  \sqrt{K} d^{3/2}\Delta_+^{-1},
$$
used to build our dual certificate in Theorem \ref{theo:main_certificate}. By the way, we stress that at the end, the only tuning parameter involved in (\ref{eq:blasso}) appears to be $m$. \\

We now focus our attention to some specific and classical examples in mixture models: 

\begin{itemize}
\item the case of severely ill-posed inverse problems with an exponential decrease of the Fourier transform for large frequencies, which corresponds to super-smooth distributions. We emphasize that this class contains the standard benchmark of the Gaussian case, which will be discussed in details.
\item  the case of \textit{mildly ill-posed inverse problems} which encompasses multivariate Laplace distributions, Gamma distributions, double exponentials among others.
\end{itemize}

\subsection{Super-smooth mixture deconvolution and Gaussian case}
\label{sec:super_smooth} 
\subsubsection{Description of the distributions}
We consider in this paragraph the statistically hard situation of the general family of mixing distribution $\varphi$ with an exponential decrease of the Fourier transform. More precisely, we assume that the spectral density $\sigma$ of $\varphi$ satisfies:
\eq
\tag{$\cH^{supersmooth}_{\alpha,\beta}$}
\exists j \in \mathds{N}^\star \quad s.t. \quad \cF[\varphi](t)= \sigma(t) =  e^{-\alpha \|t\|_{j}^\beta}\quad \forall t\in\mathds R^d, \alpha>0, \beta>0.
\qe
where for any $j\in \mathds{N}^\star$, $\|.\|_j$ denotes the $\ell^j$-norm.  The assumption ($\cH^{supersmooth}_{\alpha,\beta}$) includes obviously the Gaussian distribution but also many other distributions as suggested by the list of examples displayed below (among others). 

\paragraph{$\bullet$ The multivariate Cauchy distribution}For a dispersion parameter $\alpha$,  $\varphi$ is defined by:
$$ \varphi(x) = \frac{\Gamma(\frac{d+1}{2})}{\Gamma(\frac{1}{2}) \pi^{\frac{d}{2}} \sqrt{\alpha} \lbrace 1+ \alpha^{-1} \| x \|_2^2  \rbrace^{\frac{d+1}{2}}} \quad \forall x\in \mathds{R}^d \qquad \text{and} \qquad \sigma(t) = e^{ - \sqrt{\alpha} \| t\|_2}, \quad \forall t\in \mathds{R}^d.$$

\paragraph{$\bullet$ The tensor product of univariate Cauchy distribution} An alternative example is:
$$ \varphi(x) = \frac{1}{\pi^d} \prod_{j=1}^d \left( \frac{\alpha}{x_j^2 +\alpha^2}\right) \quad \forall x= (x_1 \dots x_d)^T \in \mathds{R}^d \quad \text{and} \quad \sigma(t) = e^{-\alpha \| t\|_1}, \quad \forall t\in \mathds{R}^d.$$

\paragraph{$\bullet$ The multivariate  Gaussian distribution} A standard benchmark study of the Gaussian law:
 $$ \varphi: x \longmapsto (2\pi)^{-d/2} e^{-\|x\|^2/2} \qquad \text{and} \qquad  \sigma(t) = e^{-\frac{\|t\|_2^2}{2}}, \quad \forall t\in \bR^d.$$

\subsubsection{General recovery result}
In the situations covered by assumption ($\cH^{supersmooth}_{\alpha,\beta}$), we shall observe that $\|t\|_j \leq d^{1/j} \|t\|_{\infty}$ and we verify that:

$$
\inf_{\|t\|_{\infty} \leq 4 m } \sigma(t) = e^{-\alpha (4  d^{1/j} m)^\beta}.
$$
In that case, we obtain that 
$$
\rho_n \mathcal{C}_m(\varphi,\lambda_{\tau}) \lesssim K^2 2^{3d/2} \times \frac{e^{\alpha (4 d^{1/j} m)^\beta}}{\sqrt{n}}.
$$
A straightforward application of Theorem \ref{thm:main} leads to the following result.

\begin{prop}
\label{prop:supersmooth}
Assume that $\varphi$ satisfies  $(\cH^{supersmooth}_{\alpha,\beta})$. Let $m \gtrsim \sqrt{K} d^{3/2} \Delta_+^{-1}$. 
Let $\hat \mu_n$ be the BLASSO estimator given by~\eqref{eq:blasso} with $\kappa=\kappa_n$ chosen as in Proposition \ref{prop:up}, then up to some universal constants (independent from $n,d,K$ and $m$):
\begin{itemize}
\item[$i)$] Far region and negative part: if $\epsilon = \mathcal{O}\left( \frac{1}{m d}\right)$, then:
$$
\E\left[ \hat\mu_n^-(\bR^d)\right] \lesssim K^2 2^{3d/2}  \times \frac{e^{\alpha (4 d^{1/j} m)^\beta}}{\sqrt{n}} 
\quad \text{and} \quad 
\E\left[\hat\mu_n^+(\bbF(\epsilon))\right] \lesssim K^2 d^3 2^{3d/2} \times \frac{e^{\alpha (4 d^{1/j} m)^\beta}}{\sqrt{n}}.
$$
\item[$ii)$] Near region (spike detection): a couple of constants $(c,\mathcal{C})$ exists such that
$$\forall A \subset \bR^d, \quad  \E[\hat\mu_n^+(A)] > c \times d^3 2^{3d/2} K^2 \times \frac{e^{\alpha (4 d^{1/j} m)^\beta}}{\sqrt{n}} \quad\Longrightarrow \quad \min_{k\in[K]} \inf_{t\in A}\| t-t_k\|_2^2 \leq \frac{1}{\mathcal{C} d^3 m^2}.  $$
\item[$iii)$] Near region (weight reconstruction): for any $k \in [K]$:
$$
\E \left[ |a_k^0-\hat{\mu}_n(\bbN_k(\epsilon))|\right] \lesssim  2^{3d/2} K^2 \times \frac{e^{\alpha (4 d^{1/j} m)^\beta}}{\sqrt{n}}. $$
\end{itemize}
\end{prop}

According to the results displayed in Proposition \ref{prop:supersmooth}, our estimation procedure $\hat \mu_n$ leads to a consistent estimation as soon as $m$ is chosen as
$$
m = \left( \frac{\delta \log n}{\alpha}\right)^{1/\beta} \frac{1}{4 d^{1/j}} \quad \mathrm{with} \quad \delta \in \left]0,\frac{1}{2} \right[.
$$

In such a case, 
$$ \max \left( \E\left[ \hat\mu_n^-(\bR^d)\right] \ , \  \E\left[\hat\mu_n^+(\bbF(\epsilon))\right] \ , \  \E \left[ |a_k^0-\hat{\mu}_n(\bbN_k(\epsilon))|\right] \right) \lesssim n^{-\frac{1}{2} +\delta}, $$
and every set $A$ such that  $\E[\hat\mu_n^+(A)] \gtrsim n^{-\frac{1}{2} +\delta}$ is at least at a logarithmic distance ($\mathcal{O}(m^{-2})$) of a true spike.  

We observe that as it is commonly observed in severely-ill conditioned inverse problems, we can  expect only logarithmic rates of convergence. This logarithmic limitation in the super-smooth situation  has been intensively discussed in the literature and we refer among others to~\cite{Fan}.  To make the situation more explicit, we illustrate it in the  Gaussian mixture model.

\subsubsection{Multivariate Gaussian mixtures}
As a specific case of super-smooth distribution with $\beta=j=2$ and $\alpha=1/2$, Proposition \ref{prop:supersmooth} holds and we obtain that if $m \gtrsim  \sqrt{K} d^{3/2}\Delta_+^{-1} $ and if $\epsilon = \mathcal{O}( \frac{1}{m d})$, then the weights of the far region and of the negative parts are upper bounded by:
\begin{equation}\label{eq:negative_far}
\E\left[ \hat\mu_n^-(\bR^d)\right] \lesssim  K^2 2^{3d/2}  \times\frac{e^{8 d m^2}}{\sqrt{n}} \quad \text{and} \quad 
\E\left[\hat\mu_n^+(\bbF(\epsilon))\right] \lesssim    K^2 d^3 2^{3d/2}  \times \frac{e^{8dm^2}}{\sqrt{n}} .
\end{equation}
Similarly,  a couple of constants $(c,\mathcal{C})$ exists such that:
\begin{equation}
\label{eq:near}
\forall A \subset \bR^d, \quad  \E[\hat\mu_n^+(A)] > c   d^3 2^{3d/2} K^2 \times \frac{e^{8dm^2}}{\sqrt{n}} \quad\Longrightarrow \quad \min_{k\in[K]} \inf_{t\in A}\| t-t_k\|_2^2 \leq \frac{1}{\mathcal{C} d^3 m^2},
\end{equation}
whereas the weights recovery is ensured by the following inequality:
 for any $k \in [K]$:
$$
\E \left[ |a_k^0-\hat{\mu}_n(\bbN_k(\epsilon))|\right] \lesssim  2^{3d/2} K^2 \times \frac{e^{8dm^2}}{\sqrt{n}}. $$
\begin{itemize}
\item{\textbf{Quantitative considerations}}
When the dimension $d$ is kept fixed (as the number of components $K$ and the minimal value for the spacings between the spikes $\Delta$), the statistical ability of the BLASSO estimator $\hat \mu_n$ is driven by the term $  e^{8dm^2} / \sqrt{n}$. In particular, this sequence converges to $0$ provided that the following condition holds:
\begin{equation} 
e^{8dm^2} \ll \sqrt{n} \quad \mathrm{i.e.} \quad m = \mathcal{O}\left( \sqrt{\frac{\log(n)}{d}}\right) \quad \mathrm{and} \quad m \longrightarrow + \infty \quad \mathrm{as} \ n\longrightarrow + \infty.
\label{eq:condgauss}
\end{equation}
In other words, the maximal admissible value for $m$ is $\sqrt{\frac{\log(n)}{16d}}$.  
In particular, if we consider $m = \sqrt{\frac{ \delta}{16} \frac{\log(n)}{d}}$ for $ \delta$ small enough, we  observe that 
$$
\E\left[ \hat\mu_n^-(\bR^d)\right] + \E\left[\hat\mu_n^+(\bbF(\epsilon_n))\right]  \lesssim  \sqrt{n}^{\delta-1}.
$$
The counterpart of this admissible size for $m$ is a slow rate for $\epsilon_n$:
 $$\epsilon_n = \mathcal{O}\left(\frac{1}{md} \right)= \frac{ \delta^{-1/2}}{\sqrt{d \log n}},$$
Said differently, the size of the near regions recovered with an almost parametric rate $n^{-1/2}$ are of the order $(d \log(n))^{-1/2}$. \\
\item{\textbf{Nature of the results}}
Item $i)$ of Proposition \ref{prop:supersmooth} and Equation~\eqref{eq:negative_far} both indicate that the mass set by $\hat \mu_n$  on the negative part and on the far region tends to $0$ as the sample size $n$ grows under Condition (\ref{eq:condgauss}). Our estimator is consistent: the mass allowed on the near region will be close to $1$ as soon as $n$ is large enough. At this step, we stress that the parameter $m$ plays the role of an accuracy index: if $m$ is constant, the mass of the near region converges to $1$ at a parametric rate... but this near region is in this case not really informative. On the opposite hand, if $m$ is close to the limit admissible value expressed in (\ref{eq:condgauss}), Item $ii)$ of Proposition \ref{prop:supersmooth} and Equation~\eqref{eq:near} translate the fact that
the near region is  close to the support of the measure $\mu^0$ but the convergence of the associated mass will be quite slow. \\
\item{\textbf{Case of dimension 1 and number of spikes detection}}
According to Item $ii)$ of Proposition \ref{prop:supersmooth} and Equation~\eqref{eq:near}, any set with a sufficiently large mass is close to a true spike $(a_k^0,t_k)$ for some $k\in [K]$.  
We stress that in the specific situation where $d=1$, $\hat \mu_n$ is a discrete measure (see Theorem~\ref{thm:sol_discrete}), namely
$$ \hat \mu_n = \sum_{\hat{t} \in \hat{S}}  \hat a_{\hat{t}} \delta_{\hat{t}}. $$
In such a case, we get from Proposition \ref{prop:supersmooth} that if a reconstructed spike $(\hat a_{\hat{t}},\hat{t})$ is large enough, it is in some sense close to a true spike. 
More formally, if $m = \mathcal{O}( \sqrt{ \delta \log(n)})$ and $\hat{t} \in \hat{S}$, then
$$ \hat a_{\hat{t}}\gtrsim K^2 n^{-1/2+ \delta} \Longrightarrow  \inf_{k\in [K]} |\hat{t} - t_k | \lesssim \frac{1}{\sqrt{ \delta \log(n)}}.$$
In particular, the BLASSO estimator $\hat \mu_n$ provides a lower bound on the number of true spikes. Once again, the value of $m$ is critical in such a case. In particular, according to~\eqref{eq:condgauss}, we cannot expect more than a logarithmic precision. \\
\item{\textbf{Importance of the mixture parameters}} 
It is also interesting to pay attention to the effect of the number of components $K$, the size of the minimal spacing $\Delta$ and of the dimension $d$ on the statistical accuracy of our method.
 In the Gaussian case, the rate is of the order $K^2 C^d e^{8 d m^2} n^{-1/2}$ but an important effect is hidden in the constraint brought by  Theorem \ref{theo:main_certificate}:
$$
m \gtrsim  \sqrt{K} d^{3/2} \Delta_+^{-1}.
$$

In particular, the behavior of our estimator  is seriously damaged in the Gaussian situation when $(\Delta^{-1} \vee K \vee d) \rightarrow + \infty$ since in that case, taking the minimal value of $m$ satisfying the previous contraint, we obtain a rate of the order 

$$
e^{d^{4} K \Delta^{-2}} n^{-1/2}.
$$

We observe that $d$, $K$ and $\Delta^{-1}$ cannot increase faster than a power of $\log(n)$: $d^{4} K \Delta^{-2} \ll \log(n)$.%$d^{17/3} K \Delta^{-2} \ll \log(n)$. 
 We will observe  in Section \ref{sec:ordinary_smooth} that a such hard constraint disappears in more favorable cases with smaller degrees of ill-posedness.\\
\item{\textbf{Position of our result on Gaussian mixture models}}

To conclude this discussion, we would like to recall that the BLASSO  estimator $\hat{\mu}_n$
 depends on $m$. This parameter plays the role of a precision filter and only provides a quantification of the performances of our method. This is one of the main differences with the classical super-resolution theory where in general $m$ is fixed and constrained by the experiment.   We should point out that many works have studied  statistical estimation in Gaussian mixture models with a semi-parametric point of view (see, \textit{e.g.}~\cite{vanderVaartmix},~\cite{BMV_2006}). 
  These investigations are often reduced to the   two-component case (K=2): we refer to~\cite{BV_2014},~\cite{KH15} or~\cite{Gadat_Kahn_Marteau_Maugis}  among others. The general case ($K \in \mathds{N}^*$) has been for instance addressed in~\cite{Cathy} using a model selection point of view: the selection of $K$ is achieved through the minimization of a criterion penalized by the number of components. We also  refer to~\cite{spade} where a Lasso-type estimator is built for mixture model using a discretization of the possible values of $t_k$.  However, this last approach is limited by some constraints on the Gram matrix involved in the model that do not allow to consider situations where $\Delta$ is small: in~\cite{spade}, the minimal separation between two spikes has to satisfy $\Delta \geq \Delta_0 >0$, \textit{i.e.} has to be lower bounded by a positive constant $\Delta_0$, which depends on the mixing distribution $\varphi$.
  We emphasize that in our work, we only need an upper bound on $K$ and a lower bound on $\Delta$ or at least to assume that these quantities are fixed w.r.t. $n$. According to Proposition \ref{prop:supersmooth}, our constraint  expressed on these parameters already allows to cover a large number of interesting situations. 
\end{itemize}

\subsection{Ordinary smooth distributions}\label{sec:ordinary_smooth}
\paragraph{General result} Ordinary smooth distributions investigated in this section are described through a polynomial decrease of their Fourier transform. The corresponding deconvolution problem is then said to be mildly ill-posed. In this section, we assume that the density $\varphi$ satisfies
\eq
\label{eq:Hsmooth}
\tag{$\cH^{smooth}_\beta$}
\cF[\varphi]= \sigma \quad \text{and} \quad
\|x\|_{2}^{-\beta}  \lesssim \sigma(x) \lesssim \|x\|_{2}^{-\beta} \quad \text{when} \quad \|x\|_2 \rightarrow + \infty.
\qe

We refer to~\cite{Fan} and the references therein for an extended description of the class of distributions involved by $(\cH^{smooth}_\beta)$ and some statistical consequences in the situation of standard non-parametric deconvolution (see also the end of this section for two examples).
For our purpose, it is straightforward to verify that
$$
\inf_{\|t\|_{\infty} \leq 4 m } \sigma(t) \leq \inf_{\|t\|_{2} \leq 4 m \sqrt{d} } \sigma(t) \lesssim [\sqrt{d} m]^{-\beta}.
$$
In that case, we obtain that 
$$
\rho_n \mathcal{C}_m(\varphi,\lambda_{\tau}) \lesssim \frac{K^2 2^{3d/2} m^{\beta} d^{\beta/2}}{\sqrt{n}}.
$$
We then deduce the following result (which is a direct application of Theorem \ref{thm:main}).
\begin{prop}
\label{prop:ordinary}
Assume that $\varphi$ is ordinary smooth and satisfies  $(\cH^{smooth}_\beta)$. Consider $m \gtrsim  \sqrt{K} d^{3/2} \Delta_+^{-1}$. 
Let $\hat \mu_n$ be the BLASSO estimator given by~\eqref{eq:blasso} with $\kappa=\kappa_n$ chosen as in Proposition \ref{prop:up}, then up to universal constants (independent from $n,d,K$ and $m$):
\begin{itemize}
\item[$i)$] Far region and negative part: if $\epsilon = \mathcal{O}\left(\frac{1}{m d}\right)$, then:
$$
\E\left[ \hat\mu_n^-(\bR^d)\right] \lesssim  K^2 2^{3d/2}d^{\beta/2} \times  \frac{m^\beta }{\sqrt{n}} \quad \text{and} \quad 
\E\left[\hat\mu_n^+(\bbF(\epsilon))\right] \lesssim K^2 2^{3d/2} d^{3+\beta/2}  \times \frac{ m^\beta }{\sqrt{n}}.
$$

\item[$ii)$] Near region (spike detection): a couple of constants $(c,\mathcal{C})$ exists such that
$$\forall A \subset \bR^d, \quad  \E[\hat\mu_n^+(A)] > c\,K^2 2^{3d/2} d^{3+\beta/2} \times \frac{m^\beta}{\sqrt{n}} \quad\Longrightarrow \quad \min_{k\in[K]} \inf_{t\in A}\| t-t_k\|_2^2 \leq \frac{1}{\mathcal{C} d^3 m^2}.  $$

\item[$iii)$] Near region (weight reconstruction): for any $k \in [K]$:
$$
\E \left[ |a_k^0-\hat{\mu}_n(\bbN_k(\epsilon))|\right] \lesssim  K^2 2^{3d/2} d^{\beta/2}   \times \frac{m^\beta}{\sqrt{n}}. $$
\end{itemize}
\end{prop}

The proof of this proposition is omitted, and we only comment on the consequences of this result for ordinary smooth mixtures. Provided $K,d$ and $\Delta$ are bounded (or fixed), we obtain a consistent estimation with the BLASSO estimator $\hat{\mu}_n$ when $m$ is chosen such that
$$m_n = n^{\delta}  \quad \text{with} \quad \delta < \frac{1}{2 \beta}  \quad \text{as} \quad n \rightarrow + \infty.
$$ 
In such a case, $\epsilon_n = \mathcal{O}(d^{-1} n^{-\delta})$.
%\footnote{
%\seb{ Je ne sais pas comment formuler cela, mais probablement que la valeur optimale pour $\delta$ est $1/(4\beta)$ ou quelque chose dans ce gout la mais je ne sais pas donner un sens \`a cette optimalit\'e.}}
Now, if $K \vee d \vee \Delta^{-1}$ is allowed to grow towards $+\infty$, setting $m \sim \sqrt{K} d^{3/2} \Delta_+^{-1}$ (the minimal value satisfying the constraint (\ref{eq:variance})) leads to a bound of order
$$
%\cancel{\frac{K^{2+\beta/2} \Delta^{-2 \beta} 2^{3d/2}\clem{d^{3/2+\beta}}}{\sqrt{n}}.}
 \max \left( \E\left[ \hat\mu_n^-(\bR^d)\right] \ , \  \E\left[\hat\mu_n^+(\bbF(\epsilon))\right] \ , \  \E \left[ |a_k^0-\hat{\mu}_n(\bbN_k(\epsilon))|\right] \right) \lesssim \  \frac{2^{3d/2} K^{2+\beta/2} \Delta_+^{-\beta} d^{2\beta+3}}{\sqrt{n}}.
$$
In particular, the maximal order for the dimension is $\mathcal{O}(\log(n))$ as $n\rightarrow +\infty$. In the same way, the minimal size of spacings to permit a consistent estimation should  not be smaller than $n^{-1/(2\beta)}$.
%\footnote{\seb{ J'ai un petit soucis: le spacing doit etre plus grand que $n^{-1/(4 \beta)}$ mais le $\epsilon_n$ lui est de taille $n^{-1/2\beta}$. Je trouve cela curieux.}}
In particular, this indicates that a polynomial accuracy is possible  (see e.g. Item $ii)$ of Proposition~\ref{prop:ordinary}). This emphasized the strong role played by the mixture density $\varphi$ in our analysis. We present below two specific examples of ordinary smooth mixture density.

\paragraph{Multivariate Laplace distributions}
In such a case:
$$
\sigma(x) = \frac{2}{2+\|x\|_2^2}.
$$
We obtain here an ordinary smooth density with $\beta=2$.
The minimal spacing for a discoverable spike is therefore of the order $n^{-1/4}$ while the constraint on the dimension is not affected by the value of $\beta$. Concerning the number of components $K$, its value should not exceed $n^{1/6}$ and the smallest size of the window $\epsilon_n$ is $n^{-1/4}$.
\paragraph{Tensor product of Laplace distributions}
%}

Another interesting case is the situation where~$\varphi$ is given by a tensor product of  standard Laplace univariate distributions:
$$\varphi(x) = \frac{1}{2^d} e^{-\sum_{j=1}^d | x_j|} \quad \mathrm{and} \quad  \mathcal{F}[\varphi](x) := \sigma(x) = \prod_{j=1}^d \frac{1}{1+x_j^2} \quad \forall x\in \bR^d.$$
In that case, $\beta = 2d$ and the previous comments   apply:
the maximal value of $m$ is $n^{1/4d}$ with an optimal size of the window of the order $n^{-1/(4d)}$ whereas   $d$ should be at least of order $\mathcal{O}(\log(n))$.

\section{Proof of the Main Results}%\label{sec:appendix}
\label{sec:proofs}

\subsection{Analysis of the Bregman divergence}
\label{s:propup}

This paragraph is devoted to the statistical analysis of the Bregman divergence whose definition is recalled below:
$$
	\DP(\hat\mu_n,\mu^0) :=  \| \hat\mu_n \|_1 - \| \mu^0\|_1 - \int_{\bR^d} \Pm\dd(\hat\mu_n - \mu^0)\geq0\,.
$$

\begin{proof}[Proof of Proposition \ref{prop:up}]

According to the definition of $\hat \mu_n$ as the minimum of our variational criterion (see Equation~\eqref{eq:muhat}), we know that:
$$
\| L\hat f_n - L\circ\Phi \hat\mu_n \|_\bbL^2 + \kappa \| \hat\mu_n \|_1  \leq \| L\hat f_n - L\circ\Phi \mu^0 \|_\bbL^2 + \kappa \| \mu^0 \|_1.
$$
\underline{Proof of $i)$.}
With our notation $\vn = L \hat f_n - L \circ \Phi \mu^0$ introduced in Section \ref{sec:bregman}, we deduce that:
$$
\| L\hat f_n - L\circ\Phi \hat\mu_n \|_\bbL^2 + \kappa \| \hat\mu_n \|_1  \leq \| \vn \|_\bbL^2 + \kappa \| \mu^0 \|_1.
$$
Using now $\Pm$ obtained in Theorem \ref{theo:main_certificate}, we deduce that
\begin{equation}\label{eq:utile_plus_loin}
\| L\hat f_n - L\circ\Phi \hat\mu_n \|_\bbL^2 + \kappa \left[ \| \hat\mu_n \|_1 - \| \mu^0\|_1 - \int_{\bR^d} \Pm\dd(\hat\mu_n - \mu^0) \right] + \kappa  \int_{\bR^d} \Pm\dd(\hat\mu_n - \mu^0)  \leq \| \vn \|_\bbL^2.
\end{equation}
Hence, we deduce the following upper bound on the Bregman divergence:% (see Equation~\eqref{eq:def_bregman}):
\begin{equation}\label{eq:bregman_upper_bound}
\| L\hat f_n - L\circ\Phi \hat\mu_n \|_\bbL^2 + \kappa \DP(\hat\mu_n,\mu^0) + \kappa  \int_{\bR^d} \Pm\dd(\hat\mu_n - \mu)  \leq \| \vn \|_\bbL^2.
\end{equation}
According to Theorem \ref{theo:main_certificate}, $\Pm= \Phi \com$ for some $\com \in \bbL$. In particular, we get% as in Equation~(4) of~\cite{super_mix_supp}:
\begin{align}
\int_{\bR^d} \Pm \dd(\hat\mu_n - \mu^0)
& =  \langle \Pm , \hat\mu_n - \mu^0 \rangle_{L^2(\bR^d)} , \nonumber \\
& =  \langle \Phi \com, \hat\mu_n - \mu^0 \rangle_{L^2(\bR^d)} ,  \nonumber \\
& =  \langle  \com, \Phi(\hat\mu_n - \mu^0) \rangle_{L^2(\bR^d)} ,  \nonumber 
\end{align}
where the last equality comes from the self-adjoint property of $\Phi$ in $L^2(\bR^d)$.
The reproducing kernel relationship yields:
\begin{align}
\int_{\bR^d} \Pm \dd(\hat\mu_n - \mu^0)
& =  \int_{\bR^d} \langle \com, \lambda(t-.)\rangle_\bbL \Phi(\hat\mu_n - \mu^0)(t)dt, \nonumber \\
& =  \langle \com, L\circ\Phi(\hat\mu_n - \mu^0) \rangle_\bbL, \nonumber \\
& =  \langle  \com, L\circ\Phi\hat\mu_n - L\hat f_n + \vn \rangle_{\bbL} .
\label{eq:up3}
\end{align}
Gathering (\ref{eq:bregman_upper_bound}) and (\ref{eq:up3}), we deduce that:
$$
\| L\hat f_n - L\circ\Phi \hat\mu_n \|_\bbL^2  + \kappa \DP(\hat\mu_n,\mu^0) + \kappa \langle  \com, L\circ\Phi\hat\mu_n - L\hat f_n \rangle_{\bbL}  + \kappa \langle  \com,   \vn \rangle_{\bbL}  \leq \| \vn \|_\bbL^2\,.$$
Using now a straightforward computation with 
$\|.\|_{\bbL}$, we conclude that:

% , \\
%& \Leftrightarrow & \left\| L\hat f_n - L\circ\Phi \hat\mu_n - \frac{\kappa}{2} \com \right\|_\bbL^2  + \kappa D(\hat\mu_n,\mu^0) -  \frac{\kappa^2}{4} \|  \com \|_\bbL^2  + \kappa \langle  \com,   \vn \rangle_{\bbL}  \leq \| \vn \|_\bbL^2,\\
$$ \left\| L \hat f_n - L\circ\Phi \hat\mu_n - \frac{\kappa}{2} \com \right\|_\bbL^2+  \kappa \DP(\hat\mu_n,\mu^0)  \leq \left\| \vn - \frac \kappa 2 \com \right\|_\bbL^2 .$$
Since the first term of the left hand side is positive, the previous inequality leads to: 
\begin{equation}
\DP(\hat\mu_n,\mu^0) \leq \frac{3}{2\kappa } \| \vn \|_\bbL^2 + \frac{3\kappa }{4} \| \com \|_\bbL^2,
\label{eq:up4}
\end{equation}
where we have used $\|a+b\|_\bbL^2\leq 1.5 \|a\|_\bbL^2 + 3 \|b\|_\bbL^2$ with $a=\vn$ and $b=-\kappa  \com/2$ %a triangle inequality 
for the right hand side. 
%Let $(\rho_n)_{n\in \mathds{N}}$ a sequence such that\footnote{Same kind of discussion for a control with high probability} $\E[ \| \vn \|^2] \leq \rho_n^2$ for all $n\in \mathds{N}^\star$. 
We now consider a sequence $(\rho_n)_{n\in \mathds{N}^*}$ such that $\E[\| \vn \|_\bbL^2] \leq \rho_n^2$ for all $n\in \mathds{N}^*$ and we choose: 
\begin{equation*}
\label{eq:choixl}
\kappa = \sqrt{2} \rho_n  / \| \com \|_\bbL.
\end{equation*}
Then  we deduce from (\ref{eq:up4}) that:
\begin{equation}
\E[ \DP(\hat\mu_n,\mu^0)] \leq  \frac{3 \sqrt{2}}{2} \rho_n \times \| \com \|_\bbL.
\label{eq:up5}
\end{equation}

\noindent
\underline{Proof of $ii)$.}
We now derive an upper bound on $\|\com\|_{\bbL}$.  Recall that according to (\ref{eq:H0}) and in particular (\ref{eq:equaL}) we have:
$$ \| g \|^2_\bbL = \int_{\bR^d} \frac{ |\mathcal{F}[g](t)|^2}{\mathcal{F}[\lambda](t)} dt \quad \forall g\in \bbL.$$
Since $\varphi$ is symmetric and $\Phi^\star = \Phi$, we have according to Theorem \ref{theo:main_certificate} that:
\begin{eqnarray}
\| \Pm \|^2_{2} &=& \| \Phi \com \|^2_{2},  \nonumber \\
& = &\int_{\bR^d} |\mathcal{F}[\varphi](t)|^2 |\mathcal{F}[\com](t)|^2 dt, \nonumber \\
& = & \int_{\bR^d} |\mathcal{F}[\varphi](t)|^2 \mathcal{F}[\lambda](t)\times \frac{  |\mathcal{F}[\com](t)|^2}{\mathcal{F}[\lambda](t)} dt, \nonumber \\
&\geq & \inf_{\|t\|_\infty\leq 4m} \left\lbrace |\mathcal{F}[\varphi](t)|^2 \mathcal{F}[\lambda](t) \right\rbrace  \| \com \|_\bbL^2.
\label{eq:inter1}
\end{eqnarray}
Indeed, $iii)$ of Theorem \ref{theo:main_certificate} entails that the support of the Fourier transform of $\Pm$ is contained in $[-4m,4m]^d$. This embedding, together with (\ref{eq:H0_intro}) entails: 
$$ \mathrm{Supp}(\mathcal{F}[\Pm]) \subset [-4m,4m]^d,$$
which provides the last inequality. The inequality (\ref{eq:inter1}) can be rewritten as:
\begin{equation}
\| \com \|_\bbL^2 \leq \frac{\| \Pm \|_2^2}{\inf_{\|t\|_\infty\leq 4m} \left\lbrace |\mathcal{F}[\varphi](t)|^2 \mathcal{F}[\lambda](t) \right\rbrace}.
\label{eq:majoc0}
\end{equation}
We use (\ref{eq:up5}), (\ref{eq:majoc0}) and observe that $|\mathcal{F}[\varphi]|=\sigma$ to conclude the proof.
\end{proof}

\subsection{Near and Far region estimations}
\label{s:proof_main}
In this paragraph, we provide the main result of the paper that establishes the statistical accuracy of our BLASSO estimation.

\begin{proof}[Proof of Theorem \ref{thm:main}]~\\
\underline{Proof of $i)$}
In a first time, we provide a lower bound on the Bregman divergence. This bound takes advantage on the properties of the dual certificate associated to Theorem \ref{theo:main_certificate}. First remark that
\begin{eqnarray*}
\int \Pm \dd(\hat \mu_n - \mu^0) &=& \int \Pm \dd\hat\mu_n - \sum_{k=1}^K a_k^0 \Pm(t_k)\\
&\leq & \|\hat \mu_n\|_1 - \|\mu^0\|_1,
\end{eqnarray*}
since $\Pm(t_k)=1$ for all $k$. This inequality yields the positiveness of the Bregman divergence:
$$\DP(\hat \mu_n,\mu^0) \geq 0.$$
Now, using similar arguments and the Borel's decomposition $\hat \mu_n = \hat\mu_n^+ - \hat\mu_n^-$, we obtain
\begin{eqnarray}
\DP(\hat\mu_n,\mu^0) & = & \| \hat\mu_n \|_1 - \| \mu^0\|_1 - \int \Pm \dd \hat\mu_n  + \int \Pm \dd \mu^0, \nonumber \\
& = & \| \hat\mu_n \|_1  - \int \Pm \dd \hat\mu_n, \nonumber\\
& = & \int \dd \hat\mu_n^+ + \int \dd \hat\mu_n^- - \int \Pm \dd \hat\mu_n^+ +\int \Pm \dd \hat\mu_n^-, \nonumber \\
& = & \int (1- \Pm) \dd \hat\mu_n^+ +  \int (1+ \Pm) \dd \hat\mu_n^-. \nonumber
%\label{eq:BorelBreg}
\end{eqnarray}
%since $\| \mu^0\|_1 = \int p^0 \dd \mu^0$. Thus, we obtain that 
%$$ D(\hat\mu_n, \mu^0) = \| \hat\mu_n \|_1 - \int p^0 \dd \hat\mu_n = \sum_{j=1}^{\hat K} \hat a_j [1- p^0(\hat t_j)]. $$ 
%In particular, we have proved the following lemma
%
%\begin{lemma}
%\label{lem:bregman}
%Let $p^0$ be a dual certificate satisfying the requirements of Lemma \ref{lem:dual}. Then 
%\begin{eqnarray*}
%a) & & D(\hat \mu_n,\mu^0) \geq 0 \\
%b) & & D(\hat \mu_n,\mu^0) = \|\hat \mu_n\|_1 - \int p^0 \dd \hat \mu_n = \sum_{j=1}^{\hat K} \hat a_j [1- p^0(\hat t_j)].
%\end{eqnarray*}
%\textcolor{red}{on a toujours}
%\[
%D(\hat \mu_n,\mu^0) = \|\hat \mu_n\|_1 - \int p^0 \hat \mu_n = \int_{\bR^d}(1-\tau(t)p^0(t))\dd|\hat\mu_n|(t)
%\]
%\textcolor{red}{using Borel's decomposition $\hat \mu_n=\tau |\hat \mu_n|$ $\hat \mu_n$-a.e. where $\tau$ is a measurable function with values in $[-1,1]$ defined as the {\it sign} function of $\hat \mu_n$. }
%\end{lemma}
 Proposition \ref{prop:up} then implies that:
\begin{equation} 
\E \left[\int (1- \Pm) \dd \hat\mu_n^+ +  \int (1+ \Pm) \dd \hat\mu_n^- \right] \leq \frac{3\sqrt{2}}{2} \rho_n \mathcal{C}_m(\varphi,\lambda). 
\label{eq:toto}
\end{equation}
\underline{Weight of the negative part.}
Since the dual certificate $\Pm$ is always positive, we have
\begin{equation}
\mu_n^{-}(\bR^d) = 
\int \dd \hat\mu_n^- \leq \int (1+\Pm)\dd \hat\mu_n^- \leq \frac{3 \sqrt{2}}{2} \rho_n \mathcal{C}_m(\varphi,\lambda).
\label{eq:inter2}
\end{equation}
Moreover, according to item $ii)$ of Theorem \ref{theo:main_certificate}, %for all $t \in \bbF(\epsilon)$
%$$ 0 \leq \Pm(t) \leq 1 - \eta \frac{v}{d^4} \quad \Rightarrow \quad \eta \frac{v}{d^4} \leq 1- \Pm(t) .$$
$$
 1- \Pm(t) \geq \etaa \frac{\upsilon^2}{d^3} \quad \forall t\in\bbF(\epsilon).
$$
Therefore, we obtain that:
\begin{equation}
\hat\mu_n^+(\bbF(\epsilon)) \leq \frac{d^3}{\etaa \upsilon^2} \int_{\bbF(\epsilon)} (1- \Pm) \dd \hat\mu_n^+ \leq \frac{d^3}{\etaa \upsilon^2} \int (1- \Pm) \dd \hat\mu_n^+.
\label{eq:inter3}
\end{equation}
Finally, the first part of $i)$ of Theorem \ref{thm:main} is a direct consequence of (\ref{eq:toto})-(\ref{eq:inter3}). 

\noindent
\underline{Weight of the far region.}
We consider $\gamma$ such that $d^3 \geq \etaa \upsilon^2$ and we know that in the far region:
$$
(1-\Pm) \mathds{1}_{\bbF(\epsilon)} \geq \frac{\gamma \upsilon^2}{d^3} \mathds{1}_{\bbF(\epsilon)}.
$$
Thus,
\begin{eqnarray*}
\DP(\hat\mu_n,\mu^0) &=& \int (1- \Pm) \dd \hat\mu_n^+ +  \int (1+ \Pm) \dd \hat\mu_n^-\\
&\geq& \int_{\bbF(\epsilon)}  \frac{\etaa \upsilon^2}{d^3} \dd \hat\mu_n^+ +  \int_{\bbF(\epsilon)} 1 \dd \hat\mu_n^-\\
&\geq&  \frac{\etaa \upsilon^2}{d^3} \int_{\bbF(\epsilon)} \dd \hat\mu_n^+ +  \int_{\bbF(\epsilon)} \dd \hat\mu_n^-\\
& \geq & \frac{\etaa \upsilon^2}{d^3} \left(  \int_{\bbF(\epsilon)} \dd \hat\mu_n^+ +   \int_{\bbF(\epsilon)}\dd \hat\mu_n^- \right)\\
& \geq & \frac{\etaa \upsilon^2}{d^3} |\hat\mu_n|(\bbF(\epsilon)).
\end{eqnarray*}
We then conclude, using the previous expectation upper bound, that:
$$
\E[|\hat\mu_n|(\bbF(\epsilon))] \leq \frac{d^3}{\etaa \upsilon^2} \frac{3\sqrt{2}}{2}\rho_n\mathcal{C}_m(\varphi,\lambda).
$$

\noindent
\underline{Proof of $ii)$.} Thanks to Theorem \ref{theo:main_certificate}, we have:
$$ 1-\Pm(t) \geq \left[\mathcal{C} m^2 \min_{k\in[K]} \| t -t_k \|_2^2 \wedge \frac{\etaa \upsilon^2}{d^3} \right] \quad \forall t\in \bR^d.$$
Then, for any subset $A\subset \bR^d$, 
\begin{eqnarray} 
\DP(\hat \mu_n,\mu^0) & \geq & \int (1-\Pm) \dd \mu_n^+ \nonumber\\
& \geq & \int_A (1-\Pm) \dd \mu_n^+, \nonumber \\
& \geq & \left[\mathcal{C} m^2 \min_{t\in A} \min_{k \in [K]} \| t-t_k \|^2 \wedge \frac{\etaa \upsilon^2}{d^3} \right]\hat\mu_n^+(A). \label{eq:mun+}
\end{eqnarray}
Equations~\eqref{eq:toto} and~\eqref{eq:mun+} lead to: 
$$ \left[\mathcal{C} m^2 \min_{t\in A} \min_{k \in [K]} \| t-t_k \|^2 \wedge \frac{\etaa \upsilon^2}{d^3} \right] \E[\hat\mu_n^+(A)] \leq \frac{3\sqrt{2}}{2} \rho_n \mathcal{C}_m(\varphi,\lambda).$$
Then, 
$$
\E[\hat\mu_n^+(A)] \geq \frac{3\sqrt{2}}{2} \rho_n \mathcal{C}_m(\varphi,\lambda) \frac{d^3}{\etaa \upsilon^2}
\Rightarrow
\min_{t\in A} \min_{k \in [K]} \| t-t_k \|_2^2 \leq \frac{\etaa \upsilon^2}{d^3 m^2 \mathcal{C}}.
$$

%$$ 1-\Pm(t) \geq \mathcal{C}m^2 \min_{k=1\dots, K} \| t -t_k \|^2 \quad \forall t\in \bR^d.$$
%Then, for any subset $A\subset \bR^d$, 
%\begin{eqnarray} 
%\int (1-\Pm) \dd \mu_n^+ 
%& \geq & \int_A (1-\Pm) \dd \mu_n^+, \nonumber \\
%& \geq & \mathcal{C} m^2 \int_A \min_{k \in [K]} \| t-t_k \|^2 \dd \hat\mu_n^+(t), \nonumber\\ 
%& \geq & \mathcal{C} m^2 \min_{t\in A} \min_{k \in [K]} \| t-t_k \|^2 \hat\mu_n(A). 
%\label{eq:mun+}
%\end{eqnarray}
%Equation (\ref{eq:toto}) together with (\ref{eq:mun+}) leads to 
%$$ \mathcal{C} m^2 \min_{t\in A} \min_{k \in [K]} \| t-t_k \|^2 \hat\mu_n(A) \leq \frac{3\sqrt{2}}{2} \rho_n \mathcal{C}_m(\varphi,\lambda).$$
%The conclusion immediately follows this inequality provided 
%$$ \hat \mu_n^+(A) > \frac{3\sqrt{2}}{2} \rho_n \mathcal{C}_m(\varphi,\lambda).$$
\noindent
\underline{Proof of $iii)$.} The idea of this proof is close to the one of~\cite[Theorem 2.1]{Azais_DeCastro_Gamboa_15}.
We consider the function $\Qmk$ given by Corollary \ref{coro:Qmk} that interpolates $1$ at $t_k$ and $0$ on the other points of the support of $\mu^0$. From the construction of $\Qmk$, we have that:
$$
a_k^0 = \int \Qmk \dd \mu^0.
$$
 We then use the decomposition:
\begin{eqnarray}
|a_k^0-\hat{\mu}_n(\bbN_k(\epsilon))|& = & |a_k^0 - \int \Qmk \dd \hat{\mu}_n + \int \Qmk \dd \hat{\mu}_n - \int_{\bbN_k(\epsilon)} \dd \hat{\mu}_n|
\nonumber\\
& \leq &\underbrace{|\int \Qmk \dd(\mu^0-\hat{\mu}_n)|}_{:=A}
+ \underbrace{\int_{\bbN_{k}(\epsilon)} |\Qmk-1| \dd|\hat{\mu}_n|}_{:=B} \nonumber\\
& &
+ \underbrace{\int_{\bbN(\epsilon) \setminus \bbN_{k}(\epsilon)}|\Qmk|\dd |\hat{\mu}_n|}_{:=C}
+ \underbrace{\int_{\bbF(\epsilon)} |\Qmk| \dd|\hat{\mu}_n|}_{:=D}. \label{eq:decomposition}
\end{eqnarray}
\underline{Study of $B+C+D$.}
On the set $\bbF(\epsilon)$, we use that $\Qmk \leq 1-\gamma \frac{\upsilon^2}{d^3}$ so that: 
$$
D \leq \int_{\bbF(\epsilon)} (1-\gamma \frac{\upsilon^2}{d^3}) \dd|\hat{\mu}_n| \leq \diamond \int_{\bbF(\epsilon)} (1-\Qmk) \dd   |\hat{\mu}_n|\quad \text{where} \quad 
\diamond = 
\frac{\left(1-\gamma \frac{\upsilon^2}{d^3}\right)}{\gamma \frac{\upsilon^2}{d^3}}.
$$ For the term $C$, we use the upper bound satisfied by $\Qmk$ in $\bigcup_{i \neq k} \bbN_i(\epsilon)$ and obtain that:
\begin{eqnarray*}
\int_{\bbN(\epsilon) \setminus \bbN_k(\epsilon)} |\Qmk| \dd |\hat{\mu}_n| &\leq &\widetilde{\mathcal{C}} m^2 \int_{\bbN(\epsilon) \setminus \bbN_k(\epsilon)} 
\min_{i \neq k} \|t-t_i\|_2^2  \dd |\hat{\mu}_n|(t)\\
& \leq &
\frac{\widetilde{\mathcal{C}}}{\mathcal{C}}
\int_{\bbN(\epsilon) \setminus \bbN_k(\epsilon)} (1-\Pm) \dd |\hat{\mu}_n|\,.
\end{eqnarray*}
Finally, for $B$, we use that on the set $\bbN_k(\epsilon)$, we have $|\Qmk -1|\leq \widetilde{\mathcal{C}} m^2 \|t-t_k\|_2^2$. Therefore, we have:
$$
B \leq \frac{\widetilde{\mathcal{C}}}{\mathcal{C}}
\int_{ \bbN_k(\epsilon)} (1-\Pm) \dd |\hat{\mu}_n|.
$$
We then conclude that:
\begin{eqnarray}
B+C+D &\leq& \left(\frac{\widetilde{\mathcal{C}}}{\mathcal{C}} \vee \diamond \right) \int_{\bR^d} (1-\Pm)(t) \dd |\hat\mu_n|(t)\nonumber \\
& \leq &  \left(\frac{\widetilde{\mathcal{C}}}{\mathcal{C}} \vee \diamond \right) \left[ \int_{\bR^d} (1-\Pm)(t) \dd \hat{\mu}_n^+(t) + \int_{\bR^d} (1+\Pm)(t) \dd \hat{\mu}_n^-(t) \right] \nonumber\\
& \leq & \left(\frac{\widetilde{\mathcal{C}}}{\mathcal{C}} \vee \diamond \right) \DP(\hat \mu_n,\mu^0). \label{eq:upBCD}
\end{eqnarray}
\underline{Study of $A$.}
We use that $\Qmk$ may be written as:
$$
\Qmk = \Phi c_{k,m}, \quad \text{where} \quad c_{k,m} \in \bbL.
$$
Since $\Phi$ is self-adjoint in $L^2$, we shall write that:
\begin{align*}
A=|\int \Qmk   \dd (\mu^0-\hat{\mu}_n)| & = |\langle \Qmk,\hat\mu_n-\mu^0\rangle_{L^2}| \\
& = |\langle c_{k,m} ,\Phi(\hat{\mu}_n-\mu^0)\rangle_{L^2}| \\
& = | \langle c_{k,m},L \circ \Phi \hat{\mu}_n - L \hat{f}_n + \vn \rangle_{\bbL}| \\
& \leq \|c_{k,m}\|_{\bbL} [\|L \circ \Phi \hat{\mu}_n - L \hat{f}_n\|_{\bbL} + \|\vn\|_{\bbL}],
\end{align*}
where we used the Cauchy-Schwarz inequality and the triangle inequality in the last line.
We then use~\eqref{eq:utile_plus_loin} and obtain:
$$
\| L\hat f_n - L\circ\Phi \hat\mu_n \|_\bbL^2  + \kappa \DP(\hat\mu_n,\mu^0) + \kappa \langle  \com, L\circ\Phi\hat\mu_n - L\hat f_n \rangle_{\bbL}  + \kappa \langle  \com,   \vn \rangle_{\bbL}  \leq \| \vn \|_\bbL^2\,.$$
Since we have obtained the positiveness of the Bregman divergence, we then conclude that: 
$$
\| L\hat f_n - L\circ\Phi \hat\mu_n \|_\bbL^2  + \kappa \langle  \com, L\circ\Phi\hat\mu_n - L\hat f_n \rangle_{\bbL}    \leq \| \vn \|_\bbL^2 - \kappa \langle  \com,   \vn \rangle_{\bbL}\,.$$
The Cauchy-Schwarz inequality yields:
$$
\| L\hat f_n - L\circ\Phi \hat\mu_n \|_\bbL^2  - \kappa \|\com\|_{\bbL} \| L\circ\Phi\hat\mu_n - L\hat f_n \|_{\bbL}    \leq \| \vn \|_\bbL^2 + \kappa \| \com\|_{\bbL} \|   \vn \|_{\bbL}\,.$$
This inequality holds for any value of $\kappa$ and we choose: 
$$
\kappa=\frac{\| L\hat f_n - L\circ\Phi \hat\mu_n \|_\bbL}{2 \|\com\|_{\bbL}}.
$$
Using this value of $\kappa$, we then obtain:
$$
\frac{\| L\hat f_n - L\circ\Phi \hat\mu_n \|_\bbL^2}{2} \leq \|\vn\|_{\bbL}^2+
 \|\vn\|_{\bbL}
\| L\hat f_n - L\circ\Phi \hat\mu_n \|_\bbL \,.
$$
Now, we define  $\Box_n= \| L\hat f_n - L\circ\Phi \hat\mu_n \|_\bbL \|\vn\|_{\bbL}^{-1}$ and remark that:
$$
\frac{\Box_n^2}{2} \leq 1+\Box_n.
$$
This last inequality implies that $\Box_n \leq 1+\sqrt{3}$, which leads to:
$$
\| L\hat f_n - L\circ\Phi \hat\mu_n \|_\bbL \leq (1+\sqrt{3}) \|\vn\|_{\bbL}.
$$
We then come back to $A$ and write that:
\begin{equation}\label{eq:upA}
A \leq (2+\sqrt{3}) \|c_{k,m}\|_{\bbL} \|\vn\|_{\bbL}.
\end{equation}
\noindent
\underline{Final bound.}
We use Equations~\eqref{eq:upA} and~\eqref{eq:upBCD} in the decomposition given in Equation~\eqref{eq:decomposition} and obtain that:
$$
\E \left[ 
|a_k^0-\hat{\mu}_n(\bbN_k(\epsilon))|\right] \lesssim 
\rho_n \left( \|c_{k,m}\|_{\bbL}+\|c_{0,m}\|_{\bbL}\right).
$$
Finally, we conclude the proof using Equation~\eqref{ineq:ckm} and $ii)$ of Proposition \ref{prop:up}:	
%$iv)$ of Corollary \ref{coro:Qmk} and $ii)$ of Theorem \ref{thm:main}:
$$
\E \left[ |a_k^0-\hat{\mu}_n(\bbN_k(\epsilon))| \right] \lesssim 
\rho_n
\frac{K^2 m^{-d/2}}
 {\sqrt{\displaystyle\inf_{\|t\|_\infty\leq 4m} \left\lbrace \sigma^2(t) \mathcal{F}[\lambda](t) \right\rbrace}}.
$$
\end{proof}

%-------------------------------------------------------------------------------
\bibliographystyle{abbrv}
\bibliography{biblio}  
%-------------------------------------------------------------------------------

%-------------------------------------------------------------------------------
\appendix
%-------------------------------------------------------------------------------

\newpage

%\section{Appendix}\label{sec:appendix}

%\noindent This document is a companion paper of \cite{super_mix}, that details the numerical algorithms and the most technical proofs associated to the supermix estimator defined by:
%\eq
%\hat \mu_n \in  \arg\min_{\mu \in \mathcal{M}(\bR^d,\bR)} 
%\left\lbrace \frac12\| L\hat f_n - L\circ\Phi \mu \|_{\bbL}^2 + \kappa \| \mu \|_1 \right\}\,.
%\label{eq:muhat}
%\qe
%We also remind the dual relationship: if we denote by $\hat z_n=L\circ\Phi\hat\mu_n$ and $
%\hat c_n=L\hat f_n-\hat z_n\,$, then:
%\eq
%\label{e:dual_poly}
%\hat\eta_n:=\frac{\Phi \hat c_n}\kappa=\frac1\kappa\Phi(L\hat f_n-z_n)\,.
%\qe
%For the sake of completeness and readability in this supplementary material, Algorithms \ref{sec:sfw-alg:sfw} and \ref{sec:cpgd} that are already given in \cite{super_mix} are also reminded in this document.

\section{Tractable Algorithms for BLASSO Mixture Models}
\label{app:algos}
We sketch  three algorithms to compute {\it approximate or exact} solutions to~\eqref{eq:muhat} following the ``{\it \!off-the-grid\,}'' methodology, {\it e.g.,} \cite{Bhaskar_Tang_Recht12,Tang_Bhaskar_Shah_Recht_13,Duval_Peyre_JFOCM_15,Azais_DeCastro_Gamboa_15,de2017exact}. This methodology searches in a {\it gridless manner} the location of the support points $t_i$ of the mixture distribution $\mu^0$. We will present the following methods:
\begin{itemize}
\item {\it Greedy} methods  provide heuristic and theoretical  results such as ‘‘{\it Sliding Frank Wolfe}\,'' \cite{denoyelle2019sliding} (also known as conditional gradient with ‘‘{\it sliding}'' step) or ‘‘{\it Continuous Orthogonal Matching Pursuit}\,'' \cite{keriven2018sketching,elvira2019does}. We describe these methods in Section \ref{sec:SFW}.
\item We  discuss in Section \ref{sec:CGPD} on {\it Conic Gradient Descent} using {\it particles}. Here,  $\hat{\mu}_n$ is approximated by a cloud of particles that is optimized all along a set of iterations.
\end{itemize}
We emphasize that if Beurling-LASSO has been studied in the past decade, the formulation~\eqref{eq:muhat} has two new important features. First the observation is a sample from a mixing law. Second, the data fidelity term has been tuned to incorporate a low pass filter kernel $\lambda$. The next paragraph carefully introduces these new features into the latter algorithms.

\subsection{Notation for algorithm design solving BLASSO Mixture Models}\label{sec:notBLASSO}
We call that primal and dual convex programs of BLASSO for Mixture Models~\eqref{eq:muhat} are given by Theorem~6, %of \cite{super_mix}, 
and that {\it strong duality} holds, leading to Equation~\eqref{e:dual_poly}. 

%\subsubsection*{Gradient of the data fidelity term}
\paragraph{Gradient of the data fidelity term}
The {\it data fidelity} term defined 
\[
\mathrm C_{\lambda}(\Phi\mu,\hat{f}_n):=\|L\hat f_n-L\circ\Phi\mu \|_{\bbL}^2, \quad  \forall  \mu \in \cM(\bR^d,\bR)\,
\]
can be seen is related to the  real-valued function~$F$ on the space of measures $\cM(\mathds R^d,\mathds R)$ endowed with the total-variation norm $\|\cdot\|_1$, namely:
\[
\forall \mu\in\cM(\mathds R^d,\mathds R)\,,\quad {\mathrm F}(\mu):=\frac{1}{2} C_{\lambda}(\Phi\mu,\hat{f}_n) = \frac12\| L\hat f_n - L\circ\Phi \mu \|_{\bbL}^2\, ,
\]
whose Fréchet differential at point $\mu\in\cM(\mathds R^d,\mathds R)$ in the direction $\nu\in\cM(\mathds R^d,\mathds R)$ is:
\[
\mathrm{dF}(\mu)(\nu):=\int_{\mathds R^d} \Phi (L\circ\Phi\mu- L\hat f_n)\mathrm{d}\nu=\int_{\mathds R^d} \nabla\mathrm{F}(\mu)\,\mathrm{d}\nu.
\]
Thanks to the convolution by $\varphi$ endowed in $\Phi$, the gradient $\nabla F(\mu)$ is given by:
\[
\nabla\mathrm{F}(\mu):=\Phi (L\circ\Phi\mu- L\hat f_n)=\varphi\star\Big[\lambda\star\varphi\star\mu -\frac1n\sum_{i=1}^n \lambda(\,\cdot-{X_i})\Big]\in\cC_0(\bR^d,\bR)\cap L^1(\bR^d)\,.
\]

%\subsubsection*{Dual functions}
\paragraph{Dual functions}
By~\eqref{e:dual_poly}, note that the {\it dual function} $\hat \eta_n$ is such that $\hat \eta_n=-{\nabla\mathrm{F}(\hat\mu_n)}/{\kappa}$. Indeed, for a given $\mu\in\cM(\mathds R^d,\mathds R)$, one may define its dual function by $\eta_\mu$:
\[
\eta_\mu:=-\frac{\nabla\mathrm{F}(\mu)}{\kappa}=\frac1{\kappa}{\Phi (L\hat f_n-L\circ\Phi\mu)}\quad\text{ so that }\quad \nabla\mathrm{F}(\mu)+\kappa = \kappa\big(1-\eta_\mu\big)\,,
\]
and we observe that $\eta_\mu$ corresponds to a {\it residual}\,, which involves the difference between $\hat{f}_n$ and $\Phi \mu$ smoothed by the convolution operator $L$.
Its gradient is given by:  $$\displaystyle\nabla\eta_\mu:=\frac1{\kappa}{\nabla\varphi\star\Big[\lambda\star\varphi\star\mu -\frac1n\sum_{i=1}^n \lambda(\,\cdot-{X_i})\Big]}.$$

%\subsubsection*{Data fidelity term for measures with finite support}
\paragraph{Data fidelity term for measures with finite support}
We pay a specific attention to discrete measures with finite support, namely a finite sum of Dirac masses. Given a number of atoms $N\geq 1$, of weights $a\in\mathds R^N$ and locations $t=(t_1,\ldots,t_N)\in(\mathds R^d)^N$, we denote  by:
\begin{equation}\label{def:FN}
\mu_{a,t}:=\displaystyle\sum_{i=1}^{N}a_i\delta_{t_i}\quad \text{and}\quad \mathrm F_N(a,t):={\mathrm F}(\mu_{a,t})
=\frac12\| L\hat f_n - \sum_{i=1}^{N}a_iL\circ\Phi\delta_{t_i} \|_{\bbL}^2\,.
\end{equation}
By Proposition 5, %of \cite{super_mix},  
$a\longmapsto \mathrm F_N(a,t)$ is a {\it positive semi-definite} quadratic form and 
\[
\mathrm F_N(a,t)= \frac12\|L\hat f_n\|_{\bbL}^2 +\sum_{i=1}^Na_ib_i+\frac12\sum_{i,j=1}^Na_ia_jq_{ij},
\,
\]
where 
$$
b_i = -\frac1n\sum_{k=1}^n\int_{\bR^d}\lambda(x-X_k)\varphi(x-t_i)\dd x \quad \text{and} \quad q_{ij}=\int_{\bR^d\times\bR^d}\lambda(x-y)\varphi(x-t_i)\varphi(y-t_j)\dd x \dd y  .
$$

\subsection{Greedy methods: Sliding Frank-Wolfe / Continuous Orthogonal Matching Pursuit \label{sec:SFW}}
\paragraph{Sliding Frank-Wolfe algorithm (SFW)}
The Frank-Wolfe algorithm is an interesting avenue for solving {\it differentiable convex} programs on {\it weakly compact convex} sets, see \cite{denoyelle2019sliding} and references therein for further details. {\it Stricto sensu}, $(\mathbf{P}_\kappa)$ is convex but not differentiable and the feasible set is convex but not weakly compact. Following~\cite[Lemma 4]{denoyelle2019sliding}, note that $\hat\mu_n$ is a minimizer of~$(\mathbf{P}_\kappa)$ {\it if and only if} $(\|\hat\mu_n\|_1,\hat\mu_n)$ minimizes:
\eq
\inf 
\left\{ \frac12\| L\hat f_n - L\circ\Phi \mu \|_{\bbL}^2 + \kappa\, m\, :\ (m,\mu) \in \mathds R\times \mathcal{M}(\bR^d,\bR)\ \text{s.t.}\ \|\hat\mu_n\|_1\leq m\leq \frac{\| L\hat f_n\|_{\bbL}^2}{2\kappa} \right\}\,,
%\label{eq:blasso}
\notag
\qe
and this latter program is a {differentiable convex} program on {weakly compact convex} set (for the {\it weak}-$\star$ topology). 

Hence we can invoke the {\it Frank-Wolfe} scheme to compute approximate solutions to BLASSO Mixture Models~\eqref{eq:muhat}. Unfortunately, the generated measures $\mu^{(k)}$ along this {\it greedy} algorithm are not very sparse compared to $\hat\mu_n$: each Dirac mass of $\hat\mu_n$ is approximated by a multitude of Dirac masses of $\mu^{(k)}$ with an inexact positions. This is why the improvement of \textit{sliding} the Frank Wolfe algorithm is suggested in \cite{denoyelle2019sliding}:  Equation~\eqref{e:non_conv}  involved in the resolution of the BLASSO allows to move the Dirac masses. Algorithm~\ref{sec:sfw-alg:sfw} is the sliding Frank-Wolfe \cite[Algorithm~2]{denoyelle2019sliding} adapted for the resolution of the BLASSO Mixture Models~\eqref{eq:muhat}.

\pagebreak[3]

Following the analysis of \cite{denoyelle2019sliding}, let us discuss the steps at lines $3$, $4$, $7$ and $8$ of Algorithm~\ref{sec:sfw-alg:sfw}:
\begin{itemize}
\item Line $3$: This step is an {\it optimal} gradient step, with the notations in Section \ref{sec:notBLASSO}:
\[
\eta^{(k)}=\frac1\kappa\Phi(L\hat f_n-L\circ\Phi\hat\mu^{(k)}) =\frac1\kappa\,\varphi\star\Big[\frac1n\sum_{i=1}^n \lambda(\,\cdot-{X_i})-\lambda\star\varphi\star\hat\mu^{(k)}\Big]\,,
\]
Note also that this step is the {\it costly} step of the algorithm since is relies on a {\it black-box} optimizer computing the {\it global} maximum of $|\eta^{(k)}|$. In general, this is done using a grid search and finding a local maxima by gradient descent.
\item Line $4$: the stopping condition implies that $\hat\mu^{(k)}$ is an {\it exact} solution and hence, $\eta^{(k)}$ is the dual function such that Equation~\eqref{e:dual_poly} holds. In this sense, we may say that SFW iteratively construct a {\it dual function} such that~\eqref{e:dual_poly} holds. 
\item Line $7$~\eqref{e:lasso_step}: note that the support is fixed and we are optimizing on the amplitudes $a$. It amounts in solving a standard LASSO, which can be efficiently done using proximal forward-backward schemes such as FISTA for instance.
\item Line $8$: it requires solving a non-convex optimization program~\eqref{e:non_conv}. As mentioned in~\cite{denoyelle2019sliding}, one does not need to exactly solve this program and their main result (namely finite convergence of the algorithm, see below) pertains if $(a^{(k+1)},t^{(k+1)})$, obtained by a gradient step initialized in $(a^{(k+\frac12)},t^{(k+\frac12)})$, diminishes the objective function. This is done by a bounded 
Broyden-Fletcher-Goldfarb-Shanno (BFGS) method in \cite{denoyelle2019sliding}, which is numerically shown to be rapid with few iterations needed.
\end{itemize}

An important feature of SFW is that it can actually lead to {\it exact} solutions. Under a {\it Non-Degeneracy Condition} alike to~\eqref{eq:NDSCB}, the result in \cite[Theorem~3]{denoyelle2019sliding} proves that Algorithm~\ref{sec:sfw-alg:sfw} recovers exactly $\hat \mu_n$ in a finite number of steps. They also show \cite[Proposition 5]{denoyelle2019sliding} that the generated measure sequence $(\mu^{(k)})_k$ converges towards $\hat\mu_n$ for the weak-$\star$ topology.

As mentioned in \cite[Remark 8]{denoyelle2019sliding}, Algorithm~\ref{sec:sfw-alg:sfw} can be adapted to build a {\it positive} measure as follows 
\begin{itemize}
\item the stopping condition $|\eta^{(k)}(t_\star^{(k)})|\leq 1$ becomes $\eta^{(k)}(t_\star^{(k)})\leq 1$;
\item the LASSO is solved on $a\in\mathds R_+^{N^{(k)}+1}$;
\item the step~\eqref{e:non_conv} is solved on $\mathds R_+^{N^{(k)}+1}\times(\mathds R^d)^{N^{(k)}+1}$.
\end{itemize}

\paragraph{Continuous Orthogonal Matching Pursuit (COMP)}
Continuous Orthogonal Matching Pursuit (COMP) \cite{elvira2019does} is another greedy approach that is the Orthogonal Matching Pursuit approximation algorithm  \cite{pati1993orthogonal} adapted in the context of continuous parametric dictionaries. This framework fits ours and COMP can be applied to Mixture Models estimation. Continuous Orthogonal Matching Pursuit (COMP) is an iterative algorithm that add a Dirac mass one at the time, building a sequence of measures $\hat\mu^{(k)}$, but it does not solve the BLASSO Mixture problem~\eqref{eq:muhat} {\it per se}. Nevertheless it builds a sequence of dual functions $\eta^{(k)}=\frac1\kappa\Phi(L\hat f_n-L\circ\Phi\hat\mu^{(k)})$, referred to as the {\it residual} in the framework of COMP. The Dirac mass added to the model is defined as in Line \ref{computeNextPos} of Algorithm~\ref{sec:sfw-alg:sfw} (SFW for BLASSO Mixture Models) but the weights $a$ are updated differently (we referred to \cite[Algorithm 1]{keriven2018sketching}): alternating between {\it hard-thresholding} \cite[Step~3 in Algorithm 1]{keriven2018sketching} and even some {\it sliding}-flavour step \cite[Step 5 in Algorithm 1]{keriven2018sketching}.

%\cite[Algorithm 1]{keriven2018sketching}
%\subsubsection{Moment-Sum-Of-Squares-Hierarchies}

\subsection{Conic Particle Gradient Descent (CPGD) \label{sec:CGPD}}
Conic Particle Gradient Descent \cite{chizat2019sparse} is an alternative promising avenue for solving BLASSO for Mixture Models~\eqref{eq:muhat}. The idea is to discretize a positive measure into a system of particles, \textit{i.e.} a sum of $N$ Dirac masses, by:
\[
\mu_{a,t}=\frac1N\sum_{i=1}^N a_i\delta_{t_i},
\]
with $a_i=r_i^2$. We observe that the objective function involved in the minimization of Equation~\eqref{eq:muhat} is given by:
\[
H(r,t)=F_N((r_i^2),(t_i))+\kappa\sum_{i}r_i^2\quad\text{with } (r,t)= ((r_i),(t_i))\in(\mathds R_+)^N\times  (\mathds R^d)^N,
\]
where $F_N$ is defined in Equation~\eqref{def:FN}. As already emphasized in \cite{chizat2019sparse}, Equation~\eqref{eq:muhat} is a convex program in $\mu$ whereas the parametrization given in $H$ translates this minimization  into a {\it non-convex} differentiable problem in terms of $r$ and $t$. This function $H$ can be seen as an instance of the BLASSO Equation~\eqref{eq:muhat} for the measure $\mu_{a,t}$, namely a convex program that does not depends on the number of Particles~$N$.
All the more, it is possible to run a  gradient descent on positions $t_i\in\mathds R^d$ and weights $r_i>0$ of the~$N$ particles system.
 The crucial ingredient is then to implement a gradient descent on the lifted problem in the {\it Wasserstein} space approximating the {\it Wasserstein gradient flow}. For two step-sizes $\alpha>0$ and $\beta>0$, and for any position $(r,t)$, we define the Riemannian inner product by: $
 \forall (\delta r_1,\delta r_2) \in \mathds{R}_+^2 \quad \forall (\delta t_1,\delta t_2) \in \{\mathds{R}^d\}^2$:
 $$\left\langle (\delta r_1,\delta t_1)(\delta r_2,\delta t_2)\right\rangle_{r,t} :=\frac{\delta r_1 \delta r_2}{\alpha}  + r^2 \frac{ \sum_{i=1}^d(\delta t_1)_i (\delta t_2)_i}{\beta}.
 $$
The  gradient w.r.t. this \textit{conic} metric is given by:
\begin{align*}
\nabla_{r_i}H&=2\alpha r_i(\nabla_{r_i} F(\mu_{a,t})+\kappa)=-2\alpha r_i \kappa(\eta_{\mu_{a,t}}-1)\,,
\\
\nabla_{t_i}H&=-\beta\kappa\nabla\eta_{\mu_{a,t}}\,,
\end{align*}
and the {\it Wasserstein gradient} \cite[Section 2.2]{chizat2019sparse} is $g_\mu(r,t)=(-2\alpha r \kappa(\eta_{\mu}(t)-1),-\beta\kappa\nabla\eta_{\mu}(t))$ for a.e. point $(r,t)\in\Omega$.

Using standard mean-field limits of gradient flows in Wasserstein space (see \textit{e.g.} \cite{Santambrogio}), it is possible to prove that the approximate gradient flows of the $N$-particles system converge towards the gradient flow on the Wasserstein space when $N \longrightarrow +\infty$ (see Theorem~2.6 and Theorem~3.5 of \cite{chizat_bach_2018_nips}). Hence, for a large enough number of particles $N$, it then implies the convergence towards  the global minimizer of $\mu_{a,t} \longmapsto F(\mu_{a,t})$ itself, despite the lack of convexity of the function $(r,t) \longmapsto F_N(r,t)$. We refer to Theorem~3.9 of \cite{chizat2019sparse} that establishes the convergence of the particle gradient descent with a constant step-size under some {\it non-degeneracy} assumptions, \textit{i.e.} the convergence of the CPGD toward $\hat\mu_n$~\eqref{eq:muhat} in {\it Hellinger-Kantorovich} metric, and hence in the {\it weak}-$\star$ sense. Furthermore, Theorem~3.9 provides an exponential convergence rate: CPGD has a complexity scaling as $\log(1/\varepsilon)$ in the desired accuracy $\varepsilon$, instead of $\varepsilon^{-1/2}$ for general accelerated convex methods.

\section{Proofs related to the kernel construction}\label{sec:appendix}

\subsection{Convolution in the RKHS}\label{proof:belongL}
\begin{proof}[Proof of Proposition 4] %of \cite{super_mix}]
Consider $\mathcal{A}: f \longmapsto x \mapsto \int \ell(x-y) f(y) dy$, $\mathcal{A}$ is a self-adjoint operator. We denote  by $(w_k)_{k \geq 1}$ the non-negative eigenvalues of $\mathcal{A}$ and   $(\psi_k)_{k \geq 1}$ the associated eigenvectors. We shall remark that the following equality holds:
$$
\ell(x,y) = \lambda(x-y) = \sum_{k \geq 1} w_k \psi_k(x) \psi_k(y),
$$
while $\bbL$ corresponds to the next Hilbert space
$$
\bbL = \Big\{ f=\sum_{k \geq 1} c_k(f) \psi_k \, : \, \sum_{k \geq 1} \frac{c_k(f)^2}{w_k} < + \infty\Big\} \qquad \text{and} \qquad  <f,g>_{\bbL} = \sum_{k \geq 1} \frac{c_k(f) c_k(g)}{w_k}\,.
$$
We now consider a non-negative measure $\nu$ and we remark that
\begin{align*}
L\nu(x) &= \lambda\star\nu(x)\\
& = \int \lambda(x-y) \nu(y) dy\\
& = \int \sum_{k \geq 1} w_k \psi_k(x) \psi_k(y) \nu(y) dy\\
& = \sum_{k \geq 1} \Big[w_k \int \psi_k(y) \nu(y) dy\Big] \psi_k(x).
\end{align*}
We  observe that the coefficients of $L\nu$ are $c_k(L\nu) = w_k \int \psi_k(y) \nu(y) dy$. We shall remark that
$$
\|L\nu\|_{\bbL}^2  = \sum_{k \geq 1} \frac{w_k^2 
\left[\int \psi_k(y) \nu(y) dy\right]^2}{w_k} 
 = \sum_{k \geq 1} w_k
\left[\int \psi_k(y) \nu(y) dy\right]^2.$$
The Jensen inequality yields
$$
\|L\nu\|_{\bbL}^2 \leq \sum_{k \geq 1} w_k \int \psi_k^2(y) \nu(y) dy = \int \sum_{k \geq 1} w_k \psi_k(y)^2 \nu(y) dy,
$$
where the last equality comes from the Tonelli Theorem.
We then observe that
$$
\|L\nu\|_{\bbL}^2 \leq  \int \ell(y,y) \nu(y) dy =  \lambda(0)\nu(\mathds R^d)< + \infty,
$$
giving the result.
\end{proof}

\subsection{Computation of the data-fidelity terms}
\label{proof:critere}

\begin{proof}[Proof of Proposition 5] %of \cite{super_mix}]
Recall that $\Phi\mu\in\cC_0(\bR^d,\bR)\cap L^1(\bR^d)$. Now, given $f\in\cC_0(\bR^d,\bR)\cap L^1(\bR^d)$, one can consider the measure $\mu$ with signed density function~$f$ and we may define:
\[
\forall f\in\cC_0(\bR^d,\bR)\cap L^1(\bR^d)\,,\quad Lf:=\lambda\star f=\int_{\bR^d}\lambda(\cdot-t)f(t)\dd t\,.
\]
The embedding $L$ allows to compare $\hat f_n$ with $\Phi\mu$ in $\cC_0(\bR^d,\bR)\cap L^1(\bR^d)$.  One has:
\begin{align*}
\|Lf^0&-Lf\|_{\bbL}^2-\|Lf^0\|_{\bbL}^2
%\\&
=-2\langle Lf^0,Lf \rangle_{\bbL}+\|Lf\|_{\bbL}^2\\
&=-2\langle Lf^0,\int_{\bR^d}\ell(\cdot,t) f(t) \dd t \rangle_{\bbL}+\|Lf\|_{\bbL}^2
\\&
=-2\int_{\bR^d}\langle Lf^0,\ell(\cdot,t)\rangle_{\bbL}f(t)\dd t +\|Lf\|_{\bbL}^2\\
&=-2\int_{\bR^d}Lf^0(t)f(t)\dd t +\|Lf\|_{\bbL}^2\\
&=\int_{\bR^d}\Big(-2\int_{\bR^d}\lambda(t-x)f^0(x)\dd x \Big)f(t)\dd t%\\ 
%&\quad
+\int_{\bR^d\times\bR^d}\lambda(x-y)f(x)f(y)\dd x \dd y\,.
\end{align*}
Replacing $f^0$, which is unknown, by the empirical measure $\hat f_n$ in the previous equation leads to the following criterion:
\begin{align}
\notag
\mathrm C_{\lambda}(f):=\int_{\bR^d}\big[-\frac2n\sum_{i=1}^n\lambda(t-X_i)\big]f(t)\dd t+\int_{\bR^d\times\bR^d}\lambda(x-y)f(x)f(y)\dd x \dd y\, .
%\\
%\label{eq:critRKHS}
\notag%{$\mathrm C_{\lambda}$}
%\mathrm C(\mu)=\mathrm C_{\lambda}(\mu):=\int_{\bR^d}\big(-\frac2n\sum_{j=1}^n\lambda(t-X_j)\big)\dd\Phi\mu(t)+\int_{\bR^d\times\bR^d}\lambda(x-y)\dd(\Phi\mu)(x)\dd(\Phi\mu)(y)\,.
\end{align}
In particular, for all $\mu \in \cM(\bR^d,\bR)$  note that $\Phi\mu\in\cC_0(\bR^d,\bR)\cap L^1(\bR^d)$ and introduce the criterion:
\[
\mathrm C_{\lambda}(\Phi\mu)=\int_{\bR^d}\big[-\frac2n\sum_{i=1}^n\lambda(t-X_i)\big](\Phi\mu)(t)\dd t+\int_{\bR^d\times\bR^d}\lambda(x-y)(\Phi\mu)(x)(\Phi\mu)(y)\dd x\dd  y\,,
\]
which will be investigated in this paper. Note that it holds 
\begin{align*}
&\|L\hat f_n-L\circ\Phi\mu\|_{\bbL}^2-\|L\hat f_n\|_{\bbL}^2\\
&=
\int_{\bR^d}\big[-\frac2n\sum_{i=1}^n\lambda(t-X_i)\big](\Phi\mu)(t)\dd t+\int_{\bR^d\times\bR^d}\lambda(x-y)(\Phi\mu)(x)(\Phi\mu)(y)\dd x\dd y
\,,
\end{align*}
as claimed. 
\end{proof}

\section{Perfect recovery properties - Theorem~2 and Theorem~3} %of \cite{super_mix}}
\subsection{Perfect recovery}\label{app-PerfRecov}

%This paragraph is dedicated to the proof of the perfect recovery property under~\eqref{eq:H0} and~\eqref{eq:H0_intro}:
%\eq
%\label{eq:H0}\text{The function} \ \varphi \ \text{is a {\it bounded continuous symmetric function of positive definite type}} \tag{$\cH_0$} \qe
%and 
%\eq
%\label{eq:H0_intro}
%\tag{$\cH_\infty$}
%\varphi=\cF^{-1}[\sigma]\text{, } \sigma(\omega)=\sigma(-\omega)\text{ a.e. with }\mathrm{Supp}(\sigma)=\bR^d: \forall \omega \in \bR^d \quad \sigma(\omega)>0.
%\qe

\begin{proof}[Proof of Theorem~2] %of \cite{super_mix}]

Remark first that under~\eqref{eq:H0} and~\eqref{eq:H0_intro}, the RKHS, denoted by $\bbH$, generated by the kernel $h(.,.)=\varphi(.-.)$ is dense in $\cC_0(\bR^d,\bR)$ with respect to the uniform norm, see \cite[Proposition 5.6]{caponnetto2008universal} for instance. 
Furthermore, using  \cite[Proposition 2]{sriperumbudur2011universality}, %$c_0$-universality is equivalent to the fact 
we can show that~\eqref{eq:H0_intro} implies that the embedding $\Phi$  is injective onto $\bbH$. This means that we have identifiability of $\mu$ from the knowledge of~$\Phi \mu$. More precisely, denote $f^0:=\Phi \mu^0$, we deduce that if it holds $\|f^0-\Phi \mu \|_{\bbH}^2=0$ then one has $\mu=\mu^0$. 
\end{proof}

\subsection{Perfect recovery with a dual certificate}
\label{s:proof_Certificate_eta}

%This paragraph is dedicated to the proof of Theorem~ 3, 
%%of \cite{super_mix}, 
%which entails the perfect recovery property under a less restrictive assumption on $\varphi$, given by~\eqref{eq:H1_intro}:
%\eq
%\label{eq:H1_intro}
%\tag{$\cH_\eta$}
%\varphi=\cF^{-1}[\sigma]\text{, } \sigma(\omega)=\sigma(-\omega)\text{ a.e. with } [-\eta,\eta]^d  \subset \mathrm{Supp}(\sigma).
%\qe
%\noindent

\begin{proof}[Proof of Theorem~ 3] %of \cite{super_mix}]
Let 
$$\hat\mu \in \arg\min_{\mu\in\cM(\bR^d,\bR)\ \text{:}\ \Phi \mu = f^0}\|\mu\|_1.$$
\noindent
\underline{Step 1: Support inclusion.}
We observe that both $\hat\mu$ and $\mu^0$ belong to $\cM(f^0)$
so that $\Phi \hat\mu =\Phi \mu^0$. Hence, considering the Fourier transform on both sides and using $\cF[ \Phi(\mu)]=\cF[\varphi]\cF[ \mu]$, we get that $\sigma\cF(\hat\mu)=\sigma\cF(\mu^0)$ which is equivalent to $\cF(\hat\mu)=\cF(\mu^0)$ on the support of $\sigma$. Now, Assumption~\eqref{eq:H1_intro} yields:
\eq
\label{eq:equality_Fourier}
(\cF(\hat\mu)-\cF(\mu^0))\1_{[-\eta,\eta]^d}=0\,.
\qe
%Note that the Fourier transform $q_\eta:=\cF(p_\eta)$ is a continuous function with compact support (by Fourier inversion Theorem~it is the Fourier transform of an element of $L^1(\bR^d)$ and hence a continuous function by Riemann-Lebesgue lemma). Hence $q_\eta$ is a Schwartz function. Recall that the Fourier transform is an automorphism on the Schwartz function space, therefore $p_\eta$ is a Schwartz function as well. Now, observe that for all $\mu\in\cM(\bR^d,\bR)$
Denote by $q_\eta:=\cF(\mathcal{P}_\eta)$ the Fourier transform of $\mathcal{P}_\eta$. By assumption, the support of $q_\eta$ is included in $[-\eta,\eta]^d$ and from~\eqref{eq:equality_Fourier} we get that:
\[
\int_{\bR^d}q_\eta\cF(\hat\mu)=\int_{\bR^d}q_\eta\cF(\mu^0)\,.
\]
Since $\mathcal{P}_\eta\in L^1(\bR^d)$, the Riemann-Lebesgue lemma shows that $q_\eta$ is continuous. Recall also that $q_\eta$ has a compact support so we deduce that $q_\eta\in L^1(\bR^d)$. By Fourier inversion theorem, we have
\[
\int_{\bR^d}\cF(q_\eta)\dd\hat\mu=
\int_{\bR^d}q_\eta\cF(\hat\mu)=\int_{\bR^d}q_\eta\cF(\mu^0)=
\int_{\bR^d}\cF(q_\eta)\dd\mu^0,
\]
namely
\[
\int_{\bR^d}\mathcal{P}_\eta\dd\hat\mu=\int_{\bR^d}\mathcal{P}_\eta\dd \mu^0\,.
\]
Remark that $\mathcal{P}_{\eta}$ satisfies
\[
\int_{\bR^d}\mathcal{P}_\eta\dd\mu^0=\|\mu^0\|_1\,,
\]
and the Hölder inequality leads to
\[
\int_{\bR^d}\mathcal{P}_\eta\dd\hat\mu\leq \|\mathcal{P}_{\eta}\|_{\infty} \|\hat\mu\|_1 = \|\hat\mu\|_1 \,.
\]
From the definition of $\hat\mu$, one also has $\|\hat\mu\|_1\leq \|\mu^0\|_1$. Putting everything together, we deduce that
\[
\|\hat\mu\|_1=\int_{\bR^d}\mathcal{P}_\eta\dd\hat\mu=\|\mu^0\|_1\,.
\]
Since $\mathcal{P}_\eta$ is continuous and strictly lower than one outside of the support of $\mu^0$, we deduce from the above equality that the support of $\hat\mu$ is included in 
%$S^0=\{t_1,\ldots,t_K\}$, 
the support of~$\mu^0$:
$$
\mathrm{Supp}(\hat{\mu}) \subset \mathrm{Supp}(\mu^0)=S^0\,.
$$
\underline{Step 2: Identifiability and conclusion.}
 We prove that $\{\varphi(\cdot-t_1),\ldots,\varphi(\cdot-t_K)\}$ spans a vector subspace of $\cC_0(\bR^d,\bR)\cap L^1(\bR^d)$ of dimension $K$.
 This proof is standard and relies on a Vandermonde argument. We assume first that $K$ coefficients $x_1,\ldots,x_K\in\bR$ exist such that:
\[
\sum_{k=1}^K
x_k\varphi(\cdot-t_k)=0\,.
\]
Applying the Fourier transform and using $\varphi=\cF^{-1}[\sigma]$, we deduce that:
\[
 \sigma(u) \sum_{k=1}^K x_ke^{\imath u^\top t_k} =0\,,\quad \forall u\in\bR^d\,.
\]
Since $\sigma$ is nonzero, there exists an open set $\Omega\subseteq\bR^d$ such that $\sigma>0$ on $\Omega$. We deduce that:
\[
\sum_{k=1}^K x_ke^{\imath u^\top t_k}=0\,,\quad \forall u\in\Omega\,.
\]
Now, we can choose some points $u_j$ in $\Omega$ so that the Fourier matrix with entries $(e^{\imath u_j^\top t_k^0})_{kj}$ is invertible. It implies that $x_k=0$ and  $\{\varphi(\cdot-t_1),\ldots,\varphi(\cdot-t_K)\}$ spans a subspace of dimension~$K$. 

We now conclude the proof. We know from Step 1 that
$$
\hat{\mu} = \sum_{k=1}^K x_k \delta_{t_k}.
$$
Since $\hat{\mu}$ and $\mu^0$ belong to $\cM(f^0)$, then
$$
\sum_{k=1}^K x_k \varphi(.-t_k) = \sum_{k=1}^K a_k^0 \varphi(.-t_k),
$$
which in turn implies that $x_k = a_k^0$ for all $k \in \{1,\ldots,K\}$, namely $\hat{\mu}=\mu^0$.
\end{proof}

\section{Primal-Dual problems and duality gap}
\label{proof:sol_discrete}

%This paragraph is dedicated to the proof of the equality between the optimal primal value~\eqref{eq:blasso} defined as:
%\eq
%\inf 
%\left\{ \frac12\| L\hat f_n - L\circ\Phi \mu \|_{\bbL}^2 + \kappa \| \mu \|_1\, :\ \mu \in \mathcal{M}(\bR^d,\bR) \right\}\,,
%\label{eq:blasso}
%\tag{$\mathbf{P}_\kappa$}
%\qe
%and the optimal dual value~\eqref{eq:blasso_dual} defined by:
%\eq
%  \sup_{c \in \bbL } \inf_{(\mu,z) \in \mathcal{M}(\bR^d,\bR) \times \bbL} \cL(\mu,z,c)\, ,
%  \label{eq:blasso_dual}
%  \tag{$\mathbf{P}^*_\kappa$}
%  \qe 
%   which implies the no duality gap result. 

\begin{proof}[Proof of Theorem~6] %of \cite{super_mix}]
We consider some primal variables $\mu\in\cM(\bR^d,\bR)$ and $z\in\bbL$ and introduce the dual variable $c\in\bbL$. The  Lagrangian is given in Equation~\eqref{eq:lagrangian}
\eq
\label{eq:lagrangian}
\cL(\mu,z,c):=
\frac12\| L\hat f_n - z \|_{\bbL}^2 + \kappa \| \mu \|_1
-\langle c,L\circ\Phi \mu-z\rangle_{\bbL}
\,,
\qe
 and we consider the dual problem~\eqref{eq:blasso_dual}.
 
 \medskip

\noindent
\underline{Proof of $i)$.}
The existence of some solutions $\hat \mu_n$ to the primal problem~\eqref{eq:blasso}
is obtained with the help of a standard argument: we prove that the primal objective function is a proper lower semi-continuous (for the weak-* topology) convex function on the Banach space~$ \mathcal{M}(\bR^d,\bR)$. 

We now consider the ``invariant property'' related to the solutions of~\eqref{eq:blasso}. The norm~$\|\cdot\|_\bbL$ satisfies
\eq
\label{eq:strict_coercive}
\forall a,b\in\bbL\,,\quad \frac{\|a\|_\bbL^2+\|b\|_\bbL^2}2-\Big\|\frac{a+b}2\Big\|_\bbL^2=\frac{\|a-b\|_\bbL^2}4\,.
\qe
Now consider two primal solutions $\mu_1$ and $\mu_2$ and define $\tilde\mu=(\mu_1+\mu_2)/2$. Using~\eqref{eq:strict_coercive} and the triangle inequality for $\|.\|_1$, one has:
\begin{align*}
\frac12\| L\hat f_n - L\circ\Phi \tilde\mu \|_{\bbL}^2 + \kappa \| \tilde\mu \|_1
&
\leq\frac12\| L\hat f_n - L\circ\Phi \tilde\mu \|_{\bbL}^2 + \kappa \frac{\| \mu_1 \|_1+\| \mu_2 \|_1}2\\
&
\leq \frac{\frac12\| L\hat f_n - L\circ\Phi \mu_1 \|_{\bbL}^2 + \kappa \| \mu_1 \|_1}2\\
&\quad +\frac{\frac12\| L\hat f_n - L\circ\Phi \mu_2 \|_{\bbL}^2 + \kappa \| \mu_2 \|_1}2\\
&\quad -\frac18\|L\circ\Phi \mu_1-L\circ\Phi \mu_2\|_\bbL^2\,.
\end{align*}
But, remind that:
\[
\frac12\| L\hat f_n - L\circ\Phi \mu_1 \|_{\bbL}^2 + \kappa \| \mu_1 \|_1=\frac12\| L\hat f_n - L\circ\Phi \mu_2 \|_{\bbL}^2 + \kappa \| \mu_2 \|_1
=\min\left\{ \frac12\| L\hat f_n - L\circ\Phi \mu \|_{\bbL}^2 + \kappa \| \mu \|_1\right\}.
\]
We then conclude that 
 $\tilde\mu$ is also a solution to the primal problem and that $L\circ\Phi \mu_1=L\circ\Phi \mu_2$. We can repeat this argument for any pair of primal solutions so that the quantity $\hat z_n:=L\circ\Phi \hat \mu_n$ is uniquely defined and does not depend on the choice of the primal solution point $\hat\mu_n$. It also implies that $\hat m_n:=\|\hat\mu_n\|_1$ is uniquely defined (does not depend on the choice of the primal solution point).
 
  \medskip

\noindent
\underline{Proof of $ii)$.}
We shall write the dual program~\eqref{eq:blasso_dual}
as follows: consider dual variable $c$ and write:
\[
\inf_{\mu,z}\cL(\mu,z,c)=\inf_{\mu,z}
\Big\{
\underbrace{\frac12\| L\hat f_n - z \|_{\bbL}^2+\langle c,z\rangle_{\bbL}}
_{\raisebox{.5pt}{\textcircled{\raisebox{-.9pt} {1}}}}
+
\underbrace{ \kappa \| \mu \|_1-\langle c,L\circ\Phi \mu\rangle_{\bbL}}
_{\raisebox{.5pt}{\textcircled{\raisebox{-.9pt} {2}}}}
\Big\}\,,
\]
and the previous infimum appears to be splitted in terms of the influence of $z$ and $\mu$.
Optimizing in $z$ the first term \raisebox{.5pt}{\textcircled{\raisebox{-.9pt} {1}}} leads to  $z=L \hat{f}_n-c$ so that:
\eq
\label{e:yoyoyo}
\inf_{z}\raisebox{.5pt}{\textcircled{\raisebox{-.9pt} {1}}}=
\langle c, L\hat f_n\rangle_{\bbL} - \frac{1}{2} \| c\|_{\bbL}^2=
\frac12\big(\| L\hat f_n\|_{\bbL}^2-\| L\hat f_n-c\|_{\bbL}^2\big)\,.
\qe
The second term \raisebox{.5pt}{\textcircled{\raisebox{-.9pt} {2}}} is more intricate. %First, introduce the dual operator $(L\circ\Phi )^\star$ from $ \bbL$ onto $L^\infty(\bR^d) $ by
%\[
%\forall t\in\bbL\,,\quad (L\circ\Phi )^\star t=\Phi(L^\star( t))=(t\star\lambda)\star\varphi\,.
%\]
%Now, observe that
Observe that:
\begin{align*}
\int_{\bR^d}\int_{\bR^d}\Big| c(s)\varphi(s-u)\mathrm d\mu(u)\Big|\dd s
& \leq \int_{\bR^d}\int_{\bR^d}\| c\|_\infty\varphi(s-u)\dd |\mu|(u)\mathrm d s\,,\\
& =\| c\|_\infty\|\varphi\|_1\|\mu\|_1=\| c\|_\infty\|\mu\|_1<\infty\,,
\end{align*}
and the Fubini yields:
\begin{align}
\langle c,(L\circ\Phi )\mu\rangle_{\bbL}
& = \langle c(\cdot),\int_{\bR^d}\lambda(\cdot-s)(\Phi \mu)(s)\dd s\rangle_{\bbL}%\,,\notag%\\
%&
 = \int_{\bR^d}\langle c(\cdot),\lambda(\cdot-s)\rangle_{\bbL}(\Phi \mu)(s)\dd s\,,\notag\\
& = \int_{\bR^d} c(s)(\Phi \mu)(s)\dd s%\,,\notag%\\
%& 
= \int_{\bR^d} c(s)\Big(\int_{\bR^d}\varphi(s-u)\dd\mu(u)\Big)\dd s\,,\notag\\
& = \int_{\bR^d} \Big(\int_{\bR^d} c(s)\varphi(u-s)\dd s\Big)\dd\mu(u)\,,\notag\\
&=\int_{\bR^d} \Phi c\,\dd\mu\,.
\label{eq:Fubini}
\end{align}
We deduce that:
\[
\raisebox{.5pt}{\textcircled{\raisebox{-.9pt} {2}}}=\kappa \| \mu \|_1-\int_{\bR^d}\Phi c\,\mathrm d\mu\,.
\] 
We  use the $L^1-L^\infty$  H\"older inequality, namely $\int_{\bR^d}\Phi c\,\mathrm d\mu\leq\|\Phi c\|_\infty\|\mu\|_1$, which yields:
\[
\inf_{\mu}\raisebox{.5pt}{\textcircled{\raisebox{-.9pt} {2}}} \geq \inf_{\mu} [\kappa -\|\Phi c\|_\infty] \|\mu\|_1 = 
[\kappa -\|\Phi c\|_\infty] 
\inf_{\mu} \|\mu\|_1 \,.
\]
Hence, we conclude that:
\eq
\label{eq:term2}
\inf_{\mu}\raisebox{.5pt}{\textcircled{\raisebox{-.9pt} {2}}} 
=-\bI_{\{\|\Phi c\|_\infty\leq \kappa\}}(c)\,,
\qe
where $\bI_{\{\|\Phi c\|_\infty\leq \kappa\}}(c)$ is the constraint $\|\Phi c\|_\infty\leq \kappa$, namely it is $0$ if $c$ such that $\|\Phi c\|_\infty\leq \kappa$ and $\infty$ otherwise. Finally, we obtain that for a fixed dual variable $c$:
\[
\inf_{\mu,z}\cL(\mu,z,c)=\frac12\Big(\| L\hat f_n\|_{\bbL}^2-\| L\hat f_n-c\|_{\bbL}^2\Big)-\bI_{\{\|\Phi c\|_\infty\leq \kappa\}}(c)\,.
\]
Hence, the dual problem~\eqref{eq:blasso_dual} shall be written as
\begin{align*}
\eqref{eq:blasso_dual} \Longleftrightarrow
\sup_{c}\inf_{\mu,z}\cL(\mu,z,c)&=\sup_{c}\Big\{\frac12\Big(\| L\hat f_n\|_{\bbL}^2-\| L\hat f_n-c\|_{\bbL}^2\Big)-\bI_{\{\|\Phi c\|_\infty\leq \kappa\}}(c)\Big\}\,,\\
&=\frac{\| L\hat f_n\|_{\bbL}^2}2-\inf_{c\,:\  \|\Phi c\|_\infty\leq \kappa\ }
\Big\{\frac12\| L\hat f_n-c\|_{\bbL}^2\Big\}\,.
\end{align*}
Here again the dual objective function is lower semi-continuous and coercive on the Hilbert space $\bbL$ so a minimizer $\hat c$ exists. Again, Inequality~\eqref{eq:strict_coercive} implies the uniqueness of $\hat c$. 

To prove that there is no duality gap, we use the Slater condition: we remark that a feasible point $(c^\circ)$ exists in the interior of the constrained set $\{\|\Phi c\|_\infty\leq \kappa\}$. Now, the {\it generalized Slater condition} shall be used (see \textit{e.g.} \cite{rockafellar1974conjugate}). Indeed, given any nonzero $c\in\bbL\subseteq\cC_0(\bR^d,\bR)$, note that the convolution operator satisfies $\|\Phi c\|_\infty\leq\| c\|_\infty$. Hence, we set $c^\circ=\kappa c/(2\|c\|_\infty)$ and these points are in the interior of the constrained set. The generalized Slater condition implies that strong duality holds, and there is no duality gap (i.e., strong duality holds):
$$
\eqref{eq:blasso} =~\eqref{eq:blasso_dual}.
$$ 
Furthermore, note that $\hat z_n:=L\circ\Phi \hat \mu_n$ and $z_n=L \hat{f}_n-\hat c$ from~\eqref{e:yoyoyo} and strong duality, we deduce that 
\[
\hat c=L \hat{f}_n-L\circ\Phi \hat \mu_n\,.
\]

 \medskip

\noindent
\underline{Proof of $iii)$.}
We consider the unique $\hat c$ solution to 
\[
\hat c=\arg\min_{c,\in\bbL\,:\  \|\Phi c\|_\infty\leq \kappa\ }
\Big\{\frac12\| L\hat f_n-c\|_{\bbL}^2\Big\}\,,
\]
and the strong duality implies that:
\[
0=\kappa \| \hat \mu \|_1-\langle \hat c,L\circ\Phi \hat \mu\rangle_{\bbL}
=\kappa \| \hat \mu \|_1-\int_{\bR^d}\Phi \hat c\,\dd\hat \mu\,.
\]
Since $\Phi \hat c$ is continuous, we verify, using the argument of  Lemma~A.1  in \cite{DeCastro_Gamboa_12}, that: 
\[
\mathrm{Supp}(\hat \mu)\subseteq\Big\{x\in\bR^d\ :\ \big|\Phi \hat c\big|(x)=\kappa\Big\}\,,
\]
where we recall that $\Phi \hat c\in L^\infty(\bR^d)$ is such that its supremum norm is less than $\kappa$.

\medskip

\noindent
\underline{Proof of $iv)$.}
The last point is a consequence of the Schwartz-Paley-Wiener Theorem~(see \textit{e.g.} Theorem~XVI, chapter VII in \cite[Page 272]{schwartz1957theorie}). Indeed, note that $\Phi\hat c$ is a continuous function whose inverse Fourier transform has a support included in the support of $\sigma\times \Lambda$. By assumption, this latter is bounded and one may apply the {\it Schwartz-Paley-Wiener Theorem}: we deduce that~$\Phi\hat c$ can be extended to complex values $\mathds{C}^d$ into an analytic entire function of exponential type. In particular, $\Phi \hat c\pm\kappa$ has isolated zeros one the real line, which concludes the proof.
\end{proof}

%\yo{
\section{Support stability}
\label{s:stability}

\begin{proof}[Proof of  Theorem~11]

We follow the ideas of \cite{Duval_Peyre_JFOCM_15} for the proof of Theorem~11. 
%of \cite{super_mix}.
 Consider the convex program
\eq
\label{eq:P0m}
\tag{$\mathbf P_0$}
\inf\Big\{\|\mu\|_{1}\ :\ \mu\in\mathcal M(\mathds R^d;\mathds R)\ \text{s.t.}\ L\circ\Phi\mu=L\circ\Phi\mu^0\Big\}
\qe
whose Lagrangian expression is, for all $\mu\in\mathcal M(\mathds R^d,\mathds R)$, $c\in\mathds L$, %$\rho\in\mathds R$,
\begin{align*}
\mathcal L(\mu,c)&=\|\mu\|_1+\langle c,L\circ\Phi(\mu^0-\mu)\rangle_{\mathds L}\,,\\
& = \|\mu\|_1+\langle c,L\circ\Phi\mu^0\rangle_{\mathds L}-\int\Phi c\,\dd\mu\,,\\
& = \|\mu\|_1-\int\Phi c\,\dd\mu+\int\Phi c\,\dd\mu^0\,,
\end{align*}
using~\eqref{eq:Fubini} in the last equation. Now, Equation~\eqref{eq:term2} yields that the dual program is:
\eq
\notag
\sup\Big\{\int_{\mathds R^d}\Phi c \,\dd\mu^0\ :\ c\in\mathds L\ \text{s.t.}\ \|\Phi c\|_\infty\leq1\Big\}\,.
\qe
Note also that the objective function of the dual program satisfies:
\[
\int_{\mathds R^d}\Phi c\,\dd\mu^0=\langle c,L\circ\Phi\mu^0\rangle_{\mathds L}\,,
\]
which gives the following equivalent formulation of the dual:
\eq
\label{eq:D0m}
\tag{$\mathbf D_0$}
\sup\Big\{\langle c,L\circ\Phi\mu^0\rangle_{\mathds L}\ :\ c\in\mathds L\ \text{s.t.}\ \|\Phi c\|_\infty\leq1\Big\}\,.
\qe
Note that the dual certificate $\Pm$ exists, then we know that $\mu^0$ is a solution to~\eqref{eq:P0m} by Theorem~3. %of \cite{super_mix}. 
As in Section~\ref{proof:sol_discrete}, we use the Slater condition to prove that there is no duality gap: we remark that a feasible point $c$ exists in the interior of the set $\{\|\Phi c\|_\infty\leq 1\}$. Now, the generalized Slater condition shall be used (see \textit{e.g.} \cite{rockafellar1974conjugate}). We get that any solution~$c$ to~\eqref{eq:D0m} satisfies that $\Phi c$ is a sub-gradient of the total variation norm at point $\mu^0$. 
We recall  the definition of condition~\eqref{eq:NDSCB}.
\eq
\label{eq:NDSCB}
\mathcal P_0\text{ exists}\,,\quad
\forall t\in\mathds F(r)\,,\  |\mathcal P_0(t)|<1-q\,,\quad
\forall t\in\mathds N(r)\,,\  \nabla^2 \mathcal P_0(t)\prec -\rho\,\mathrm{Id}_d\,,
\tag{NDB}
\qe
where $\rho>0$. Under condition~\eqref{eq:NDSCB}, we know that $\mathcal P_0:=\Phi c_0$ is a solution to~\eqref{eq:D0m}.

\medskip

\noindent
Consider also the following convex program:
\eq
\label{eq:Pkm}
\tag{$\mathbf P_\kappa(\Phi\mu^0)$}
\inf_{\mu \in \mathcal{M}(\bR^d,\bR)} 
\left\lbrace \frac12\| L\circ \Phi\mu^0 - L\circ\Phi \mu \|_{\bbL}^2 + \kappa \| \mu \|_1 \right\}\,,
\qe
which is the same as the one used in Section~\ref{proof:sol_discrete} and Theorem~6, 
%of \cite{super_mix}, 
exchanging $L\hat f_n$ by $ L\circ \Phi\mu^0$. Following the arguments  used in   Section~\ref{proof:sol_discrete}, one may prove that there is no duality gap and the dual program is given by:
\eq
\label{eq:Dkm}
\tag{$\mathbf D_\kappa(\Phi\mu^0)$}
\frac{\| L\circ \Phi\mu^0\|_{\bbL}^2}2-\kappa\inf
\Big\{\frac\kappa2\Big\| \frac{L\circ \Phi\mu^0}k-c\Big\|_{\bbL}^2\,:\ c\text{ s.t. }\|\Phi c\|_\infty\leq 1\Big\}.\,
\qe

We denote by $c_\kappa$ the solution to~\eqref{eq:Dkm} (unicity can be proven by~\eqref{eq:strict_coercive}) and by $\mathcal P_\kappa:=\Phi c_\kappa$ the dual polynomial. Its gradient is denoted by $\nabla\mathcal P_\kappa$, and its Hessian is denoted by $\nabla^2\mathcal P_\kappa$. We first state the next lemma.
\begin{lemma}
If $c_0$ exists, then $\|c_\kappa-c_0\|_\bbL\to0$, $\nabla\mathcal P_\kappa\to\nabla\mathcal P_0$ uniformly, and $\nabla^2\mathcal P_\kappa\to\nabla^2\mathcal P_0$ uniformly as $\kappa\to0$.
\end{lemma}
\begin{proof}
Since $c_\kappa$ is a solution to~\eqref{eq:Dkm}, it holds that:
\[
\frac\kappa2\Big\| \frac{L\circ \Phi\mu^0}k-c_\kappa\Big\|_{\bbL}^2
\leq
\frac\kappa2\Big\| \frac{L\circ \Phi\mu^0}k-c_0\Big\|_{\bbL}^2\,,
\]
leading to:
\eq
\label{yoyoweak}
\langle c_\kappa,L\circ\Phi\mu^0\rangle_{\mathds L}-\frac\kappa2\|c_\kappa\|_{\mathds L}^2
\geq\langle c_0,L\circ\Phi\mu^0\rangle_{\mathds L}-\frac\kappa2\|c_0\|_{\mathds L}^2
\,,
\qe
and $c_0$ being a solution to~\eqref{eq:D0m} implies that:
\[
\langle c_\kappa,L\circ\Phi\mu^0\rangle_{\mathds L}
\leq\langle c_0,L\circ\Phi\mu^0\rangle_{\mathds L}
\,.
\]
We deduce that $\|c_\kappa\|_{\bbL}\leq\|c_0\|_{\bbL}$. Closed unit balls of Hilbert spaces being weakly sequentially compact we deduce that  given $\kappa_n\to 0$, one shall extract a subsequence such that $c_{\kappa_n}$ weakly converges to some $c^*\in\bbL$. Taking the limit as $\kappa\to 0$ in~\eqref{yoyoweak}, we obtain that:
\[
\langle c^*,L\circ\Phi\mu^0\rangle_{\mathds L}
\geq\langle c_0,L\circ\Phi\mu^0\rangle_{\mathds L}\,.
\]

Note that $\Phi c_{\kappa_n}$ converges weakly to $\Phi c^*$ so that:
\[
\|\Phi c^*\|_\infty\leq\lim\inf_n\|\Phi c_{\kappa_n}\|_\infty\leq1
\]
We deduce that $c^*\in\mathds L$ is a solution to~\eqref{eq:D0m} and hence:
\[
\|\Phi c^*\|_{\infty}\leq 1\text{ and }(\Phi c^*)(t_k)=1\,,\ k\in[K]\,.
\]
Furthermore, $c^*$ is the solution of minimal norm since:
\[
\|c^*\|_{\bbL}\leq\lim\inf_n\|c_{\kappa_n}\|_{\bbL}\leq \|c_0\|_{\bbL}\,.
\]
The solution of minimal norm is unique by strict coercivity of the norm $\|\cdot\|_{\mathds L}$, see~\eqref{eq:strict_coercive}. We deduce that $c^*=c_0$, $\|c_{\kappa_n}\|_{\mathds L}\to \|c_0\|_{\mathds L}$, and $c_{\kappa_n}\to c_0$ strongly in ${\mathds L}$. Note that it implies that $\lim_{\kappa\to0}\|c_\kappa-c_0\|_\bbL=0$, since otherwise one can extract a subsequence $c_{\kappa_n}$ such that $\|c_{\kappa_n}-c_0\|_\bbL>\varepsilon$, and by the above argument, one can extract a sequence such that $c_{\kappa_n}\to c_0$.

Now, the Cauchy-Schwarz inequality yields:
\[
\forall t\in\mathds R^d,\quad\|\nabla^2\mathcal P_\kappa(t)-\nabla^2\mathcal P_0(t)\|_\infty\leq(\sup_{i,j}\|\partial^2\varphi/(\partial x_i\partial x_j)\|_\bbL)\|c_{\kappa}-c_0\|_\bbL\,,
\]
which proves the uniform convergence. The same computation gives the uniform convergence of the functions and their gradients.
\end{proof}
We denote by $c_{\kappa,n}$ the dual solution of~\eqref{eq:blasso} , namely:
\eq
\label{eq:Dkmn}
\tag{$\mathbf D_\kappa(\hat f_n)$}
\frac{\| L\hat f_n\|_{\bbL}^2}2-\kappa\inf
\Big\{\frac\kappa2\Big\| \frac{L\hat f_n}\kappa-c\Big\|_{\bbL}^2\,:\ c\text{ s.t. }\|\Phi c\|_\infty\leq 1\Big\}\,
\qe
and $\mathcal P_{\kappa,n}=\Phi c_{\kappa,n}$. The primal solution is denoted by $\hat \mu_n$.

\begin{lemma}
Let $0<t\leq r$ and assume \eqref{eq:NDSCB}. If $\kappa$ and $\| \vn \|_\bbL/\kappa$ are sufficiently small, any solution $\hat \mu_n$  has support of size $\hat K=K$ with one and only one spike in each near region~$\mathds N_k(t)$ for $k\in[K]$.
\end{lemma}
\begin{proof}
Note that~\eqref{eq:Dkmn} and~\eqref{eq:Dkm} are projection onto a closed convex set. We deduce that
\[
\|c_{\kappa,n}-c_{\kappa}\|_\bbL\leq\frac{\| \vn \|_\bbL}\kappa\,,
\]
and that $ \|\nabla^2\mathcal P_\kappa -\nabla^2\mathcal P_{\kappa,n}\|_\infty=\mathcal O(\frac{\| \vn \|_\bbL}\kappa)$ (the same result holds for the functions and their gradients). Under~\eqref{eq:NDSCB}, we know that there exists $0<q<1$, $r>0$ and $\rho>0$ such that $\nabla^2\mathcal P_0\prec  -\rho\mathrm{Id}_d$ on $\mathds N(r)$ and $|\mathcal P_0|<1-q$ on $\mathds F(r)$. We deduce that, for sufficiently small $\kappa$ and ${\| \vn \|_\bbL}/\kappa$, $\mathcal P_{\kappa,n}$ is such that $\nabla^2\mathcal P_{\kappa,n}\prec -(\rho/2)\mathrm{Id}_d$ on $\mathds N(r)$ and $|\mathcal P_{\kappa,n}|<1-q/2$ on $\mathds F(r)$. We deduce that at most $1$ point in each $\mathds N_k(r)$ is such that $\mathcal P_{\kappa,n}(\hat t_k)=1$. 

But, since $\mu^0$ is the unique solution of~\eqref{eq:P0m} (see Theorem~3), %of \cite{super_mix}, 
we deduce that $\hat\mu_n$ converges to~$\mu^0$ in the weak-*topology as $\kappa$ and ${\| \vn \|_\bbL}/\kappa$ go to zero. Hence, it holds that $\hat\mu_n(\mathds N_k(r))\to\mu^0(\mathds N_k(r))=a_k^0$. In particular, $\hat\mu_n$ has one spike in $\mathds N_k(r)$.

Now, by Taylor's theorem, observe that if \eqref{eq:NDSCB} with neighborhood size $r$ holds then it holds with neighborhood size~$t$.
\end{proof}

%\pagebreak

\noindent
It remains to bound $\| \vn \|_\bbL^2$, which is the purpose of the next lemma.

\begin{lemma} A large enough universal constant $C>0$ exists such that for any RKHS $\bbL$ associated to a nonnegative measure $\Lambda$:
\[
\| \vn \|_\bbL^2\leq C^2  \Lambda(\mathds R^d)\frac{\log(C/\alpha)}{n}
\]
\end{lemma}
with probability at least $1-\alpha$. Or equivalently
\[
\forall u>0,\quad
\mathds P\big[\| \vn \|_\bbL^2\geq u v_n \big]\leq C\exp(-{u})\,,
\]
where $v_n:=\frac{C^2\Lambda(\mathds R^d)}n=\frac{C^2\lambda(0)}n$.
\begin{proof}
Let $X$ be a random variable with density $f^0$, we observe that $\bbE_X L \delta_X=Lf^0$, and denote by $(Z_i)_{i \in [n]}$ the i.i.d. random variables:
\[
\forall i\in[n],\quad
Z_i:=L\delta_{X_i} - \bbE_X L \delta_X\,,
\]
which are i.i.d. centered random variables with values in $\bbL$. Observe that $\|L\delta_{X_i}\|_\bbL^2=\lambda(0)=\Lambda(\mathds R^d)$ by the representation property of RKHS and the definition of its spectral measure $\Lambda$. We deduce that 
\eq
\label{e:bound_yoyo}
\|Z_i\|_\bbL^2\leq 2 \Lambda(\mathds R^d)\,.
\qe
Using this inequality it holds that  
\begin{align*}
\| \vn \|_\bbL^2&=\| L\hat f_n - L f^0\|_\bbL^2\\
&=\|\frac1n\sum_i[ L\delta_{X_i} - \bbE_X L \delta_X]\|_\bbL^2\\
&=\frac1{n^2}\sum_i \|Z_i\|_\bbL^2 + \frac1{n^2}\sum_{i\neq j} \langle Z_i,Z_j\rangle_\mathbb{L}\\
&\leq \frac{2}n \Lambda(\mathds R^d) + \frac1{n^2}\sum_{i\neq j} \langle Z_i,Z_j\rangle_\mathbb{L} \,.
\end{align*}
Now, consider the kernel $h(X_i,X_j)=\langle Z_i,Z_j\rangle_\mathbb{L}$ and observe that the latter right hand side is a $U$-process. First, the Cauchy–Schwarz inequality and \eqref{e:bound_yoyo} lead to $\| h\|_\infty\leq 2 \Lambda(\mathds R^d)$. Second, check that this kernel is $\sigma$-canonical, namely:
\[
\mathds E_{X_j} h(X_i,X_j) = \mathds E_{X_i,X_j} h(X_i,X_j) =0\,.
\]
By Proposition 2.3 of \cite{arcones1993limit}, it follows that there exists two universal constants $C_1,C_2>0$ such that 
\[
\frac1{n^2}\sum_{i\neq j} \langle Z_i,Z_j\rangle_\mathbb{L} \leq 2C_1  \Lambda(\mathds R^d)\frac{\log(C_2/\alpha)}{n}\,,
\]
with probability at least $1-\alpha$.
\end{proof}

Let $\delta_\kappa>0$ be arbitrarily small. Set $\kappa=\kappa_n=\sqrt{\Lambda(\mathds R^d)}\,n^{-\frac12 + \delta_\kappa}$ so that 
\[
\|\vn\|_\bbL\leq C\sqrt{\Lambda(\mathds R^d)} \times n^{-\frac{1}2+\frac{\delta_\kappa}2} = o(\kappa_n)
\]
with probability greater than $1-e_n:=1-Ce^{-n^{\delta_\kappa }}$. In this case, with an overwhelming probability, 
 the requirements of the aforementioned Lemma 2 are met: $\kappa_n$ and $
\| \Gamma_n\|_{\bbL} \kappa_n^{-1}$ are small enough. Hence, a sequence of probability events {$(e_n)_{n \geq 1}$} exists such that {$\lim_{n \rightarrow + \infty} e_n=0$} (almost) exponentially fast and for which the desired result holds (with $\delta_\kappa=1/2-r_\kappa$).
It ends the proof of Theorem~11.
\end{proof}

%%%%%%%%%%%%%%%%%%%%%%%%%%%%%%%%%%%%%%%%%%%%%%%%%

\section{Construction of a dual certificate (proof of Theorem~7)} %of \cite{super_mix})}
\label{s:dual}

For a given set of points  $S^0=\{t_1,\ldots,t_K\}$, we recall that $\Delta := \min_{k \neq \ell} \|t_k-t_\ell\|_2$. For any $\alpha \in \mathds{R}^K$ and $\beta \in \mathds{R}^{Kd}$, we consider the function
\begin{equation}\label{def:pm}
 p^{\alpha,\beta}_m(t) = \sum_{k=1}^K \left\lbrace \alpha_k \psi_{m}(t-t_k) + \langle \beta_k,  \nabla \psi_{m}(t-t_k) \rangle \right\rbrace, \quad \forall t\in \bR^d.
 \end{equation}
For the sake of convenience, we omit the dependency in $\alpha$ and $\beta$ of the previous function and simply denote it by $p_m$.
We are interested in the existence of a set of coefficients $(\alpha,\beta)$ such that
$p_m$ defined in~\eqref{def:pm} satisfies the two interpolation conditions:
\begin{equation}\label{eq:interpolation}
\forall k \in \{1,\ldots, K\} \qquad p_m(t_k)=1 \quad \text{and} \quad \nabla p_m(t_k)=0.
\end{equation}
The following proposition establishes the control of $\alpha$ and $\beta$ due to the conditions~\eqref{eq:interpolation}.
%The following proposition displays the consequences of (\ref{eq:interpolation}) on the values of $\alpha$ and $\beta$.  

\begin{prop}\label{prop:sol_alpha_beta}
If $m$ is chosen such that $m \geq \frac{K^{1/4} d^{3/4}}{ \mathcal{C}\Delta}$ for some positive constant $\mathcal{C}$ small enough, then $(\alpha,\beta)$ exists such that~\eqref{eq:interpolation} holds 
and:
\begin{itemize}
\item $i)$ The supremum norms are upper bounded by:
\begin{equation*}%\label{convergence_infini}
 \|\alpha - \mathbf{1}_K \|_\infty \lesssim \frac{K d^3}{m^4 \Delta^4} \textrm{ and } \sup_{1 \leq k \leq K} \|\beta_k\|_\infty  \lesssim \frac{1}{m} \frac{K d^2}{m^4 \Delta^4}.
\end{equation*}
\item $ii)$ The Euclidean norm is upper bounded by: 
\begin{equation*}%\label{convergence_infini}
\sqrt{\sum_{k=1}^K \|\beta_k\|_2^2} \lesssim \frac{\sqrt{K}}{m\sqrt{d}} \times \frac{K d^3}{m^4 \Delta^4}.
\end{equation*}
\end{itemize}
\end{prop}

Even though not directly usable in our framework, we emphasize that the stability result and the  construction given in \cite{Candes_FernandezGranda_14} played a central role in our work to prove Proposition \ref{prop:sol_alpha_beta}.

\begin{proof}
The proofs of $i)$ and $ii)$ are divided into four steps. 

\noindent
\underline{\textbf{Step 1: Matricial formulation of~\eqref{eq:interpolation}}.}\\
The certificate $p_m$ should satisfy the following properties:
\begin{eqnarray*} 
& & \forall i\in[K]: \, \left\lbrace \begin{array}{c}
p_m(t_i) = 1 \\
\nabla p_m(t_i) = 0  
\end{array} \right. \\
& \Longleftrightarrow & \left\lbrace \begin{array}{l} 
      \alpha_i + \sum_{k\not = i} \alpha_k \psi_m(t_i-t_k) + \sum_{k=1}^K
      \sum_{v=1}^d \beta_k^v \partial_{v}(\psi_m)(t_i-t_k) = 1 \\
      \sum_{k=1}^K \alpha_k \partial_u (\psi_m) (t_i-t_k) + \sum_{k=1}^K \sum_{v=1}^d\beta_k^v \partial^2_{u,v}(\psi_m)(t_i-t_k) = 0  \end{array}\right. \quad \forall u\in [d], \forall i\in[K].
\end{eqnarray*} 
We can organize the above equations to obtain a linear system of $K (d+1)$ equations with $K(d+1)$ parameters. In the following, we denote these parameters by $\alpha=(\alpha_1,\ldots,\alpha_K)^T \in \mathds{R}^d$ 
and $\beta=(\beta_{1}^1,\ldots,\beta_1^d,\beta_2^1,\ldots,\beta_2^d,\ldots,\beta_K^1,\ldots,\beta_K^d)^T \in \mathds{R}^{Kd}$. The above equations can be rewritten as:
\begin{equation}\label{syst}
\left(\begin{array}{c c} I_K+A_m & D_m\\D_m^T & B_m - \frac{4}{3} m^2 I_{K\times d}\end{array}\right) \left( \begin{array}{c} \alpha \\ \beta \end{array}\right) = \left( \begin{array}{c} \mathbf{1}_K \\ \mathbf{0}_{Kd} \end{array} \right),
\end{equation}
where $\mathbf{1}_K$ denotes the vector of size $K$ having all its entries equal to $1$ (similar definition for $\mathbf{0}_{Kd} $). The matrix $A_m \in \mathds{R}^{K\times K}$ acts on the coefficients $\alpha$ as:
$$ (A_m)_{i,k} = \mathds{1}_{i\not =k} \psi_m(t_i-t_k) \quad \forall i,k \in [K],$$
while $D_m \in \mathds{R}^{K\times Kd}$ describes the effect of the partial derivatives of $\psi_m$ on $\alpha$ and $\beta$ as: 
$$
(D_m)_{i,(k,v)} = \partial_v (\psi_m)(t_i-t_k) \qquad \forall i,k \in [K] \quad \mathrm{and} \quad v \in [d].
$$
Finally, the squared matrix $B_m \in \mathds{R}^{Kd\times Kd}$ is given by:
$$
(B_m)_{(i,u),(k,v)} = \mathds{1}_{(i,u)\not =(k,v)}	\partial^2_{u,v}(\psi_m)(t_i-t_k) \qquad \forall i,k \in [K] \quad  u,v \in [d].
$$

\noindent\underline{\textbf{Step 2: Inversion of the system~\eqref{syst}}}\\
According to linear algebra results (see e.g. \cite{horn2012matrix}), the system~\eqref{syst} is invertible if and only if 
$$G_m:=B_m - \frac{4}{3} m^2 I_{Kd}$$ 
and its Schur complement 
$$H_m:=(I_K+A_m) - D_m G_m^{-1}D_m^T$$
 are both invertible. To prove this assertion, we remember that a symmetric matrix $M$ is invertible if $\|I -M\|_\infty < 1$, where $\|.\|_{\infty}$ is the subordinate matrix infinity norm ($\|M\|_{\infty} = \underset{i}{\max} \sum_{j} |M_{ij}|$). In such a case $\|M^{-1}\|_\infty \leq \frac{1}{1-\|I -M\|_\infty}$. \\
 Moreover, we will use in the sequel that $\|M\|_1 = \underset{j}{\max} \sum_{i} |M_{ij}| =\|M^T\|_\infty$.\\

$\bullet$ Invertibility of $G_m$ and computation of $\| G_m^{-1} \|_\infty$:

For all $i,k \in [K]$ and $u,v \in [d]$,
$$
(G_m)_{(i,u)(k,v)} = \left\{\begin{array}{l l l} -\frac{4m^2}{3} & \textrm{ if } & i=k, u=v\\
\partial^2_{(u,v)}(\psi_m)(0) & \textrm{ if } & i=k, v\neq u\\
\partial^2_{(u,v)}(\psi_m)(t_i-t_k) & \textrm{ if } & i\neq k, v\neq u
\end{array}\right.
$$
according to the definition of $B_m$. Setting $\tilde G_m = \frac{-3}{4 m^2} G_m$, we get
$$
\|I_{Kd}-\tilde G_m\|_\infty = \underset{(i,u)}{\max}\ \sum_{(k,v)} |(I_{Kd}-\tilde G_m)_{(k,v)}|
$$
with 
\begin{eqnarray*}
\sum_{(k,v)} |(I_{Kd}-\tilde G_m)_{(k,v)}| &=& \frac{3}{4 m^2} \sum_{k\neq i} \sum_{v=1}^d \left| \partial^2_{(u,v)}(\psi_m)(t_i-t_k) \right|  +\frac{3}{4 m^2}  \sum_{v\neq u} \left| \partial^2_{(u,v)}(\psi_m)(0) \right|\\
&\lesssim&  K d \frac{1}{m^2} \frac{d^2}{m^2 \Delta^4} = \frac{K d^3}{m^4 \Delta^4}
\end{eqnarray*}
according to Lemma \ref{lemma:tools} and $iii)$ of Lemma~\ref{lemma:borne_derivees}. 
Thus, if there exists a positive constant $\mathcal C$ small enough such that $\frac{K d^3}{m^4 \Delta^4}\leq \mathcal C$, then  $\|I_{Kd}-\tilde G_m\|_\infty<1/2$ and the matrix $G_m$ is invertible. 
Moreover, 
\begin{eqnarray}
\|G_m^{-1}\|_\infty &=& \frac{3}{4m^2} \|\tilde G_m^{-1}\|_\infty 
\leq  \frac{3}{4m^2} \frac{1}{1-\|I_{Kd} - \tilde G_m\|_\infty} \lesssim \frac{1}{m^2}. \label{Gminv}
\end{eqnarray} 

$\bullet$ Invertibility of $H_m$ and computation of $\| H_m^{-1} \|_\infty$:\\
In the same way, we want to prove that $\|I_K - H_m\|_\infty<1$. According to the properties of  the $\infty$-norm, 
\begin{eqnarray}
\|I_K - H_m\|_\infty &=& \|D_m G_m^{-1} D_m^T - A_m\|_\infty \nonumber\\
&\leq & \|A_m\|_\infty + \|D_m G_m^{-1} D_m^T\|_\infty \nonumber\\
&\leq & \|A_m\|_\infty + \|D_m\|_\infty \|G_m^{-1}\|_\infty \|D_m\|_1. \label{Hinvdecomp}
\end{eqnarray}

In a first time, we provide an upper bound on $\|A_m\|_{\infty}$. Remark that for any $i \in [K]$
\begin{eqnarray*}
 \sum_{j=1}^K | (A_m)_{ij} |  =  \sum_{j=1}^K | \psi_m(t_i - t_j) | \mathds{1}_{i\not = j}.
\end{eqnarray*}
Applying $i)$ of Lemma \ref{lemma:borne_derivees}, we hence obtain
\begin{equation}\label{eq:linfiniAm}
\|A_m\|_{\infty} \lesssim \frac{K d^2}{m^4 \Delta^4}.
\end{equation}
Now, recall that
\begin{eqnarray*}
   \sum_{k=1}^K \sum_{v=1}^d | (D_m)_{i,(k,v)} |
& = & \sum_{k=1}^K\sum_{v=1}^d | \partial_v (\psi_m)(t_i-t_k)|.
\end{eqnarray*}
Applying  $ii)$ of  Lemma \ref{lemma:borne_derivees}, % and use that when $i=k$, $\nabla \psi_m(t_i-t_k)=\nabla \psi_m(0)=0$ 
we deduce that:
\begin{equation}\label{eq:linfiniDm}
\|D_m\|_{\infty} \lesssim \frac{K d^3}{m^3 \Delta^4}.
\end{equation}

%Remark that, when $i\not = j$, 
%$$ \psi_m(t_i-t_j) = \psi^2(m (t_i-t_j)) = \left[ \frac{\sin(m (t_i-t_j))}{m (t_i-t_j)}   \right]^2 \leq \frac{1}{m^2 \Delta^2},$$
%where $\Delta = \min_{i,j} |t_i - t_j|$. In the same time, using the inequality $|\sin(x)| \leq |x|$, we have
%\begin{eqnarray*} 
%|(\psi_m)'(t_i-t_j) |
%& = & 2m |\psi'(m (t_i-t_j)) \psi(m (t_i-t_j))|, \\
%& = & 2m \left|\left( \frac{\cos(m(t_i-t_j))}{m(t_i-t_j)} - \frac{\sin(m(t_i-t_j))}{(m(t_i-t_j))^2}  \right)\left[ \frac{\sin(m (t_i-t_j))}{m (t_i-t_j)}   \right]\right|, \\
%& \leq & 2 m \left[ (m \Delta)^{-1}+(m \Delta)^{-1} \left|  \frac{\sin(m(t_i-t_j))}{(m(t_i-t_j))} \right| \right] (m \Delta)^{-1}\\
%& \leq & \frac{4}{m   \Delta^2}.
%\end{eqnarray*}
%Applying $i)$ and $ii)$ of Lemma \ref{lemma:borne_derivees} and using that 
%when $i=k$, then $\nabla \psi_m(t_i-t_k)=\nabla \psi_m(0)=0$, we deduce that:
%$$ \sum_{j=1}^K | (A_m)_{ij} | + \sum_{(k,v)} | (D_m)_{i,(k,v)} | \lesssim \frac{K d^2}{m^4 \Delta^4} + \frac{ K d^3}{m^3  \Delta^4 } \lesssim \frac{K d^3}{m^3 \Delta^4}.$$
Following the same ideas, for any pair $(i,u)$ with
 $i \in [K]$ and $u \in [d]$, we have: 
 $$
\sum_{j=1}^K |(D_m)_{j,(i,u)}| = \sum_{j=1}^K |\partial_u( \psi_m)(t_i - t_j) |.$$
 Again, $ii)$ of Lemma \ref{lemma:borne_derivees}  yields:
\begin{equation}\label{eq:linfiniDmT}
 \|D_m\|_{1} = \|D_m^T\|_{\infty}
 \lesssim  \frac{K d^2}{m^3 \Delta^4}.
 \end{equation}
 Gathering~\eqref{Gminv},~\eqref{eq:linfiniAm},~\eqref{eq:linfiniDm} and~\eqref{eq:linfiniDmT} in~\eqref{Hinvdecomp}, we deduce that:
 \begin{eqnarray}
 \|I_K - H_m\|_\infty &\lesssim& \frac{K d^2}{m^4 \Delta^4} + \frac{K d^2}{m^3 \Delta^4}\frac{K d^3}{m^3 \Delta^4} \frac{1}{m^2} 
 \lesssim  \frac{K d^3}{m^4 \Delta^4}. \label{Hinv}
\end{eqnarray}
provided $\frac{Kd^3}{m^4\Delta^4} \leq \mathcal{C}$ for some constant $\mathcal C$ small enough. This implies that under such a condition, $\|I_K - H_m\|_\infty <1/2$. Moreover,  the Schur complement $H_m=(I_K+A_m) - D_m G_m^{-1}D_m^T$ is then invertible and
\begin{equation}
\|H_m^{-1}\|_\infty \leq \frac{1}{1-\|I_K - H_m\|_\infty} \leq 1+2\|I_K - H_m\|_\infty \leq 1 + C  \frac{K d^3}{m^4 \Delta^4},
\label{Hminv}
\end{equation}
for some positive constant $C$, provided the constraint $\frac{Kd^3}{m^4\Delta^4} \leq \mathcal{C}$ is satisfied. 
 
\noindent To conclude this second step, the system~\eqref{syst} is invertible if $\frac{K d^3}{\Delta^4 m^4}\leq \mathcal C$ for some constant $\mathcal{C}$ small enough. In such a case
\begin{equation}
\left(\begin{array}{c} \alpha \\ \beta \end{array}\right) = \left(\begin{array}{c} I_K \\ -G_m^{-1} D_m^{T}\end{array}\right) H_m^{-1} \mathbf{1}_K. 
\label{systinv}
\end{equation}

\noindent\underline{\textbf{Step 3: Proof of i) }}
In the sequel, we assume that $\frac{Kd^3}{m^4\Delta^4} \leq \mathcal{C}$ for some positive constant $\mathcal{C}$ small enough. First, according to (\ref{Hminv}) and~\eqref{systinv}, we obtain that:
$$
\|\alpha \|_\infty  = \| H_m^{-1} \|_\infty  \lesssim 1+C \frac{K d^3}{m^4 \Delta^4}.
$$
Moreover
$$
\alpha-\mathbf{1}_K = (H_m^{-1} - I_K) \mathbf{1}_K = ( (I_K + \tilde H_m)^{-1} -I_K) \mathbf{1}_K
$$
with $\tilde H_m = A_m - D_mG_m^{-1}D_m^T$. 
Hence, since for $\tilde{H}_m$ small enough (\textit{i.e.} for a sufficiently small norm) we have
$$
(I_K+\tilde H_m)^{-1} = \sum_{k \geq 0} (-\tilde{H}_m)^{k} = I_K + \sum_{k \geq 1} (-\tilde{H}_m)^{k}.
$$
Hence, for $\tilde{H}_m$ small enough,
\begin{equation}\label{eq:upLinfini}
\|\alpha-\mathbf{1}_K\|_{\infty}  \leq \left\| \sum_{k \geq 1} (-\tilde{H}_m)^k \right\|_{\infty} \leq 
\left\|\tilde{H}_m \right\|_{\infty} \sum_{k \geq 0} \left\|\tilde{H}_m \right\|_{\infty}^k.
\end{equation}
According to~\eqref{Hinv},
\begin{equation}\label{eq:linfiniH}
\|\tilde{H}_m\|_{\infty} =\|H_m - I_K\|_\infty \lesssim \frac{K d^3}{m^4 \Delta^4},
\end{equation}
and we can choose the constant $\mathcal C$ small enough in the constraint $\frac{K d^3}{m^4\Delta^4}\leq \mathcal C$ such that $\| \tilde H_m \|_\infty \leq 1/2$. 
We conclude that 
$$
\|\alpha -\mathbf{1}_K\|_\infty \lesssim \frac{K d^3}{m^4 \Delta^4}.
$$
In a second time, gathering~\eqref{Gminv},~\eqref{eq:linfiniDmT} and~\eqref{Hminv}, we deduce that: 
\begin{eqnarray} 
\|\beta\|_\infty &\leq & \|G_m^{-1} D_m^T H_m^{-1} \mathbf{1}_K\|_\infty,  \nonumber\\
&\leq &  \|G_m^{-1} D_m^T H_m^{-1}\|_\infty, \nonumber\\
&\leq &  \|G_m^{-1}\|_\infty \|D_m\|_1 \|H_m^{-1}\|_\infty,  \nonumber\\
&\lesssim & \frac{1}{m} \frac{Kd^2}{m^4 \Delta^4}. \label{betanorm} 
\end{eqnarray} 
 
\noindent\underline{\textbf{Step 4: Proof of ii) }} 
According to~\eqref{systinv}, 
\begin{eqnarray}
\|\beta\|_2 = \|G_m^{-1} D_m^T H_m^{-1} \mathbf{1}_K\|_2
& \leq & \sqrt{K}\, \|G_m^{-1} D_m^T H_m^{-1} \|_2, \nonumber \\
& \leq & \sqrt{K} \sqrt{ \|G_m^{-1} D_m^T H_m^{-1}\|_{1} \|G_m^{-1} D_m^T H_m^{-1}\|_{\infty}}.
\end{eqnarray} 

%Using the definition of the $\|.\|_2$ matricial norm related to the spectral radius $\rho$,  
%\begin{align*}
%\|G_m^{-1} D_m^T H_m^{-1}\|_2 &= \sqrt{\rho((G_m^{-1} D_m^T H_m^{-1})^T (G_m^{-1} D_m^T H_m^{-1}))} \\
%& \leq \sqrt{ \|(G_m^{-1} D_m^T H_m^{-1})^T (G_m^{-1} D_m^T H_m^{-1})\|_{\infty}}\\
%& \leq \sqrt{ \|(G_m^{-1} D_m^T H_m^{-1})^T \|_{\infty} \|G_m^{-1} D_m^T H_m^{-1}\|_{\infty}}\\
%& \leq \sqrt{ \|G_m^{-1} D_m^T H_m^{-1}\|_{1} \|G_m^{-1} D_m^T H_m^{-1}\|_{\infty}}.
%\end{align*}

Using~\eqref{betanorm}, $\|G_m^{-1} D_m^T H_m^{-1} \|_\infty \leq \|G_m^{-1}\|_\infty \|D_m\|_1 \|H_m^{-1}\|_\infty \lesssim \frac{\mathcal C}{md}$.
For the second term, we use the dual relationship between $\|.\|_\infty$ and $\|.\|_1$ and  that the matrices $G_m$ and $H_m$ are symmetric. Gathering~\eqref{Gminv},~\eqref{eq:linfiniDm} and~\eqref{Hminv}, we obtain that: 
\begin{eqnarray}
\|G_m^{-1} D_m^T H_m^{-1} \|_1 &\leq & \|G_m^{-1}\|_1 \|D_m^T\|_1 \|H_m^{-1}\|_1 \nonumber\\
&\leq & \|G_m^{-1}\|_\infty \|D_m\|_\infty \|H_m^{-1}\|_\infty \nonumber\\
&\lesssim&\frac{1}{m^2} \frac{Kd^3}{m^3 \Delta^4} = \frac{1}{m} \frac{K d^3}{m^4 \Delta^4}.
\end{eqnarray}
We then deduce that
$$
\|\beta\|_2 \lesssim \frac{\sqrt{K}}{m\sqrt{d}} \frac{K d^3}{m^4 \Delta^4}. 
$$
\end{proof}

Thanks to the previous proposition, we are now ready to prove Theorem~7. 
%of \cite{super_mix}. % \ref{theo:main_certificate}. 
Our strategy is inspired from the one of \cite{Candes_FernandezGranda_14}.

\begin{proof}[Proof of Theorem~7] %of \cite{super_mix}]
We define an integer $m$ that will be chosen large enough below and consider $\Pm = p_m^2$. 

\noindent
\underline{\textbf{Proof of $i)$ and $ii)$:}}
From Proposition \ref{prop:sol_alpha_beta}, we know that if $m$ satisfies 
$
m \geq \mathcal{C} \frac{K^{1/4} d^{3/4}}{\Delta},
$
for a constant $\mathcal{C}$ large enough independent from $K$, $\Delta$ and $d$, then $\Pm$ satisfies  the interpolation properties:
$$
0 \leq \Pm \leq 1 \qquad \text{with} \qquad \Pm(t) = 1 \Longleftrightarrow t \in \{t_1,\ldots,t_K\}.
$$
Our strategy relies on a   study of the variations of $\Pm$ near each  support points $\{t_1,\ldots,t_K\}$, whose union defines the \textit{near region}, and far from these support points, which is then the \textit{far region}.

\paragraph{Near region}
Let $\epsilon \in \left(0,\frac{\Delta}{2}\right)$ a parameter whose value will be made precise later on. The near-region $\bbN(\epsilon)$ is the union of $K$ sets that are defined by:
$$ \bbN(\epsilon) = \bigcup_{i=1}^K \lbrace t\in \bR^d, \, \|t-t_i\|_2\leq \epsilon \rbrace := \bigcup_{i=1}^K \bbN_{i}(\epsilon).$$
The far region is therefore given by:
$$ \bbF(\epsilon) = \bR^d \setminus  \bbN(\epsilon).$$ 
Let $i\in \lbrace 1,\dots, K \rbrace$ be fixed, the function $p_m$ involves a sum over $k \in [K]$ and we consider two cases:
\begin{itemize}
\item If $k\not = i$, then, for all $t\in \bbN_{i}(\epsilon)$, $\xi_{t,i,k}$ exists such that
$$ \psi_m(t-t_k) = \psi_m(t_i-t_k)  + \langle (t-t_i), \nabla \psi_m(t_i-t_k)\rangle + \frac{1}{2} (t-t_i)^T D^2 \psi_m(\xi_{t,i,k}-t_k)(t-t_i),$$
with $\|\xi_{t,i,k} - t_i \|_2 \leq \|t-t_i\|_2$. 
Moreover, for any $u \in [d]$, a $\tilde\xi_{t,i,k}^u$ exists such that:
\begin{align*}
 \partial_u(\psi_m)(t-t_k)& = \partial_u(\psi_m)(t_i-t_k)  + \langle (t-t_i) ,(\partial_{u,v}(\psi_m)(t_i-t_k))_v \rangle \\
  &  + \frac{1}{2} (t-t_i)^T D^2\{\partial_u(\psi_m)\}(\tilde\xi_{t,i,k}^u-t_k)(t-t_i),\end{align*}
with $\|\tilde\xi_{t,i,k}^u - t_i \|_2 \leq \|t-t_i\|_2$.
\item If $k=i$,  since $\nabla \psi_m(0)=0$ and $D^3(\psi_m)(0) = 0$, for all $t\in \bbN_{i}(\epsilon)$, a $\xi_{t,i,i}$ exists such that:
\begin{align*}
 \lefteqn{\psi_m(t-t_i) = \psi_m(0) + \frac{1}{2} (t-t_i)^T D^2(\psi_m)(0)(t-t_i)}\\
 &  + \frac{1}{24} \underbrace{\sum_{1 \leq u_1,u_2,u_3,u_4 \leq d} (t^{u_1}-t_i^{u_1})
(t^{u_2}-t_i^{u_2})(t^{u_3}-t_i^{u_3})(t^{u_4}-t_i^{u_4}) \partial_{u_1,u_2,u_3,u_4}(\psi_m)(\xi_{t,i,i} - t_i)}_{:=(t-t_i)^T A(\xi_{t,i,i} - t_i)(t-t_i)}
\end{align*}
with $\|\xi_{t,i,i} - t_i \|_2 \leq \|t-t_i\|_2$. We also have that for any $u \in \{1,\dots,d\}$, the existence of $\tilde\xi_{t,i,i}^u$ such that: 
$$ \partial_u \psi_m(t-t_i) = \partial_u \psi_m (0) + \langle t-t_i, (\partial_{u,v}(\psi_m)(0))_v\rangle + \frac{1}{2} (t-t_i)^T D^2(\partial_u(\psi_m))(\tilde \xi_{t,i,i}^u - t_i)(t-t_i),$$
with $\|\tilde\xi_{t,i,i}^u - t_i \|_2 \leq \|t-t_i\|_2$.
\end{itemize}

 \noindent
Hence, for all $t\in \bbN_{i}(\epsilon)$, we can use the previous Taylor formulas and obtain that:
{\small 
\begin{eqnarray*}
p_m(t) & = & \sum_{k=1}^K \left[ \alpha_k \psi_m(t-t_k) + \langle \beta_k, \nabla \psi_m(t-t_k)\rangle \right], \\
& = & \alpha_i \psi_m(t-t_i) + \langle \beta_i, \nabla \psi_m(t-t_i)\rangle +\sum_{k\not = i} \alpha_k \psi_m(t-t_k)  + \sum_{k\not = i} \langle \beta_k, \nabla \psi_m(t-t_k)\rangle \\
& = & \alpha_i \left[ \psi_m(0) + \frac{1}{2} (t-t_i)^T D^2(\psi_m)(0)(t-t_i) + \frac{1}{24} (t-t_i)^T A(\xi_{t,i,i} - t_i) (t-t_i)\right] \\
&  & + \left\langle \beta_i,   \nabla \psi_m(0) +  D^2(\psi_m)(0)(t-t_i) + \frac{1}{2} \left((t-t_i)^T D^2 \partial_u(\psi_m)(\tilde \xi_{t,i,i}^u - t_i)  (t-t_i)\right)_u\right\rangle\\
&  & + \sum_{k\not = i} \alpha_k \left[  \psi_m(t_i-t_k)  + \langle t-t_i, \nabla \psi_m(t_i-t_k)\rangle + \frac{1}{2} (t-t_i)^T D^2(\psi_m)(\xi_{t,i,k}-t_k)(t-t_i)  \right] \\
&  & + \sum_{k\not = i} \left\langle \beta_k,\nabla \psi_m(t_i-t_k)  + D^2( \psi_m)(t_i-t_k)(t-t_i) + \frac{1}{2} \left((t-t_i)^T D^2\partial_u(\psi_m)(\tilde\xi_{t,i,k}-t_k)(t-t_i)\right)_u   \right\rangle.
\end{eqnarray*}
}
These terms can be re-arranged as follows:
\begin{eqnarray*}
p_m(t) & = & \sum_{k=1}^K \left[ \alpha_k \psi_m(t_i-t_k) + \langle \beta_k, \nabla \psi_m(t_i-t_k)\rangle \right] \\
&  & +   \left \langle  D^2(\psi_m)(0)\beta_i + \sum_{k\not =i} \alpha_k \nabla\psi_m(t_i-t_k) +  D^2(\psi_m)(t_i-t_k)\beta_k , (t-t_i) \right\rangle \\
&  & + \frac{(t-t_i)^T}{2} 
\left[ \alpha_i D^2(\psi_m)(0) + \sum_{k\neq i} \alpha_k D^2(\psi_m)(\xi_{t,i,k}-t_k)  \right. \\
& & \left. + \frac{\alpha_i}{12} A(\xi_{t,i,i} - t_i)+ \sum_{k=1}^K \sum_{u=1}^d \beta_i^u D^2 (\partial_u\psi_m)(\tilde\xi_{t,i,k}^u - t_i)
\right](t-t_i)\\
%\frac{\alpha_i}{12} (t-t_i)^2 \psi_m^{(4)} (\xi_{t,i,i} - t_i)\right\rbrace + \beta_i \psi_m^{(3)}(\tilde \xi_{t,i,i} - t_i)   \right. \\
%&  & \left. + \sum_{k\not = i} \alpha_k \psi_m''(\xi_{t,i,k} - t_k) + \sum_{k\not = i } \beta_k \psi_m^{(3)} (\tilde \xi_{t,i,k} - t_k) \right], \\
& = & C_0 + \langle C_1,t-t_i\rangle + \frac{1}{2}(t-t_i)^T C_2(t) (t-t_i).
\end{eqnarray*}
Of course, the construction of Proposition \ref{prop:sol_alpha_beta} yields 
$$ C_0 = \sum_{k=1}^K \left[ \alpha_k \psi_m(t_i-t_k) + \langle \beta_k,  \nabla\psi_m (t_i-t_k)\rangle \right] = p_m(t_i) = 1,$$
and 
$$ C_1 =    \sum_{k\not =i} \alpha_k \nabla\psi_m(t_i-t_k) + \sum_{k=1}^K  D^2(\psi_m)(t_i-t_k)\beta_k = \nabla p_m(t_i) = 0,$$
%\sum_{k\not =i} \alpha_k \psi_m'(t_i-t_k) + \sum_{k=1}^s \beta_k \psi_m''(t_i-t_k) = p_m'(t_i) =0,$$
thanks to the constraints expressed on the function $p_m$. Hence, for all $t\in \bbN_{i}(\epsilon)$ we have
$$ p_m(t) = 1 + \frac{1}{2} (t-t_i)^T C_2(t)(t-t_i).$$
In the following, we prove that $C_2$ is a negative matrix and bounded from below. 
%Recall that 
%\begin{eqnarray*}
%C_2(t) & = & \left\lbrace \alpha_i \psi_m''(0) + \frac{\alpha_i}{12} (t-t_i)^2 \psi_m^{(4)} (\xi_{t,i,i} - t_i)\right\rbrace + \beta_i \psi_m^{(3)}(\tilde \xi_{t,i,i} - t_i)  \\
%&   & + \underbrace{  \sum_{k\not = i} \alpha_k \psi_m''(\xi_{t,i,k} - t_k) + \sum_{k\not = i } \beta_k \psi_m^{(3)} (\tilde \xi_{t,i,k} - t_k)}_{:=\vartheta_i(t)} . \\
%& := & \left\lbrace \alpha_i (\psi_m)''(0) + \frac{\alpha_i}{12} (t-t_i)^2 (\psi_m)^{(4)} (\xi_{t,i,i} - t_i)\right\rbrace + \beta_i (\psi_m)^{(3)}(\tilde \xi_{t,i,i} - t_i)  + \epsilon_i(t).
%\end{eqnarray*}
Thanks to Lemma \ref{lemma:borne_derivees}, we can compute the first term of $C_2$ and we have 
$$ D^2 (\psi_m)(0) = -\frac{4m^2}{3} I_d,$$
which entails
\begin{equation}\label{eq:c1}
\frac{1}{2} (t-t_i)^T D^2 (\psi_m)(0)(t-t_i) = - \frac{2 m^2}{3} \|t-t_i\|_2^2.
\end{equation}
The second term of $C_2$ may be upper bounded with the help of the spectral radius of $D^2(\psi_m)(\xi_{t,i,k} - t_k)$: (denoted by $\rho(M)$ for any squared symmetric matrix $M$):
$$
\frac{1}{2} (t-t_i)^T \sum_{k \neq i} \alpha_k D^2 (\psi_m)(\xi_{t,i,k} - t_k)(t-t_i)
\leq \|\alpha\|_{\infty} \|t-t_i\|_2^2 \sum_{k\neq i} \rho\left(D^2 (\psi_m)(\xi_{t,i,k} - t_k)\right)\,.
$$

To handle this last term, we use the fact that in the near region $\bbN_{i}(\epsilon)$,  $\|\xi_{t,i,k} - t_k \|_2$ is far from $0$. Using the triangle inequality, since $\epsilon < \frac{\Delta}{2}$, we have for any $k\in [K]$ with $k\not = i$
$$ \| \xi_{t,i,k} - t_k \|_2 \geq \|t_i - t_k\|_2 - \| \xi_{t,i,k} - t_i \|_2 \geq \|t_i - t_k\|_2 - \|t-t_i\|_2 \geq \Delta - \epsilon \geq \frac{\Delta}{2}.$$

Using the inequality $\rho(M) \leq \|M\|_{\infty}$ for any symmetric matrix, Proposition \ref{prop:sol_alpha_beta} and $iii)$ of Lemma \ref{lemma:borne_derivees}, we obtain that:
\begin{eqnarray}
\frac{1}{2} (t-t_i)^T \sum_{k \neq i} \alpha_k D^2 (\psi_m)(\xi_{t,i,k} - t_k)(t-t_i)
&\lesssim &K \|\alpha\|_{\infty}   \left(d \times \frac{d^2}{m^2 \Delta^4} \right)  \|t-t_i\|_2^2 \nonumber\\
& \lesssim  &\frac{K d^3}{m^2 \Delta^4}   \|t-t_i\|_2^2.\label{eq:c2}
\end{eqnarray}

The third term of $C_2$ is described by the matrix:
$$
\frac{\alpha_i}{12} \left( A(\xi_{t,i,i}-t_i)\right)_{u,v} =
\frac{\alpha_i}{12} \left( \sum_{p=1}^d \sum_{q=1}^d 
(\xi_{t,i,i}^{p}-t_i^p)(\xi_{t,i,i}^{q}-t_i^q) \partial_{u,v,p,q}\psi_m(\xi_{t,i,i}-t_i)\right)_{u,v} \quad \forall u,v\in [d].
$$
Using that %$\|g'\|_\infty \vee \|g^{(2)}\|_{\infty} \vee \|g^{(3)}\|_\infty \vee \|g^{(4)}\|_\infty \leq 1/2$, 
$\|\sinc'\|_\infty \vee \|\sinc^{(2)}\|_{\infty} \vee \|\sinc^{(3)}\|_\infty \vee \|\sinc^{(4)}\|_\infty \leq 1/2$, 
we obtain that $\partial_{u,v,p,q}\psi_m(\xi_{t,i,i}-t_i) \lesssim m^4$. Therefore, for any $(u,v)\in [d]^2$, we have:
\begin{eqnarray*}
\left|
\frac{\alpha_i}{12}  A(\xi_{t,i,i}-t_i)_{u,v}\right| &\leq& 
\|\alpha\|_{\infty} \sum_{p=1}^d \sum_{q=1}^d 
\left|\xi_{t,i,i}^{p}-t_i^p \right| \left| \xi_{t,i,i}^{q}-t_i^q \right| \left|\partial_{u,v,p,q}\psi_m(\xi_{t,i,i}-t_i)\right| \\
& \lesssim & \|\alpha\|_{\infty} m^4 \sum_{p=1}^d \left| \xi_{t,i,i}^{p}-t_i^p \right| \sum_{q=1}^d \left| \xi_{t,i,i}^{q}-t_i^q \right| \\
& \lesssim &d m^4 \|\alpha\|_{\infty} \epsilon^2,
\end{eqnarray*}
where the last line comes from the Cauchy-Schwarz inequality.
Again, the inequality $\rho(M) \leq \|M\|_\infty$ and Proposition \ref{prop:sol_alpha_beta}   yield:
\begin{equation}\label{eq:c3}
\left|\frac{1}{2} (t-t_i)^T \frac{\alpha_i}{12} A(\xi_{t,i,i} - t_i) (t-t_i)\right| \lesssim d^2 m^4 \|\alpha\|_{\infty} \epsilon^2 \|t-t_i\|_2^2
\lesssim d^2 m^4   \epsilon^2 \|t-t_i\|_2^2.
\end{equation}

The last term of $C_2$ is studied into two steps. We first consider the situation when $k \neq i$:  the triangle inequality, $iv)$ of Lemma \ref{lemma:borne_derivees} and the inequality $\rho(M) \leq \|M\|_{\infty}$ yield:
\begin{eqnarray*}
\rho \left( 
\sum_{k \neq i} \sum_{u=1}^d \beta_i^u D^2(\partial_u \psi_m)(\tilde{\xi}_{t,i,k}^u-t_i)\right) & \leq & K \|\beta\|_{\infty} d \sup_{1 \leq u \leq d} \rho\left( D^2(\partial_u \psi_m)(\tilde{\xi}_{t,i,k}^u-t_i)\right), \\
& \lesssim & K \|\beta\|_{\infty} d \times \left( d \times  \frac{d^2}{m \Delta^4}\right), \\
& \lesssim & K \|\beta\|_{\infty} \frac{d^4}{m \Delta^4}.
\end{eqnarray*}
%because each term involved in $D^2(\partial_u \psi_m)(\tilde{\xi}_{t,i,k}^u-t_i)$ is upper bounded by $\frac{d^2}{m \Delta^4}$ thanks to $iv)$ of Lemma \ref{lemma:borne_derivees}.
Hence, we deduce from Proposition \ref{prop:sol_alpha_beta} that:
\begin{equation}\label{eq:c4}
\rho \left( 
\sum_{k \neq i} \sum_{u=1}^d \beta_i^u D^2(\partial_u \psi_m)(\tilde{\xi}_{t,i,k}^u-t_i)\right) \lesssim K \times \frac{1}{m} \frac{Kd^2}{m^4\Delta^4} \times \frac{d^4}{m\Delta^4} \lesssim  m^2 \left(\frac{K d^3}{m^4 \Delta^4}\right)^2.%\frac{Kd^2}{m^4\Delta^4}   
%\frac{K^2 d^7}{m^4 \Delta^8}.
\end{equation}

Now, we consider the situation where $k=i$. For any pair $(u,v) \in [d]^2$:
\begin{eqnarray*}
\sum_{w =1}^d \beta_i^w \partial_{u,v,w}(\psi_m)(\tilde{\xi}_{t,i,i}^u-t_i) 
&\lesssim& d \|\beta\|_{\infty} m^3 (m \epsilon+(m \epsilon)^3) , \\
&\lesssim& d \times \frac{1}{m} \frac{K d^2}{m^4 \Delta^4} \times m^3 (m \epsilon +(m \epsilon)^3), \\
&\lesssim &  m^2 \times \frac{Kd^3}{m^4\Delta^4} \times (m \epsilon +(m \epsilon)^3)  ,
\end{eqnarray*}
where we used $iv)$ of Lemma~\ref{lemma:tools},  $\nabla \psi_m(0)=0$, $D^3 \psi_m(0)=0$ and $\|m (\tilde{\xi}_{t,i,i}^u-t_i)\|_2 \leq m \epsilon$ and $i)$ of Proposition \ref{prop:sol_alpha_beta}.

\noindent
Using the previous bounds, we then conclude that 
\begin{eqnarray}\label{eq:c5}
\rho \left( 
 \sum_{u=1}^d \beta_i^u D^2(\partial_u \psi_m)(\tilde{\xi}_{t,i,i}^u-t_i)\right)
 & \lesssim & m^2\left(\frac{Kd^3}{m^4\Delta^4} \right)^2+ m^2 \frac{Kd^3}{m^4\Delta^4} ((m\epsilon) +(m\epsilon)^3)\nonumber  \\
 %& \lesssim & \frac{K d^3}{m^2 \Delta^4} C^{-1} + C^{-1} m^2 (m \epsilon +(m \epsilon)^3) \nonumber\\
 %&\lesssim & m^2 C^{-1} \left\{\frac{K d^3}{m^4 \Delta^4} + (m \epsilon +(m \epsilon)^3)  \right\} \nonumber\\
&\lesssim & m^2 \left\{ 1 + (m \epsilon) +(m \epsilon)^3)  \right\} 
% \frac{K^2 d^7}{m^4 \Delta^8} + \frac{K d^4}{\Delta^4} (m \epsilon +(m \epsilon)^3) \nonumber \\
% & \lesssim
%& \frac{K d^4}{\Delta^4} \left[ \frac{K d^3}{m^4 \Delta^4} + (m \epsilon +(m \epsilon)^3)\right] \nonumber\\
%& \lesssim
%&\frac{K d^4}{\Delta^4} \left[ m^{-1} + (m \epsilon +(m \epsilon)^3)\right],
\end{eqnarray}
provided $ \frac{Kd^3}{m^4 \Delta^4} \leq \mathcal{C}$ for a constant $\mathcal{C}$ small enough. 

%Since $t\in \bbN_{i}(\epsilon)$, we deduce that $\| t-t_i \|_2 < \epsilon$.
%We choose 
%\begin{equation}
%\epsilon = \frac{\delta}{m}.
%\label{epsilon}
%\end{equation}
We now plug Equations~\eqref{eq:c1},~\eqref{eq:c2},~\eqref{eq:c3},~\eqref{eq:c4} and~\eqref{eq:c5} in $C_2(t)$ and deduce that a constant $\square$ exists such that
$$
\frac{1}{2} (t-t_i)^T C_2(t)(t-t_i) \leq m^2 \|t-t_i\|_2^2 \left[ \underbrace{
-\frac{2}{3} \alpha_i   + \square \left[ \frac{K d^3}{m^4 \Delta^4} +  d^2 m^2 \epsilon^2  +  [1+m \epsilon + (m \epsilon)^3] \right]
}_{:=A_{\epsilon,m}}\right].
$$
%$$
%\frac{1}{2} (t-t_i)^T C_2(t)(t-t_i) \leq m^2 \|t-t_i\|_2^2 \left[ \underbrace{
%-\frac{2}{3} \alpha_i   + \square \left[ \frac{K d^3}{m^4 \Delta^4} +  \alpha_i d^3 m^2 \epsilon^2  +  \frac{K d^4}{m^2 \Delta^4} [m^{-1}+m \epsilon + (m \epsilon)^3] \right]
%}_{:=A_{\epsilon,m}}\right].
%$$
Then, we choose $\epsilon$ and $m$ such that $A_{\epsilon,m} \leq - \frac{\alpha_i}{3}$. A careful inspection of the above terms prove that a sufficiently small $\upsilon$  and large enough $C$  (both independent of $d$, $K$ and $\Delta$) exist such that
\begin{equation}\label{eq:nearup}
\epsilon \leq  \frac{\upsilon}{m d} \quad \text{and} \quad m \geq C \frac{K^{1/4}d^{3/4}}{\Delta} \Longrightarrow \frac{1}{2} (t-t_i)^T C_2(t)(t-t_i)   \leq - \frac{\alpha_i m^2}{3}  \|t-t_i\|_2^2.
\end{equation}

\paragraph{Far region $\mathds{F}(\epsilon)$}

The relationship between $\epsilon,m$ and $d$ being established in (\ref{eq:nearup}), we are looking for a value of $\eta>0$ such that 
$$ t \in \bbF(\epsilon) \Rightarrow | p_m(t) | < 1-\eta.$$
The definition of $p_m$ and the Cauchy-Schwarz inequality yield
$$
|p_m(t) | \leq \sum_{k=1}^K   |\alpha_k| |\psi_m(t-t_k)| + \sum_{k=1}^K \|\beta_k\|_2 \|\nabla \psi_m(t-t_k)\|_2  .
$$
We  consider the second term of the right hand side with the help of Lemma \ref{lemma:tools} and Proposition~\ref{prop:sol_alpha_beta}:
\begin{eqnarray*}
\sum_{k=1}^K
\|\beta_k\|_2 \|\nabla \psi_m(t-t_k)\|_2 & \lesssim & \sum_{k=1}^K \|\beta_k\|_2  m \|\nabla \psi\|_{\infty} |\psi(m (t-t_k))^3| , \\
& \lesssim & \sqrt{K}  \times \frac{\sqrt{K}}{m\sqrt{d}} \frac{Kd^3}{m^4\Delta^4} \times m, \\
& \lesssim & \frac{K}{\sqrt{d}} \times \frac{Kd^3}{m^4\Delta^4}.
\end{eqnarray*}
%\textcolor{red}{Sebastien: ici j'ai fait la meilleure borne possible en imaginant que d peut Ãªtre grand mais pas K. Si c'est l'inverse alors il faut utiliser plutÃŽt un $\sqrt{d}$ en plus et un $\sqrt{K}$ en moins.}
%\textcolor{red}{
%$$
%\|\beta_k\|_2 \|\nabla \psi_m(t-t_k)\|_2 \lesssim \|\beta_k\|_2  m \|\nabla \psi\|_{\infty} |\psi(m (t-t_k))^3| \lesssim \|\beta\|_\infty \sqrt{d}  m \|\nabla \psi\|_{\infty} |\psi(m (t-t_k))^3| \lesssim \frac{K d^{7/2}}{m^3 \Delta^4} m \lesssim \frac{K d^{7/2}}{m^2 \Delta^4}.
%$$
%}

In particular, there exists a constant $\check C$ such that
$$ |p_m(t)|  \leq \sum_{k=1}^K |\alpha_k| \left| \psi_m(t-t_k)\right| + \check C \frac{K}{\sqrt{d}} \frac{K d^3}{m^4 \Delta^4}. $$
%Then, remark that for all $k\in \lbrace 1,\dots, s \rbrace$,
%$$ |\psi_m(t-t_k)| = \prod_{\ell=1}^d \sinc^4 (m (t^\ell-t_k^{\ell})) \leq \frac{1}{ m^4 \sup_{1 \leq \ell \leq d} | t^\ell-t^{\ell}_k|^4},$$ 
Let $t \in \bbF(\epsilon)$ and $t_{i^\star}$ the closest point of $t$ in the set $\lbrace t_1,\dots, t_K \rbrace$, the triangle inequality shows that $ \forall k \not = i^\star$, we have $\| t-t_k \|_2 > \frac{\Delta}{2}.$
%In particular, a $\ell_0$ exists such that 
%$$ | t^{\ell_0}-t^{\ell_0}_k | > \frac{\Delta}{2 d} \quad \forall k \not = i^\star.$$
Hence, since $\|\alpha\|_{\infty}$ is upper bounded by a universal constant (see Proposition \ref{prop:sol_alpha_beta}), we deduce from $i)$
 Lemma \ref{lemma:borne_derivees} that 
$$ \sum_{k\not = i^\star} |\alpha_k| |\psi_m(t-t_k)| \lesssim \frac{  K d^2}{m^4 \Delta^4}.$$

In the same time, the last term that involves $i^\star$ is upper bounded by
$$ |\alpha_{i^\star}| | \psi_m(t-t_{i^\star}) | \leq \|\alpha \|_\infty \max_{\|x\|_2 > \frac{\upsilon}{md}} |\psi_m(x)| \leq \left(1+C_0\frac{K d^3}{m^4 \Delta^4 }\right) \max_{\|y\|_2 >\upsilon d^{-1}} \psi^4(y),$$
where $C_0$ is a large enough universal constant.
Using that 
$$
|g(x)| = \frac{|\sin(x)|}{|x|} \leq (1-x^2/12) \mathds{1}_{|x| \leq 2} + \frac{1}{2} \mathds{1}_{|x| \geq 2},
$$
and the fact that when $\|y\|_2 \geq \upsilon d^{-1}$, then the absolute value of one of the coordinate of $y$ is greater than $\upsilon d^{-3/2}$, 
we deduce that 
$$
 |\alpha_{i^\star}|  |\psi_m(t-t_{i^\star}) | \leq \left(1+C_0\frac{K d^3}{m^4 \Delta^4 } \right) \left[\left(1-\frac{\upsilon^2}{12 d^3} \right) \vee \frac{1}{2} \right]^4 \leq  \left(1+C_0\frac{K d^3}{m^4 \Delta^4 }\right) (1-\eta)^4,
$$
where $\eta \asymp \upsilon^2 d^{-3}$.
This entails the desired result  as soon as $m$ is chosen such that 
\begin{equation}
m \geq \frac{ K^{1/2} d^{3/2}}{\mathcal{C}\Delta},
\label{eq:condm1}
\end{equation}
for some positive constant $\mathcal{C}$ small enough. It is easy to check that in this case, a small enough   $\upsilon$ exists (independent of $d$, $K$, $m$ and $\Delta$) such that:
\begin{equation}\label{eq:farup}
 m \gtrsim \frac{ K^{1/2} d^{3/2}}{\Delta} \quad \text{and} \quad t \in \bbF\left(\frac{\upsilon}{m d}\right)\Longrightarrow 
|p_m(t)| \leq 1- \frac{\upsilon^2}{  d^3}.
\end{equation}

\paragraph{Conclusion of the interpolation}

To accomodate with conditions~\eqref{eq:nearup} and~\eqref{eq:farup}, we consider an integer $m$ such that $m \gtrsim K^{1/2} d^{3/2} \Delta^{-1}$ and $\epsilon = \upsilon m^{-1} d^{-1}$. We deduce that $p_m$ satisfies in the far region $\bbF(\epsilon)$:
$$
\forall t \in \bbF(\epsilon)\qquad 
-\left(1-\frac{\upsilon^2}{ 2 d^3}\right) \leq p_m(t) \leq \left(1-\frac{\upsilon^2}{ 2 d^3}\right),
$$
while in the near region we have:
$$
\forall i \in \{1, \ldots, K\} \quad 
\forall t \in \bbN_i(\epsilon) \qquad 
0 \leq p_m(t) \leq 1 - \mathcal{C} m^2 \|t-t_{i}\|^2.
$$
We then set $\Pm = p_m^2$. This function satisfies both the constraints and the interpolation conditions in the statement of Theorem~7. %of \cite{super_mix}.
We then obtain $i)$ and $ii)$. \\

\noindent
\underline{\textbf{Proof of $iii)$:}}

Remark first that $p_m$ is a linear combination of shifted sinus cardinal functions and derivatives of sinus cardinal functions up to the power $4$ used in $\psi$. Moreover, it is straightforward to check that
$$
\cF[\psi^4] = \cF[\psi] \star \cF[\psi]  \star \cF[\psi]  \star \cF[\psi].
$$
Therefore, the Fourier transform of $\psi^4$ has a compact support of size $[-2,2]^d$ since the Fourier transform of the sinus cardinal is the rectangular indicator function of $[-1/2,1/2]$. Using the effect on the Fourier transform of scaling and shifting a function  we deduce that the Fourier transform of $p_m$ has a compact support, which size varies linearly with $m$:
$$\mathrm{Supp}(\cF[p_m]) \subset [-2m,2m]^d.$$
Since $\Pm = p_m^2$, we have $\cF[\Pm]= \cF[p_m] \star \cF[p_m]$ so that 
$$\mathrm{Supp}(\cF[\Pm]) \subset [-4m,4m]^d.
$$

\noindent

We now compute an upper bound of  $\|\Pm\|_2$: the isometry property entails the several inequalities:
$$
\| \Pm\|_2 = \|\cF[\Pm]\|_2  = \|\cF[p_m] \star \cF[p_m]\|_2  \leq \|\cF[p_m]\|_2 \|\cF[p_m]\|_1,
$$
where we used the standard inequality $\|g \star h\|_2 \leq \|g\|_2 \|h\|_1$.

Now, the triangle inequality yields
\begin{align*}
\|\cF[p_m]\|_2& = \left\| \sum_{k=1}^K \alpha_k \cF[\psi_m(.-t_k)] + \cF[\langle \beta_k,\nabla \psi_m(.-t_k) \rangle] \right\|_2,\\
& \leq \sum_{k=1}^K |\alpha_k| \, \|\cF[\psi_m(.-t_k)] \|_2 +\|\cF[\langle \beta_k,\nabla \psi_m(.-t_k) \rangle] \|_2,\\
& \leq K \sup_{1 \leq k \leq K} \left( |\alpha_k| \|\cF[\psi_m(.-t_k)] \|_2 + \|\beta_k\|_2 \|\cF[\nabla \psi_m(.-t_k) ]\|_2\right), \\
&  \leq K \left( \|\alpha\|_{\infty} \|\cF[\psi_m] \|_2 + \sup_{1 \leq k \leq K} \|\beta_k\|_2 \left\| \sqrt{\sum_{i=1}^d \cF[\partial_i \psi_m(.-t_k)]^2 }\right\|_2\right),
\end{align*}
where the last line comes from the Cauchy-Schwarz inequality.

We then deduce that
$$
\|\cF[p_m]\|_2 \leq K \left( \|\alpha\|_{\infty} \|\cF[\psi_m]\|_2 +  \sup_{1 \leq k \leq K} \|\beta_k\|_2 \| |\cF[\nabla \psi_m]|_2 \|_2 \right),
$$ 
where $|\cF[\nabla \psi_m]|_2$ refers to the Euclidean norm of the $d$-dimensional vector $\cF[\nabla \psi_m]$.
 Now, remark that a dilatation by a ratio $m$ yields on $L^2$ norms:
$$
\|\cF[\psi_m]\|_2 \lesssim m^{-d/2} \qquad \text{and} \qquad 
\| |\cF[\nabla \psi_m] |_2 \|_2 \lesssim d m^{-d/2}.
$$
Hence 
$$ \| \mathcal{F}[p_m] \|_2 \lesssim K m^{-d/2}  \left( \|\alpha\|_\infty + d\ \sup_{1 \leq k \leq K} \|\beta_k\|_2   \right).$$

We use a similar argument and obtain that
$$
\|\cF[p_m]\|_1 \leq K \left( \|\alpha\|_{\infty} \|\cF[\psi_m] \|_1 + \sup_{1 \leq k \leq K} \|\beta_k\|_2 \left\| \sqrt{\sum_{i=1}^d \cF[\partial_i \psi_m(.-t_k)]^2 }\right\|_1\right)
$$
In the meantime, the effect of this dilatation on the $L^1$ norms is managed by:
$$
\|\cF[\psi_m]\|_1 = \int | \cF[\psi_m](\xi)| \dd \xi \leq m^{-d} \|\cF[\psi]\|_{\infty} |\mathrm{Supp}(\cF[\psi_m])| \lesssim m^{-d} \|\cF[\psi]\|_{\infty} m^d \lesssim 1,
$$
and with a same argument we obtain that:
$
\| |\cF[\nabla \psi_m]|_2 \|_1 \lesssim d.
$
Hence  
$$ \|\cF[p_m]\|_1 \lesssim K \left( \|\alpha\|_\infty + d\ \sup_{1 \leq k \leq K} \|\beta_k\|_2   \right).$$
%Using our choice of $m$, 
We then obtain that
$$
\|\Pm\|_2 \lesssim K^{2} m^{-d/2}\left( 1 + \frac{\sqrt{Kd}}{m} \times \frac{Kd^3}{m^4\Delta^4}  \right) \lesssim K^2 m^{-d/2},
$$  
provided 
\begin{equation}
m > \frac{1}{\mathcal{C}} \left( \frac{K^{1/4} d^{3/4}}{\Delta} \vee \sqrt{Kd}\right)
\label{eq:condm2}
\end{equation}
for some constant $\mathcal{C}$ small enough. 
\noindent
%\textcolor{red}{Sebastien: calculs detailles. Attention, il n'y a pas de d en facteur car la norme 2 des beta est toute petite.}

\noindent
\underline{\textbf{Proof of $iv)$:}}
The last point is a simple consequence of the convolution kernel induced by $\Phi$. Since $\varphi$ satisfies $(\mathcal{H}_{4m})$, then $\forall \xi \in [-4m,4m]^d$, we have
$\sigma(\xi) \neq 0$. Hence, we can define $\com$ through its Fourier transform:
\begin{equation}\label{eq:c0mdef}
\forall \xi \in \bR^d \qquad 
\cF[\com](\xi) = \frac{\cF[\Pm](\xi)}{\sigma(\xi)} \mathds{1}_{ \xi \in \mathrm{Supp}( \cF[\Pm])}.
\end{equation}
Moreover, the Fourier transform of $\com$ is naturally compact, which entails that 
$\com \in \bbL$.\\

\noindent
{\it Conclusion:} The constraints (\ref{eq:condm1}) and (\ref{eq:condm2}) together with $\epsilon \sim \frac{1}{md}$ and $\epsilon \leq \Delta/2$ leads to the condition 
$$ m \gtrsim  \frac{K^{1/2} d^{3/2}}{\Delta} \vee \sqrt{Kd} \vee \frac{1}{d\Delta}.$$
Provided $\Delta$ is small or bounded, this condition reduces to $m \gtrsim \frac{K^{1/2}d^{3/2}}{\Delta}$ and $m\gtrsim \sqrt{Kd}$ when $\Delta$ is large. 
\end{proof}

Some useful properties of the sinus cardinal function are detailed in the following basic lemma.
\begin{lemma}\label{lemma:calculpsi}
If $g(x) = \sinc(x)$, then for any $x \in \bR$:
\begin{itemize}
\item[$i)$] $$
g'(x) = \frac{x \cos x - \sin x}{x^2} \quad \text{and} \quad \|g'\|_{\infty} \leq \frac{1}{2}.
$$
\item[$ii)$] $$
g''(x) = -  \frac{(x^2-2) \sin x + 2 x \cos x}{x^3} \quad \text{and} \quad \|g''\|_{\infty} \leq \frac{1}{2}.
$$
\item[$iii)$] $$
g^{(3)}(x) = \frac{3 (x^2-2) \sin x -  x (x^2-6)\cos x}{x^4} \quad \text{and} \quad \|g^{(3)}\|_{\infty} \leq \frac{1}{2}.
$$
\item[$iii)$] $$
g^{(4)}(x) = \frac{4 x (x^2-6) \cos x +   (x^4-12 x^2+24 )\sin x}{x^5} \quad \text{and} \quad \|g^{(4)}\|_\infty \leq \frac{1}{2}.
$$
\end{itemize}
\end{lemma}

Some additional ingredients on $\psi_m$ are detailed below where $\psi_m$ is defined by:
\begin{equation}
\psi_m(.) = \psi^4(m .) \text{ with } 
\forall u=(u^1,\ldots,u^d) \in \bR^d \quad \psi(u) = \prod_{j=1}^d \mathrm{sinc}(u^j) 
\text{ and }  \mathrm{sinc}(x) =\frac{\sin(x)}{x}.
\label{eq:psim_init}
\end{equation} 
 In the sequel, we will use the shortcut of notation $\partial_u$ instead of $\partial^{|u|}_u \psi$ for any multi-index $u$.
\begin{lemma}
\label{lemma:tools}
Let $\psi_m$ be the function defined in (\ref{eq:psim_init}). Then
\begin{itemize}
\item $i)$ $\psi_m(0) =g^4(0)^d = 1$. 
\item $ii)$ $\nabla \psi_m(0) =0$ and $$\nabla \psi_m (x) = 4m \psi^3(m x) \nabla \psi(m x) .$$
\item $iii)$ $D^2 \psi_m (0) = - \frac{4}{3} m^2 I_d$ and $$(D^2 \psi_m (x))_{i,j} = 4m^2[ \psi^3 \partial^2_{i,j}  + 3  \psi^2 \partial_i \partial_j ](m x).$$
\item $iv)$ $(D^3 \psi_m)(0) =0$ and 
$$
(D^3 \psi_m (x))_{i,j,k} = 4 m^3 [\psi^3  \partial^3_{i,j,k}+ 6 \psi
 \partial_{i}  \partial_{j}   \partial_{k} + 3 \psi^2[\partial^2_{i,j}  \partial_k  + \partial^2_{i,k}  \partial_j  +\partial^2_{j,k}  \partial_i ]  ](mx)
$$
%\begin{eqnarray*}
%\lefteqn{ (D^3 \psi_m (x))_{i,j,k}=4 m^3 [\psi^3  \partial^3_{i,j,k}+ 8 \psi
% \partial_{i}  \partial_{j}   \partial_{k}  }  \\
%&  &  + 3 \psi^2[\partial^2_{i,j}\psi  \partial_k  + \partial^2_{i,k}\psi  \partial_j \psi  +\partial^2_{j,k}\psi  \partial_i \psi ]  ](mx)
%\end{eqnarray*}

\item $v)$ Finally
\begin{eqnarray*}
(D^4 \psi_m)(x)_{i,j,k,l} &=& 4 m^4 [\psi^3 \partial_{i,j,k,l}^4 + 3 \psi^2 \square_{i,j,k,l}  + 6 \psi \tilde\square_{i,j,k,l} + 6 \check\square_{i,j,k,l}](m x),
\end{eqnarray*}
with 
$$
\square_{i,j,k,l}=  \partial_i  \partial^3_{j,k,l} + \partial_j  \partial^3_{i,k,l} + \partial_k  \partial^3_{i,j,l} +\partial_l  \partial^3_{i,j,k}   + 
 \partial^2_{i,j}  \partial^2_{k,l} + \partial^2_{i,k}   \partial^2_{j,l} + \partial^2_{i,l}   \partial^2_{j,k},
 $$
 $$
 \tilde\square_{i,j,k,l} = \partial_{i,j}^2\partial_k \partial_l + 
 \partial_{i,k}^2\partial_j \partial_l+\partial_{i,l}^2\partial_k \partial_j
 + \partial_{j,k}^2\partial_i \partial_l+ \partial_{j,l}^2\partial_i \partial_k
+ \partial_{k,l}^2\partial_i \partial_j $$
and 
$$
\check\square_{i,j,k,l} = \partial_i  \partial_j   \partial_k   \partial_l 
$$
\end{itemize}
\end{lemma}

\vspace{1em}
\noindent

Several bounds on the successive derivatives of $\psi_m$ are given in the following lemma.
\begin{lemma} \label{lemma:borne_derivees} 
For any pair $(i,j)$ such that $i \neq j$:
\begin{itemize}
\item $i)$ $ |\psi_m(t_i-t_j)| \lesssim \frac{d^2}{ m^4 \Delta^4}.$
\item $ii)$ $ |\partial_{u}\psi_m(t_i-t_j)|  \lesssim \frac{d^2}{m^3 \Delta^4}$
\item $iii)$ $ |\partial^2_{u,v}\psi_m(t_i-t_j)| \lesssim  \frac{d^2}{m^2 \Delta^4}.$
\item $iv)$ $ |\partial^3_{u,v,w}\psi_m(t_i-t_j)| \lesssim \frac{d^2}{m \Delta^4}.$
\end{itemize}
\end{lemma}

\begin{proof}
In what follows, we deliberately choose to omit the multiplicative constants since the rest of the paragraph will be managed in the same way.

\noindent
Point $i)$: we use $|\sinc(x)| \leq |x|^{-1}$ and remark that
$\|t_i-t_j\|_2 \geq \Delta$ so that 
$$
\sum_{\ell=1}^d (t_i^{\ell}-t_j^{\ell})^2 \geq \Delta^2.
$$
We then deduce that
$$\psi_m(t_i-t_j) = \prod_{\ell=1}^d \sinc(m (t^{\ell}_i-t^{\ell}_j))^4
\leq \frac{1}{m^4 (\Delta^2/d)^2}$$
because one coordinate $\ell_0$ exists such that
$|t_i^{\ell_0}-t_j^{\ell_0}|^2 \geq \Delta^2 d^{-1}$.

\noindent
Point $ii)$: we use Lemma \ref{lemma:calculpsi}, Lemma \ref{lemma:tools}
and 
$$
\partial_u \psi(t) = g'(t^u) \prod_{\ell \neq u} g(t^\ell) ,
$$
associated with $|g(x)| \vee |g'(x)| \lesssim \frac{1}{|x|}$. It yields
$$
|\partial_u \psi_m(t_i-t_j) |\lesssim m \frac{d^{1/2}}{m \Delta} \left( \frac{d^{1/2}}{m \Delta} \right)^{3} \lesssim \frac{d^2}{m^3   \Delta^4}.
$$
We then obtain $ii)$.

\noindent
Point $iii)$: we still use Lemma \ref{lemma:calculpsi} and Lemma \ref{lemma:tools}, the fact that
$$
\partial^2_{u,v} \psi (t)= \mathds{1}_{u \neq v} g'(t^u)g'(t^v) \prod_{\ell \neq u, \ell \neq v} g(t^\ell) + 
\mathds{1}_{u = v} g''(t^u)  \prod_{\ell \neq u} g(t^\ell) 
$$
and $|g(x)| \vee |g'(x)| \vee |g''(x)| \lesssim \frac{1}{|x|}$. It leads to
$$
| \partial^2_{u,v}\psi_m(t_i-t_j)| \lesssim m^2 \left[\frac{d^{1/2}}{m \Delta} \frac{d^{3/2}}{(m\Delta)^3}+\frac{d}{(m \Delta)^2} \frac{d}{(m\Delta)^2} \right] \lesssim \frac{d^2}{m^2 \Delta^4}.
$$

\noindent
Point $iv)$: the proof follows the same lines with the help of the previous lemmas, we check that
\begin{eqnarray*}
| \partial^3_{u,v,w}\psi_m(t_i-t_j)|
& \lesssim &  m^3 \left[ \frac{d^{1/2}}{m \Delta} \frac{d^{3/2}}{(m \Delta)^3} + \frac{d^{1/2}}{m \Delta} \frac{d^{1/2}}{m \Delta}  \frac{d}{(m \Delta)^2}+\frac{d^{3/2}}{(m \Delta)^3} \frac{d^{1/2}}{m \Delta} \right]
 \lesssim    \frac{d^2}{m \Delta^4}.
\end{eqnarray*}
\end{proof}

\end{document}